\documentclass[final]{siamltex}

\usepackage[sort]{cite}
\usepackage[dvipsnames,usenames]{color}
\usepackage[dvips]{graphicx}
\usepackage{psfrag}
\usepackage{epsfig}
\usepackage{amsmath,bbm,times,stmaryrd} 
 \usepackage[mathscr]{eucal}
 
 \usepackage{amssymb}
\usepackage{tabularx}
\usepackage{multirow}
\usepackage{enumerate}

\usepackage[hang,raggedright]{subfigure}

\usepackage{colortbl}
\usepackage{dcolumn}
\newcolumntype{.}{D{.}{.}{1.3}}

\usepackage{txfonts}
\usepackage{url}
\usepackage[subnum]{cases}

\usepackage{eqnarray}
\usepackage{equationarray}

\usepackage{tikz}
\usetikzlibrary{positioning}

\newlength{\XWidth}

\def\diag{\operatorname{diag}}

\newcommand{\ba}{{\bvec{a}}}

\newcommand{\bg}{\bvec{g}}

\newcommand{\bu}{\bvec{u}}

\newcommand{\bw}{\bvec{w}}

\newcommand{\by}{\bvec{y}}

\newcommand{\bA}{{\bf A}}
\newcommand{\bB}{{\bf B}}
\newcommand{\bC}{{\bf C}}
\newcommand{\bD}{{\bf D}}
\newcommand{\bE}{{\bf E}}
\newcommand{\bF}{{\bf F}}
\newcommand{\bG}{{\bf G}}
\newcommand{\bH}{{\bf H}}
\newcommand{\bI}{{\bf I}}
\newcommand{\bJ}{{\bf J}}
\newcommand{\bK}{{\bf K}}
\newcommand{\bL}{{\bf L}}

\newcommand{\bP}{{\bf P}}
\newcommand{\bQ}{{\bf Q}}

\newcommand{\bS}{{\bf S}}
\newcommand{\bSS}{{\bf S}}

\newcommand{\bU}{{\bf U}}
\newcommand{\bV}{{\bf V}}

\newcommand{\bX}{{\bf X}}
\newcommand{\bY}{{\bf Y}}
\newcommand{\bZ}{{\bf Z}}

\newcommand{\bsI}{{\boldsymbol I}}

\newcommand{\calO}{{\mathcal{O}}}

\newcommand{\bGamma}{\mbox{\boldmath $\Gamma$}}

\newcommand{\bSigma}{\mbox{\boldmath $\Sigma$}}

\newcommand{\bPhi}{\mbox{\boldmath $\Phi$}}
\newcommand{\bPsi}{\mbox{\boldmath $\Psi$}}

\newcommand{\Real}{\mathbb R}
\newcommand{\Complex}{\mathbb C}

\newcommand{\1}{\mbox{\boldmath $1$}}
\newcommand{\0}{\mbox{\boldmath $0$}}
\newcommand{\be}{\begin{eqnarray}}
\newcommand{\ee}{\end{eqnarray}}

\newcommand{\matrixb}{\left[ \begin{array}}
\newcommand{\matrixe}{\end{array} \right]}

\newcommand{\tr}{\mathop{\rm tr}\nolimits}

\def\*{\circledast}

\newcommand{\bvec}[1]{\boldsymbol{#1}}

\newcommand{\ve}{\bvec{e}}
\newcommand{\vf}{\bvec{f}}

\def\vectorize{\operatorname{vec}}
\newcommand{\vtr}[1]{\vectorize\hspace{-.3ex}\left(#1\right)}

\newcommand{\tensor}[1]{\boldsymbol{\mathscr{\MakeUppercase{#1}}}} %tensor
\newcommand{\tA}{\tensor{A}}

\newcommand{\tE}{\tensor{E}}
\newcommand{\tF}{\tensor{F}}

\newcommand{\tN}{\tensor{N}}

\newcommand{\tY}{\tensor{Y}}

\newtheorem{notation}{Notation}[section]

\usepackage[vlined,ruled,commentsnumbered]{algorithm2e}

\usepackage{aurical}

\newcommand{\tauri}[1]{{\selectfont\Fontauri{\Large{\text#1}}}}
\newcommand{\sta}{\tauri{a}}

\newcommand{\sty}{\tauri{y}}

\def\blkdiag{\mathop{\mbox{{\tt{blkdiag}}}}}
\def\bigcircledast{\mathop{\mbox{\fontsize{18}{19}\selectfont $\circledast$}}}

\renewcommand{\bigodot}{\mathop{\mbox{\fontsize{18}{19}\selectfont$\odot$}}}

\usepackage{arydshln}

\usepackage{soul}
  \usepackage[usenames,dvipsnames]{color}
  \usepackage{color}

\definecolor{lightblue}{rgb}{.5,.85,1}
\definecolor{lightred}{rgb}{1,.6,.5}
\definecolor{lightorange}{rgb}{1,.7,.4}

\providecommand{\add}[1]{#1}% FINISHED
\providecommand{\rem}[1]{}% FINISHED

\title{Low Complexity Damped Gauss-Newton Algorithms for CANDECOMP/PARAFAC}

\author{
    Anh Huy Phan
    \thanks{Brain Science Institute, RIKEN, Wakoshi, Japan
         ({\tt phan@brain.riken.jp}).}
    \and Petr Tichavsk{\'y}
    \thanks{Institute of Information Theory and Automation, Prague, Czech
    ({\tt tichavsk@utia.cas.cz}).
    {The work of P. Tichavsk{\'y} was supported by Ministry of Education,
Youth and Sports of the Czech Republic through the project 1M0572
and by Grant Agency of the Czech Republic through the project
102/09/1278.}
    }
        \and Andrzej Cichocki
            \thanks{Brain Science Institute, RIKEN, Wakoshi, Japan
         ({\tt cia@brain.riken.jp}).}
}

\begin{document}

\maketitle

\begin{abstract}
The damped Gauss-Newton (dGN) algorithm for CANDECOMP/PARAFAC (CP) decomposition \add{can handle the challenges of collinearity of factors and different magnitudes of factors; nevertheless, for factorization of an $N$-D tensor of size $I_1\times\ldots\times I_N$ with rank $R$, the algorithm is computationally demanding due to construction of large approximate Hessian of size $(RT \times RT)$ and its inversion where $T = \sum_n I_n$}. In this paper, we propose a fast implementation of the dGN algorithm which is based on novel
expressions of the inverse approximate Hessian in block form. \add{The new implementation has lower computational complexity, besides computation of the gradient (this part is common to both methods), requiring the inversion of a matrix of size $NR^2\times NR^2$}, which is much smaller than the whole approximate Hessian, if $T \gg NR$. \add{In addition, the implementation has lower memory requirements, because neither the Hessian nor its inverse
never need to be stored in their entirety}. A variant of the algorithm working with complex valued data is
proposed as well. Complexity and performance of the proposed algorithm is compared with those of dGN and ALS with line search on \add{examples of difficult benchmark tensors}.
\end{abstract}



\pagestyle{myheadings}
\thispagestyle{plain}
\markboth{PHAN AND TICHAVSK{\'Y} AND CICHOCKI}{Low Complexity Damped Gauss-Newton Algorithms for CANDECOMP/PARAFAC}

\section{Introduction}
\label{sec:introduction}

Algorithms for canonical polyadic decomposition, also coined CANDECOMP/PARAFAC (CP), can work well for general data \cite{Bro1997,Li_Sidiropoulos,Kolda08}. However, \add{they often fail for data with factors of different magnitudes} \cite{Paatero97} \add{or collinear factors  such as bottlenecks and swamps}. Bottlenecks arise when two or more components are collinear \cite{Comon09,Guo2011}, and \add{swamps arise when collinearity exists in all modes} \cite{MITCHELL94,Comon09}.
 %\cite{Rayens97,MITCHELL94,bro_book,Comon09}.
Alternating least squares (ALS) algorithms with line searches, regularization, and rotation can improve performance, but they do not completely solve the problems. \add{The damped Gauss-Newton (dGN) or Levenberg-Marquardt (LM) algorithm has been confirmed to successfully decompose such difficult data} \cite{Hayashi,Paatero,Paatero97,Paatero99,TomasiBro05,Tomassithesis}.
However, because \add{these methods require the inverse of a large-scale approximate Hessian matrix}, the dGN algorithm is not applicable to real-world large-scale and high-dimensional data.
In this paper, \add{we establish a fast inverse of the approximate Hessian for low-rank tensor factorization} by proving that the approximate Hessian for low-rank tensor factorization is a low-rank adjustment to a block diagonal matrix, and propose fast dGN algorithms \add{that do not need to store the approximate Hessian and its inverse entirely at one time}.

The paper is organized as follows. Notation and basic multilinear algebra are briefly reviewed in  Section~\ref{sec:notation}.
CP model and common algorithms are shortly reviewed in Section~\ref{sec::CP}.
Section \ref{sec::GNalgorithm} derives the fast dGN algorithm.
Low-rank adjustment of approximate Hessian is derived, and its fast inverse is deduced in this section.
The fast dGN algorithm with two \add{variants} has been proposed in Section~\ref{sec:fastdGN}.
The fast dGN is extended to complex-valued tensor factorization in Section \ref{sec::complexalgorithm}.
In Section \ref{sec::experiment} we provide examples illustrating the validity and performance of the proposed algorithms.
Finally, Section \ref{sec:conclusion} concludes the paper.

\section{Tensor notation and CANDECOMP/PARAFAC (CP) model}
 \label{sec:notation}

We shall denote a tensor by bold calligraphic letters, e.g.,
$\tA \in \Real^{I_1 \times I_2 \times \cdots \times I_N}$,
matrices by bold capital letters, e.g., $\bA$ =$[\ba_1,\ba_2, \ldots, \ba_R] \in \Real^{I \times R}$, and vectors by bold italic letters, e.g., $\ba_j$ or $\bsI = [I_1, I_2,\ldots, I_N]$.
Mode-$n$ tensor unfolding of ${\tY}$ is denoted by $\bY_{(n)}$.
Generally, we adopt notation used in \cite{NMF-book,Kolda08}.
The Kronecker, Khatri-Rao (column-wise Kronecker) and Hadamard products and \rem{element-wise division} are  denoted respectively by  $\otimes$, $\odot$, $\*$, \rem{$\oslash$} \cite{Kolda08,NMF-book}.
%A  product of two matrices may also be called the Khatri-Rao product \cite{Sidiropoulos04low-rankdecomposition}.

\begin{notation}\label{def_Hadamard_KhatriRao_Nmatrices}
Given $N$ matrices $\bA^{(n)} \in \Real^{I_n \times R}$, we consider the following products\\[-3.5ex]
\be
 \bigcircledast_{n=1}^{N} \bA^{(n)} &=& \bA^{(N)} \circledast \cdots \circledast \bA^{(n)} \circledast
\cdots \circledast \bA^{(1)}  , \qquad I_n = I,   \forall n, \notag \\
 \bigcircledast_{k\neq n} \bA^{(k)} &=& \bA^{(N)} \circledast \cdots \circledast \bA^{(n+1)} \circledast \bA^{(n-1)}  \circledast
\cdots \circledast \bA^{(1)} , \qquad I_n = I,  \forall n, \notag \\
% \bigodot_{n=1}^{N} \bA^{(n)} &=& \bA^{(N)} \odot \cdots \odot \bA^{(n)} \odot
%\cdots \odot \bA^{(1)}  , \qquad R_n  = R, \forall n, \notag \\
 \bigodot_{k\neq n} \bA^{(k)} &=& \bA^{(N)} \odot \cdots \odot \bA^{(n+1)} \odot \bA^{(n-1)} \,
\cdots \odot \bA^{(1)}   \notag . 
%\\
%\bigodot_{k\neq n} \bA^{(k)} &=& \bA^{(N)} \otimes \cdots \otimes \bA^{(n+1)} \otimes \bA^{(n-1)} \,
%\cdots \otimes \bA^{(1)} . \notag
 \ee
 \end{notation}\\[-5ex]
\begin{definition}\label{def_partition_block}{\bf(Partitioned matrix and block matrix)}
A partitioned matrix $\bU$ of $N$ matrices $\bU^{(n)}$  along the mode-2 (horizontal) is denoted by\\[-3.5ex]
\be
    \bU = \left[  \bU^{(1)} \; \cdots \; \bU^{(n)} \; \cdots \; \bU^{(N)} \right] =
    \left[  \bU^{(n)}  \right]_{n = 1}^{N}\, ,
\ee
and a partitioned matrix $\bV$ of $NM$ matrices $\bV^{(n,m)}$ along two modes is denoted by
$
    \bV = \left[  \bV^{(n,m)}  \right]_{n = 1,m=1}^{N,M}
$.
A block diagonal matrix $\bB$ of $N$ matrices $\bU^{(n)}$ is denoted by\\[-3ex]
\be
    \bB = \left[ \begin{array}{ccc}
    \bU^{(1)} \\[-1ex]
    & \ddots \\[-1ex]
    && \bU^{(N)}
    \end{array}
    \right] = {\blkdiag}\left(\bU^{(1)}, \cdots, \bU^{(N)} \right) = {\blkdiag}\left(\bU^{(n)} \right)_{n=1}^{N}\, . \label{equ_not_blkdiag}
\ee
\end{definition}\\[-4ex]

%\begin{notation}\label{def_ttm}{\bf(Tensor-matrix product)}
%The product of a tensor and a matrix along mode-$n$ is denoted as
%\be
%    {{\tY}} = {{\tG}} \times_n \, \bA\, .
%\ee

%Multiplication in all possible modes
%$(n=1,2,\ldots,N)$ of a tensor  ${\tG}$ and a set of matrices $\bA^{(n)}$
% is denoted as
% \be
%    {\tG} \times \{\bA\}
%        &=& {\tG} \times_1 \bA^{(1)} \times_2 \bA^{(2)} \cdots \times_N \bA^{(N)}.
%\ee
%\end{notation}

%\begin{definition}\label{def_vectensor}
%{\bf (vectorization)} Vectorization of an $N$-D tensor $\tA \in \Real^{I_1 \times I_2 \times \cdots \times I_N}$ is to map $\tA$ to be a column vector which is denoted by a bold script letter $\sta$  and recursively defined \cite{2011arXiv1101.2005R}
%\be
%	\sta = \vtr{\tA} = \left[ \begin{array}{c}
%				\vtr{\tA^{(1)}} \\
%				\vdots \\
%				\vtr{\tA^{(I_N)}} \\
%			\end{array}
%	\right]\, ,
%\ee
%where $\tA^{(i_N)}$ is an ($N-1$)-order subtensor: $\tA^{(i_N)}(i_1,i_2, \ldots, i_{N-1}) = \tA(i_1,i_2, \ldots, i_{N-1}, i_N)$.
%An entry $a_{\bi} = \tA(i_1,i_2, \ldots, i_N)$ will be an entry $\sta(\ivec{\bi}{\bsI})$
%\be
%	\ivec{\bi}{\boldsymbol I} = i_1 + (i_2 - 1)I_1 + (i_3 - 1)I_1I_2 +  \cdots+ (i_N - 1) I_1  \cdots  I_{N-1}.
%\ee
%\end{definition}
\begin{definition}\label{def_CP}{\bf(CANDECOMP/PARAFAC (CP))}
% \com{The imperative form reads very odd as a definition}
\add{A CPD consists in representing a given $N$-th order data tensor ${\tY} \in \Real^{I_1 \times I_2 \times \cdots \times I_N}$ by a set of  $N$ matrices (factors): $\bA^{(n)}$ = $[\ba^{(n)}_1, \ba^{(n)}_2,\ldots,\ba^{(n)}_R]$ $ \in \Real ^{I_n \times R}$, $(n=1,2, \ldots, N)$}\cite{Hitchcock1927,Harshman,Carroll_Chang} such that\\[-4ex]
\be
{\tY}  &\approx&  \sum\limits_{r = 1}^R {\ba^{(1)}_{r}  \circ \ba^{(2)}_{r} \circ  \ldots  \circ \ba^{(N)}_{r}}  = {{\tensor{\widehat Y}}}
%\notag \\
% &=& \llbracket \bA^{(1)},\bA^{(2)},\ldots,\bA^{(N)} \rrbracket
%  =   {\llbracket \{\bA \} \rrbracket }
  ,
\label{equ_CP_nolambda}
\ee\\[-3ex]
where symbol ``$\circ$"  denotes  outer product. Tensor ${{{\tensor{\widehat Y}}}}$ is an approximation of the data tensor ${\tY}$.
\end{definition}
 We often assume  unit-length components $\|\ba^{(n)}_{r}\|_{2} =1 $ for $ n=1, 2, \ldots ,N-1$,  $r = 1, 2, \ldots ,R$. 

\section{CP Algorithms}\label{sec::CP}
%\begin{figure}[t]
%\centering
%\psfrag{Ccong}[tc][tc]{\scalebox{.8}{\color{black}\large $\cong$}}
%\psfrag{(IK1xJ)}[tc][tc]{\scalebox{.8}{\color{black}\small $(I_{1} \times R)$}}
%\psfrag{(JxIK2)}[bc][bc]{\scalebox{.8}{\color{black}\small $(R \times I_{2})$}}
%\psfrag{(IK3xJ)}[bc][bc]{\scalebox{.8}{\color{black}\small $(I_{3} \times R)$}}
%\psfrag{I}[bc][bc]{{\color{black} \Large $\bI$}}
%\psfrag{(JxJxJ)}[bl][bl]{\resizebox{1.2cm}{!}{\color{black}\small  $(R \times R \times R)$}}
%\includegraphics[width=.6 \linewidth]{fig_CP.eps}
%\caption{Illustration of  a 3-D tensor factorization.The $I_1 \times I_2 \times I_3$ dimensional tensor $\bY$ is explained by the three factors $\bA^{(1)}, \bA^{(2)}, \bA^{(3)}$ along the three modes. Factors consist of the same number of components $R$.}\label{fig_CP3D}
%\end{figure}

%Most CP algorithms minimize the squared Euclidean distance (Frobenius norm)
%given by
%\be
%    D({\tY} ||{{\tensor{\widehat Y}}}) &=& \frac{1}{2}\|{\tY} -  {{\tensor{\widehat Y}}} \|_F^2
%    =   \frac{1}{2}\|\vtr{\tY} - \vtr{{\tensor{\widehat Y}}} \|_2^2 \, .
%    \label{equ_Frobcost}
%\ee
%where ${{\tensor{\widehat Y}}}$ is an approximate tensor of ${\tY}$.

%\subsection{ALS Algorithm}
The Alternating Least Squares (ALS) algorithm\cite{Carroll_Chang,Harshman,Bro1997,Nwaytoolbox,doi:10.1137/110843587} sequentially updates $\bA^{(n)}$ using the update rule given by\\[-3ex]
\be    \bA^{(n)} 
    %\leftarrow \bY_{(n)} \left\{ \bA \right\}^{\odot_{-n}}
     %   \left(\left\{ \bA \right\}^{\odot_{-n}T} \left\{ \bA \right\}^{\odot_{-n}} \right)^{-1}
        = \bY_{(n)} \left( \bigodot_{k \neq n}  \bA^{(k)} \right) 
        \left({\bGamma}^{(n)}\right)^{\dagger}\, , \quad (n=1,2,\ldots,N), \label{equ_ALS}
\ee
where \add{${\bGamma}^{(n)} = \bigcircledast_{k \neq n} \bC^{(k)}$}, 
$\bC^{(n)} = \bA^{(n) \, T} \, \bA^{(n)}$ ($n = 1, 2, \ldots, N$)
is defined as in Notation~\ref{def_Hadamard_KhatriRao_Nmatrices}, ``$\dagger$'' denotes the pseudo-inverse.

Denote by ${\sta} \in \Real^{R T}$, \add{$T = \sum_n I_n$}, concatenation of vectorizations of $\bA^{(n)}$, $n = 1, 2, \ldots, N$,\\[-3ex]
\be
    {\sta} &=& \left[\vtr{\bA^{(1)}}^T \; \cdots  \; \vtr{\bA^{(n)}}^T  \; \cdots \; \vtr{\bA^{(N)}}^T\right]^T  \label{equ_vecV} \,.
 \ee
All-at-once algorithms such as the OPT algorithm \cite{JChem-CPOPT}, the PMF3, damped Gauss-Newton (dGN) algorithms \cite{Hayashi,Paatero97,TomasiBro05,Tomassithesis} simultaneously update ${\sta}$.
The dGN algorithm is given by
\be
 {\sta} &\leftarrow& {\sta} + \left(\bH+ \mu \bI_{RT} \right)^{-1} \, \bg, \label{equ_GaussNewton_v2}\\
\bH &=&  \bJ^T \, \bJ  , \qquad  \bg  = \bJ^T \, \vtr{\tE}. 
\ee
where $\tE =  \tY - {\tensor{\widehat Y}}$, $\bJ \in \Real^{ {J} \, \times \, R T}$, \add{$({J} = \prod_n I_n)$} is the Jacobian of $\vtr{{\tensor{\widehat Y}}}$ with respect to $\sta$, \add{$\bH$ denotes the approximate Hessian}, and the damping parameter $\mu > 0$.
Paatero \cite{Paatero97} emphasized advantage of dGN compared with ALS when dealing with problems regarding swamps, different magnitudes of factors.
%
%We note that the approximate Hessian $\bH = \bJ^T \, \bJ$ is rank-deficient \cite{Paatero97,Paatero99,Tomassithesis,TomasiBro06}.
%%To deal with this problem and to improve convergence and stability of the algorithm, the Jacobian can be forced to be full rank by adding additional rows as proposed by Bini and Boito \cite{Bini2010}.
%To improve convergence and stability of the algorithm, the approximate Hessian is added by a damping parameter $\mu$.

The Gauss-Newton (GN) algorithm can be derived from Newton's method. Hence, the rate of convergence of the update rule (\ref{equ_GaussNewton_v2}) is at most quadratic.
However, these methods face problems involving the large-scale Jacobian and large-scale inverse of the approximate Hessian $\bH = \bJ^T \, \bJ \in \Real^{RT \times RT}$.
In order to eliminate the Jacobian, Paatero \cite{Paatero97} established explicit expressions for submatrices of $\bH$.
We note that inverse of $\bH$  is the largest workload of the GN algorithm with a complexity of order $\calO(R^3T^3)$ besides the computation of the gradient $\bg$.
Paatero \cite{Paatero97} solved the inverse problem $\bH^{-1}$ by Cholesky decomposition of the approximate Hessian and back substitution.
However, the algorithm is still computationally demanding.
Tomasi \cite{Tomassithesis} extended Paatero's results \cite{Paatero97},
and derived a convenient method to construct $\bH$ and the gradient for $N$-way tensor without using the Jacobian.
In order to cope with the inverse of $\bH$, Tomasi  \cite{Tomasi_tricap06} used QR decomposition.
However, \add{the efficiency of  existing dGN algorithms are still not sufficient for the large-scale problems due to the inverse $\bH^{-1}$}.

Recently, Tichavsk{\'y} and Koldovsk{\'y} \cite{Petr_10} have proposed a novel method to invert the approximate Hessian based on $3R^2 \times 3R^2$ dimensional matrices.
For low-rank approximation $R \ll I_n, \forall n$, this method dramatically improves the running time. However, the algorithm still demands significant temporary extra-storage, 
%and is of high computational cost due to employment of Kronecker products,
and it is restricted for third-order tensors.

%An other approach is to consider the CP decomposition as a joint diagonalization problem \cite{deLathauwer-JAD,RH:08,DBLP:journals/tsp/LathauwerC08,DBLP:conf/icassp/LucianiA11}. However, the method will not be discused in this paper.

\section{Fast damped Gauss-Newton algorithm} \label{sec::GNalgorithm}
In this section, we will derive a fast dGN algorithm for low-rank approximation of tensors with arbitrary dimensions.
The most important challenge of the update rule (\ref{equ_GaussNewton_v2}) is to reduce the computational cost for construction of the approximate Hessian $\bH$ and its inverse.
%The proposed algorithm solves the large scale inverse problem $\bH^{-1}$. 

%\hspace{0ex}\\
\begin{theorem}[Fast dGN algorithm]
\label{theo_fastdGN}
Define matrices ${\bGamma}^{(n,m)}$ of size $(R \times R)$, $n = 1, 2, \ldots, N$, $m = 1, 2, \ldots, N$, and a partitioned matrix $\bK$ of size ($NR^2 \times NR^2$) comprising matrices $\bK^{(n,m)}$
\be
	{\bGamma}^{(n,m)} &=& \left[{\bGamma}^{(n,m)}\right]^T = \left[{\bGamma}^{(m,n)}\right]^T =   \bigcircledast\limits_{k\neq n,m}{\bC^{(k)}} \, ,\qquad \bC^{(n)} = \bA^{(n) T} \bA^{(n)} \quad \in \Real^{R \times R},\label{equ_Gamma_a}\\
             \bK^{(n,m)} &=& (1-\delta_{n,m}) \, 
            \bP_{R} \, \diag\left(\vtr{{\bGamma}^{(n,m)}\right)} \, \quad \in \Real^{R^2 \times R^2}, \; n = 1,\ldots,N, m = 1, \ldots, N,  \label{equ_Knm_a}
          \ee
where $\delta_{n,m}$ is the Kronecker delta, \add{$\bP_{I,J}$ is a permutation matrix for any $I \times J$ matrix $\bX$ such that $\bP_{I,J} \vtr{\bX^T} = \vtr{\bX}$, $\bP_{R} \equiv \bP_{R,R}$ and $\bGamma^{(n)}  \equiv \bGamma^{(n,n)}$.}
%
%(see Appendix \ref{sec::app_Permute}).

For $NR \ll T$, the fast dGN algorithm is written for \add{each factor}  $\bA^{(n)}$ as follows
\be
    \boxed{\bA^{(n)}  \leftarrow  \bA^{(n)}_{\mbox{\add{\scriptsize$\mu$}}}  +  \bA^{(n)} \left( \bI_R -
         \left(  {\bF}_{n} + {\bGamma}^{\mbox{\add{\scriptsize$(n)$}}}  \right) \,
         \widetilde{\bGamma}_{\mu}^{\mbox{\add{\scriptsize$(n)$}}}  \right) \, , \quad n = 1, 2, \ldots, N,}\label{equ_fast_low_LM1a}
\ee
where  \add{$\bA^{(n)}_{\mu}$} is a \add{variant} of the ALS update rule (\ref{equ_ALS}) with a damping parameter $\mu >0$, $\bF_n$ of size ($R \times R$) are frontal slices of $\tF$ whose $\vtr{\tF}  = \bB_{\mu} \, \bw_{\mbox{\add{\scriptsize$\mu$}}}$, and 
\be
        \bA^{(n)}_{\mbox{\add{\scriptsize$\mu$}}}  &=&  {\bY_{(n)}} \, \left({\displaystyle\bigodot_{k\neq n} \bA^{(k)}}\right)\,  \widetilde{\bGamma}_{\mu}^{(n)}\\
	\widetilde{\bGamma}_{\mu}^{\mbox{\add{\scriptsize$(n)$}}} & =& \left( {\bGamma}^{\mbox{\add{\scriptsize$(n)$}}}  + \mu \bI_R \right)^{-1}, \\
	\bB_{\mu} &=& \begin{cases} 
 		 \left(\bK^{-1} + {\bPsi}_{\mu} \right)^{-1} \,      \, , \quad  &\text{for invertible $\bK$},  \label{equ_inverse_f2a} \\
		\bK \, \left( \bI_{NR^2} + {\bPsi}_{\mu} \bK \right)^{-1}     ,&\text{otherwise},   \label{equ_inverse_f1a}
		 \end{cases}\qquad 	\bB_{\mu} \in \Real^{NR^2 \times NR^2},
		 \\ %
        {\bPsi}_{\mbox{\add{\scriptsize$\mu$}}}  &=&  {\blkdiag}\left( \widetilde{\bGamma}_{\mu}^{\mbox{\add{\scriptsize$(n)$}}}\otimes
            \bC^{(n)}  \right)_{n=1}^{N}  \,\qquad \in \Real^{NR^2 \times NR^2} , \label{equ_ZiGZ_LM_a} \\
\bw_{\mbox{\add{\scriptsize$\mu$}}} &=&  \vtr{\left[ \bA^{(n) \, T}   \bA^{(n)}_{\mbox{\add{\scriptsize$\mu$}}} -  \,{\bGamma} \, \widetilde{\bGamma}_{\mu}^{\mbox{\add{\scriptsize$(n)$}}}  \right]_{n=1}^{N}} \quad \in \Real^{NR^2}
        \, , \quad {\bGamma} = \bigcircledast_{n=1}^{N}{\bC^{(n)}} \label{equ_proj_ZGJE_LMa} .
        \ee
 \end{theorem}

In order to prove Theorem~\ref{theo_fastdGN}, we derive a low rank adjustment for $
\bH$
 and employ the binomial inverse theorem \cite{0521386322} to invert a smaller matrix of size $NR^2 \times NR^2$ instead of $\bH^{-1}$.
%This result allows formulating a low complexity update rule equivalent to the dGN rule (\ref{equ_GaussNewton_v2}).

% \subsection{Fast inverse of the approximate Hessian}\label{sec:lowrankH}

\subsection{Fast inverse of the approximate Hessian $\bH$}\label{sec:fastinverseH}
 \hspace{1ex}\\[-2.5ex]
\begin{figure}[t]
\centering
%\includegraphics[width=.25\linewidth, trim = 1.2cm 1.0cm 1.2cm 1cm , clip = true]{fig_Hess5D}
%\hfill\includegraphics[width=.25\linewidth, trim = 1.2cm 1.0cm 1.2cm 1cm , clip = true]{fig_diag_hess5D}
%\hfill\includegraphics[width=.2\linewidth, trim = 2cm 1.0cm 2cm 1cm , clip = true]{fig_Z_Hess5D}
%\hfill\includegraphics[width=.25\linewidth, trim = 1.2cm 1.0cm 1.2cm 1cm , clip = true]{fig_V_hess5D}
\includegraphics[width=.95\linewidth, trim = 1.2cm 1.0cm 1.2cm 1cm , clip = false]{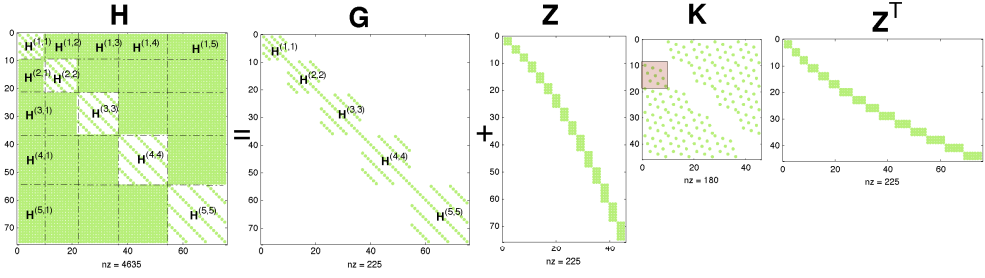}
\caption{Illustration of the approximate Hessian for a 5-D tensor which can be expressed as a low rank adjustment $\bH = \bG + \bZ \, \bK \, \bZ^T$ as in Theorem \ref{theo_decomH}. Green dots indicate nonzero elements.}
\label{fig_Hess}
\end{figure}
 \begin{theorem}[Low rank adjustment for the approximate Hessian $\bH$] \label{theo_decomH}
With $\bK$ defined in Theorem~\ref{theo_fastdGN}, the approximate Hessian $\bH$ can be decomposed into\\[-4ex]
\be
    \bH &=& \bG + \bZ \, \bK \, \bZ^T, \\
    	    \bG &=&  {\blkdiag}\left(  {\bGamma}^{\mbox{\add{\scriptsize$(n)$}}} \otimes \bI_{I_n}
            \right)_{n=1}^{N} \, \quad \in \Real^{RT\times RT} ,  \label{equ_G}\\
                \bZ &=& {\blkdiag}\left(\bI_R \otimes \bA^{(n)}
            \right)_{n=1}^{N}  \quad \in \Real^{RT \times NR^2}     \,  \label{equ_Z} .
\ee
%where $\bG$ is an invertible block diagonal matrix, $\bZ$ is a block diagonal matrix consisting of $N \, R^2$ columns, and kernel matrix $\bK$ is an $N \times N$ block of matrices $\bK^{(n,m)}$
%\be
%	    \bG &=&  {\blkdiag}\left(  {\boldsymbol\Gamma}^{(n)} \otimes \bI_{I_n}
%            \right)_{n=1}^{N} \, ,  \label{equ_G}\\
%                \bZ &=& {\blkdiag}\left(\bI_R \otimes \bA^{(n)}
%            \right)_{n=1}^{N}      \, , \label{equ_Z} 
%            \\
%             \bK^{(n,m)} &=& (1-\delta_{n,m}) \, 
%            \bP_{R} \, \diag\left(\vtr{{\bGamma}^{(n,m)}\right)} \, , \label{equ_Knm}
%              \bK^{(n,m)} &=& \begin{cases}
%            \bP_{R} \, \diag\left(\vtr{{\bGamma}^{(n,m)}\right)} \, ,\quad & m\neq n, \\
%            {\bf 0} \, ,\quad & m  =  n. \\
%            \end{cases}  \label{equ_Knm},
%\ee
%where $\delta_{n,m}$ is the Kronecker delta. The permutation matrix $\bP_{R}$ is defined for an $R \times R$ matrix (see Appendix \ref{sec::app_Permute}), and ${\bGamma}^{(n,m)}$ are symmetric matrices defined as
%\be
%    {\bGamma}^{(n,m)} = \left[{\bGamma}^{(n,m)}\right]^T = \left[{\bGamma}^{(m,n)}\right]^T =   \bigcircledast\limits_{k\neq n,m}{\bC^{(k)}} \, ,\qquad \bC^{(n)} = \bA^{(n) T} \bA^{(n)}. \label{equ_Gamma}
%\ee
\end{theorem}
%Proof of Theorem~\ref{theo_decomH} follows from Theorems presented in the subsequent sections.
Proof of Theorem~\ref{theo_decomH} is given in Appendix~\ref{sec::proof_theo_decomH}, whereas an example of $\bH$ for a 5-D tensor of size $3 \times 4\times 5 \times 6 \times 7$ composed by 5 factors each of which has 3
components is illustrated in Fig.~\ref{fig_Hess}. In the left hand side of Fig.~\ref{fig_Hess}, $\bH$ consists of $(N(N-1)) R^2$ rank-one matrices and $NR^2$ diagonal matrices which are located along its main diagonal.

 \begin{theorem}[Fast inverse of the damped approximate Hessian] \label{theo_inverseH}
Inverse of the damped approximate Hessian $\bH_{\mu} = \bH + \mu \, \bI_{RT}$ \add{can be computed through}
%inverse of an $NR^2 \times NR^2$ matrix
\be
\bH_{\mu}^{-1} 
        = {\widetilde{\bG}_{\mu}}  -{\bL}_{\mu}  \,   \bB_{\mu}  \,  {\bL}_{\mu}^T  \, ,\label{equ_inverse_H_LM}
\ee
where $\bB_{\mu}$ is an $NR^2 \times NR^2$ matrix defined in (\ref{equ_inverse_f2a})
%a matrix of size  $NR^2 \times NR^2$ given by
%\begin{numcases}{\bB_{\mu} = }
%\left(  \bK^{-1} +    {\bPsi} \right)^{-1} \, 
%\label{equ_inverse_B1}, \quad & if $\bK$ is invertible,\\ 
% \bK \,  \left( \bI_{NR^2}  +   {\bPsi}  \, \bK \right)^{-1}, & otherwise
%    ,\label{equ_inverse_B2}
%\end{numcases}
%\be
%\bB_{\mu} = \left(  \bK^{-1} +    {\bPsi} \right)^{-1} \, ,
%\label{equ_inverse_B1}\ee
%if $\bK$ is invertible, or in an alternative form
%\be
%\bB_{\mu} =   \bK \,  \left( \bI_{NR^2}  +   {\bPsi}  \, \bK \right)^{-1}
%    ,\label{equ_inverse_B2}\,
%\ee
and
\be
    {\widetilde{\bG}_{\mu}} &=&    {\blkdiag}
\left( \widetilde{\bGamma}_{\mu}^{\mbox{\add{\scriptsize$(n)$}}} \otimes \bI_{I_n} \right)_{n = 1}^N \, \quad \in \Real^{RT \times RT},\label{equ_invG}
%   \bG^{-1} =  {\blkdiag}\left(  {{\boldsymbol\Gamma}^{(n)}}^{-1} \otimes \bI_{I_n}
%           \right)_{n=1}^{N}   \,.\label{equ_invG}
 \\
     {\bL_{\mu}}   &=& {\blkdiag}\left(
	 \widetilde{\bGamma}_{\mu}^{\mbox{\add{\scriptsize$(n)$}}}  \otimes \bA^{(n)} \right)_{n = 1}^{N}  \, \quad \in \Real^{RT \times NR^2} . \label{equ_iGZ_LM}
%       {\bPsi}  &=&  {\blkdiag}\left( \widetilde{\bGamma}_{\mu}^{(n)}\otimes
%            \bC^{(n)}  
%            \right)_{n=1}^{N}  \, , \label{equ_ZiGZ_LM}
%            \\
%\widetilde{\bGamma}_{\mu}^{(n)}  &=&  \left({\bGamma^{\mbox{\add{\scriptsize$(n)$}}}}+ \mu \, \bI_R \right)^{-1}.
\ee
\end{theorem}
%
% obtained using properties of product of partitioned matrices

The matrix $\bK$ can also be expressed as a partitioned matrix of matrices $\bD^{(n,m)} = (1 -\delta_{n,m}) \diag\left(\vtr{{\bGamma}^{(n,m)}}\right) \in \Real^{R^2 \times R^2}$
\be
    \bK  = \left(\bI_N \otimes \bP_{R} \right) \, \left[ \bD^{(n,m)} \right]_{n,m} \, .\label{equ_kernel_K_D}
\ee
 If all the entries $\gamma^{(n,m)}_{r,s}$ of ${\bGamma}^{(n,m)}$ are non-zeros,
the matrix $\bD$ is invertible, and its inverse is also a partitioned matrix comprising diagonal matrices.
%Proof is briefly described in Appendix \ref{sec::app_Kernel}.
Inverse of $\bK$ is briefly described in Appendix \ref{sec::app_Kernel}.

An alternative expression $\bH_{\mu}^{-1}$  can be written in block form.
%From block form of $\bH$ in (\ref{equ_H_viaHnm}) and Theorem~\ref{theo_Hnm_full},
%the matrix $\bH_\mu$ can be written as an $N \times N$ block matrix of submatrices $\bH^{(n,m)}_\mu$ given by
%\be
%    \bH^{(n,m)}_\mu = \delta_{n,m}\left({\bGamma}^{(n)}_\mu \otimes \,
%    \bI_{I_n}\right) + \left(  \bI_R \otimes  \bA^{(n)} \right)\bSS^{(n,m)}
%    \left( \bI_R \otimes  \bA^{(m) \, T }  \right)\label{defHnmmu}
%\ee
%with $\bSS^{(n,m)}= (1-\delta_{n,m})  \bK^{(n,m)}= (1-\delta_{n,m}) \bP_{R} \,
%\diag\left(\vtr{{\bGamma}^{(n,m)}}\right)$.

\begin{theorem}[Fast inversion of $\bH_\mu$ in the block form]\label{theo_inverseH_block}
Inverse of $\bH_\mu$ can be written as 
\\[-3ex]
\be
   \bH_\mu^{-1} = \widetilde\bH_\mu = \left[
                    \begin{array}{ccccc}
              \widetilde\bH^{(1,1)}_\mu  &
              \cdots &
              \widetilde\bH^{(1,m)}_\mu  &
              \cdots &
              \widetilde\bH^{(1,N)}_\mu \\
              \vdots & \ddots &\vdots & \ddots &\vdots\\
              \widetilde\bH^{(n,1)}_\mu &
              \cdots &
              \widetilde\bH^{(n,m)}_\mu  &
              \cdots &  \widetilde\bH^{(n,N)}_\mu \\
              \vdots & \ddots &\vdots & \ddots &\vdots\\
              \widetilde\bH^{(N,1)}_\mu  &
              \cdots &
              \widetilde\bH^{(N,m)}_\mu  &
              \cdots &  \widetilde\bH^{(N,N)}_\mu \\
            \end{array}
           \right]              \, , \label{equ_HHnn2}
\ee\\[-3ex]
 where \\[-3ex]
 \be
    \widetilde\bH^{(n,m)}_\mu = \delta_{n,m}\left(\widetilde{\bGamma}^{(n)}_\mu \otimes \,
    \bI_{I_n}\right) - \left(  \bI_R \otimes  \bA^{(n)} \right)\widetilde\bSS^{(n,m)}_\mu
    \left( \bI_R \otimes  \bA^{(m) \, T }  \right) \, , \label{fastblock}
\ee
and
$\widetilde\bSS^{(n,m)}_\mu = \left(\widetilde{\bGamma}^{\mbox{\add{\scriptsize$(n)$}}}_\mu \otimes \,
    \bI_R\right) \bB_{\mu}^{(n,m)}
   \left(\widetilde{\bGamma}^{\mbox{\add{\scriptsize$(m)$}}}_\mu \otimes \,
    \bI_R\right)$  are matrices of size
$R^2\times R^2$.
\end{theorem}
\begin{proof}
From (\ref{equ_inverse_H_LM}), denote by $\bB_{\mu}^{(n,m)}$ the $(m,n)-$th block of
$\bB_{\mu}$, we have
 \be
    \widetilde\bH^{(n,m)}_\mu &=&
 \delta_{n,m}\left(\widetilde{\bGamma}^{\mbox{\add{\scriptsize$(n)$}}}_\mu \otimes \,
    \bI_{I_n}\right) - \left(\widetilde{\bGamma}^{\mbox{\add{\scriptsize$(n)$}}}_\mu \otimes \,
    \bI_{I_n}\right)\left(  \bI_R \otimes  \bA^{(n)} \right) \bB_{\mu}^{(n,m)}
    \left( \bI_R \otimes  \bA^{(m) \, T }  \right)\left(\widetilde{\bGamma}^{\mbox{\add{\scriptsize$(m)$}}}_\mu \otimes \,
    \bI_{I_n}\right) \notag \\ &=&
     \delta_{n,m}\left(\widetilde{\bGamma}^{\mbox{\add{\scriptsize$(n)$}}}_\mu \otimes \,
    \bI_{I_n}\right) - \left(  \bI_R \otimes  \bA^{(n)} \right)\left(\widetilde{\bGamma}^{\mbox{\add{\scriptsize$(n)$}}}_\mu \otimes \,
    \bI_R\right) \bB_{\mu}^{(n,m)}
   \left(\widetilde{\bGamma}^{\mbox{\add{\scriptsize$(m)$}}}_\mu \otimes \,
    \bI_R\right) \left( \bI_R \otimes  \bA^{(m) \, T }  \right). \notag
\ee 
%Therefore (\ref{fastblock}) holds true with
%    \be\widetilde\bSS^{(n,m)}_\mu= \left(\widetilde{\bGamma}^{(n)}_\mu \otimes \,
%    \bI_R\right) \bB_{\mu}^{(n,m)}
%   \left(\widetilde{\bGamma}^{(m,m)}_\mu \otimes \,
%    \bI_R\right).\ee
Please note that \add{the inversion of  $\bH_\mu$ in the block form saves memory}. \add{It requires to save only the matrices
$\widetilde{\bGamma}^{(n)}_\mu$} and
$\widetilde\bSS_\mu$.
While the full matrix $\bH$ or its inverse \add{has} $R^2 T^2$
elements, the memory saving format only requires to store $NR^2$ elements of matrices \add{$\widetilde\bGamma_\mu^{(n)}$} and $N^2R^4$ elements of $\widetilde\bS_\mu$.
\end{proof}

\subsection{Proof of Theorem~\ref{theo_fastdGN}}\label{sec:fastdGN}
%This section derives the fast dGN algorithm in Theorem~\ref{theo_fastdGN} from that in (\ref{equ_GaussNewton_v2}).
We replace $\bH_{\mu}^{-1}$ in (\ref{equ_GaussNewton_v2}) by those in (\ref{equ_inverse_H_LM}) in Theorem~\ref{theo_inverseH} or Theorem~\ref{theo_inverseH_block} and formulate the fast dGN algorithm
%\be
% {\sta} \leftarrow {\sta} +  {\widetilde\bG_{\mu}} \bJ^T \, \left(\sty - {\hat\sty} \right) -  \bL_{\mu} \,  {\bB_{\mu}} \, \bL_{\mu}^T   \, \bJ^T \, \left(\sty - {\hat\sty} \right). \label{equ_dGN_fastHinv}
% \ee
\be
 {\sta} \leftarrow {\sta} +  {\widetilde\bG_{\mu}} \bg-  \bL_{\mu} \,  {\bB_{\mu}} \, \bL_{\mu}^T   \, \bg. \label{equ_dGN_fastHinv}
 \ee
The Jacobian{,} which may demand high computational cost{,} still exists in the gradient $\bg$ in the update rule (\ref{equ_dGN_fastHinv}).
We also note that $\bL_{\mu}$ is a block diagonal matrix of $N$ Kronecker products \add{$\left(\widetilde{\bGamma}_{\mu}^{(n)}  \otimes \bA^{(n)} \right) \in \Real^{R I_n \times R^2}$} given in (\ref{equ_iGZ_LM}).
Construction of $\bL_{\mu}$ has a computational complexity of order $\textsl{O}\left(T\,  R^3 \right)$, and requires an extra-storage of $\textsl{O}\left(T R^3\right)$.
In order to completely bypass the Jacobian $\bJ$ in (\ref{equ_dGN_fastHinv}) {and} avoid building up the matrix $\bL_{\mu}$, we seek convenient methods for computing ${\widetilde\bG_{\mu}}
\bg$, \add{$\bw_{\mu} = \bL_{\mu}^T   \, \bg$}, and product \add{$\bL_{\mu} \, \bB_{\mu} \, \bw_{\mu}$}.

\begin{lemma}[Optimize the update rule (\ref{equ_dGN_fastHinv})]\label{theo_elim_Jacobian}
With \add{$\bA^{(n)}_{\mu}$}, $\bGamma$ and the tensor $\tF$ defined in Theorem~\ref{theo_fastdGN},
\be
	({\widetilde\bG_{\mu}} \, \bg)^T &=& {\left[
        \vtr{ \bA^{(n)}_{\mu}-  \bA^{(n)} \,{\bGamma^{\mbox{\add{\scriptsize$(n)$}}}} \, \widetilde{\bGamma}_{\mu}^{(n)}}^T
        \right]_{n=1}^{N} }\, , \label{equ_proj_GJE_a} \\
       \bw_{\mu} = \bL_{\mu}^T   \, \bg  &=&  \vtr{\left[ \bA^{(n) \, T}   \bA^{(n)}_{\mu} -  \,{\bGamma} \, \widetilde{\bGamma}_{\mu}^{(n)}  \right]_{n=1}^{N}}
        \, , \label{equ_proj_ZGJE_LM_a} \\
     {\bL_{\mu}} \, \bB_{\mu}\, \bw_{\mu}  &=&
%       {\blkdiag}{\left(\widetilde{\bGamma}_{\mu}^{(n)}  \otimes \bA^{(n)} \right)_{n = 1}^{N}}   \,\left( \left[\vtr{\bF_{n}}^T\right]_{n=1}^{N}  \right)^T\,
%        \notag\\
%    &=&
 \left[
 \begin{array}{c}
     \vtr{ \bA^{(1)}  \,
     \bF_{1} \,
     \widetilde{\bGamma}_{\mu}^{\mbox{\add{\scriptsize$(1)$}}}} \\[-1ex]
     \vdots\\[-1ex]
     \vtr{ \bA^{(n)}  \,
     \bF_{n} \,
     \widetilde{\bGamma}_{\mu}^{\mbox{\add{\scriptsize$(n)$}}}}
     \\[-1ex]
     \vdots \\[-1ex]
     \vtr{ \bA^{(N)}  \,
     \bF_{N} \,
     \widetilde{\bGamma}_{\mu}^{\mbox{\add{\scriptsize$(N)$}}}}
\end{array}     \right] . \label{equ_fast_L_psi_a}
\ee
\end{lemma}

Proof of Lemma~\ref{theo_elim_Jacobian} is given in Appendix~\ref{sec::proof_theo_elim_Jacobian}.
By replacing ${\widetilde\bG_{\mu}} \bg$, $\bL_{\mu}^T  \bg$, and $\bL_{\mu} \bB_{\mu} \bw_{\mbox{\add{\scriptsize$\mu$}}}$ in (\ref{equ_dGN_fastHinv}) by those in (\ref{equ_proj_GJE_a}), (\ref{equ_proj_ZGJE_LM_a}) and (\ref{equ_fast_L_psi_a}), \add{we obtain} a compact update rule for each factor $\bA^{(n)}, n = 1, 2, \ldots, N$ as given in Theorem~\ref{theo_fastdGN}.

We note that linear systems $\bB_{\mu} \, \bw_{\mbox{\add{\scriptsize$\mu$}}}$ in (\ref{equ_inverse_f1a})
have a computational complexity of order $\textsl{O}(N^3\, R^6)$ which is much lower than $\textsl{O}(R^3 T^3)$ for $\left(\bH + \mu \, \bI \right)^{-1}$ \add{for $N R \ll T$}.
Pseudo code of the proposed algorithm based on the update rule (\ref{equ_fast_low_LM1a}) is given in Algorithm \ref{alg_fLM}.
If components of $\bA^{(n)}$ are mutually non-orthogonal, $\bK$ is invertible, and its inverse can be explicitly computed as in Appendix~\ref{sec::app_Kernel}.
In this case, Step~\ref{Step_alg2_K_build} is replaced by (\ref{equ_inv_K}).
%According to the CP model described in (\ref{equ_CP_nolambda}), all the components
%$\ba^{(n)}_{r}, n \neq N$ except ones of the last factor are unit-length vectors.
%Although normalization for factors $\bA^{(n)}  (n \neq N)$ is not explicitly shown in Algorithm~\ref{alg_fLM},
%it is always needed after
%Step~\ref{step_update_v_lowR} in Algorithm \ref{alg_fLM}.
A practical normalization in Step~\ref{Step_alg2_normalize} is that the energy of the components is equally distributed in all modes.
The method often enhances the convergence speed of the LM iteration \cite{TomasiBro06,doi:10.1137/110843587}.

\setlength{\algomargin}{2ex}
\begin{algorithm}[t!]
\SetKwComment{mtcc}{\% }{}
\caption{\textbf{Fast Algorithm for Low-Rank Approximation}\label{alg_fLM}}
\DontPrintSemicolon
\KwIn{${\tY}$: input data of size $I_1 \times I_2 \times \cdots \times I_N$,\\
     $R$: number of basis components}
\KwOut{$N$ factors $\bA^{(n)} \in \Real^{I_n \times R}$.}
\BlankLine
\Begin{
\nl    Random or SVD initialization for $\bA^{(n)}, \forall n$\;
    \Repeat{a stopping criterion is met}{\label{step_repeat}
    \nl      \add{$\bw_{\mu} = []$}\;
        \For{$n=1$ to $N$}{
             \For(\mtcc*[f]{\small $\bK$ in Eq.~(\ref{equ_Knm_a})}){$m = n+1$ to $N$}{
\nl            ${\bK}^{(n,m)} = { \bK}^{(m,n)} = \bP_{R}\, \mbox{\add{$\diag$}}\left(\vtr{{\bGamma}^{(n,m)}}\right)${\mtcc*[f]{\small${\bGamma}^{(n,m)} = \displaystyle\bigcircledast_{k \neq n,m} \bC^{(k)},
\bC^{(n)} = \bA^{(n)T} \bA^{(n)}$}}\label{Step_alg2_K_build}
        }
\nl        ${\widetilde{\bGamma}_{\mu}^{(n)}} = \left({{\bGamma}^{(n)}}   + \mu \, \bI_R\right)^{-1}$\label{Step_alg2_invGamma}\;
\nl        $\bA^{(n)}_{\mu}   \leftarrow  \bY_{(n)} \, \left({\displaystyle\bigodot_{k\neq n} \bA^{(k)}}\right) \,   {\widetilde{\bGamma}_{\mu}^{(n)}}${\mtcc*{\small damped ALS factor}}\label{Step_alg2_dALS}
\nl        $\mbox{\add{$\bw_\mu$}} = \left[\mbox{\add{$\bw_\mu^T$}} \;  \vtr{\bA^{(n)\, T} \,  \bA^{(n)}_{\mu} -  \, {\bGamma}  \,   {\widetilde{\bGamma}_{\mu}^{(n)}}}^T\right]^T${\mtcc*{\small Eq.~(\ref{equ_proj_ZGJE_LM})}}
\nl                ${\bPsi}_{\mu}^{(n)} =   {\widetilde{\bGamma}_{\mu}^{(n)}} \otimes \bC^{(n)}${\mtcc*{\small  $\bPsi_{\mu} = \blkdiag\left({\bPsi}_{\mu}^{(n)}\right)$ in Eq.~(\ref{equ_ZiGZ_LM_a})}}\label{Step_alg2_Phi}
    }
    %
%    $\bB_{\mu} = \left({\bK}^{-1} + \blkdiag\left({\bPhi}^{(n)}\right)_{n = 1}^{N}\right)  $\;
%    {\tcc*[f]{ALS factors}}\;
    %
\nl    $\vf  =  \left({\bK}^{-1} + {\boldsymbol\Psi_{\mu}}\right)^{-1} \bw_{\mu}${\mtcc*{\small or  $ \vf  =  \bK \left(\bI + {\boldsymbol\Psi_{\mu}} \bK\right)^{-1} \bw_{\mu}$ in Eq.~(\ref{equ_inverse_f1a})}}
    \label{Step_alg2_inverse}
    \For({\mtcc*[f]{\small Update $\bA^{(n)}$ using Eq.~(\ref{equ_fast_low_LM1a})}}){$n=1$ to $N$}{
\nl        $\bA^{(n)}  \leftarrow  \bA^{(n)}_{\mu}  +  \bA^{(n)} \left( \bI_R -
         \left(  {{\bF}_{n}}  + {\bGamma}^{(n)}  \right) \,
        {\widetilde{\bGamma}_{\mu}^{(n)}}  \right)${\mtcc*[f]{\small$\vtr{\tF} = \vf $}}\label{step_update_v_lowR}
    }
\nl    Normalize $\bA^{(n)}$, $n = 1, 2, \ldots, N$       \label{Step_alg2_normalize} \;
\nl    Update $\mu$
%   $ \sta  \leftarrow \sta_{ALS}  -  {\bL_{\mu}} \,${\tcc*[f]{Update factors $\bA^{(n)}$}}\;
   }
}
\end{algorithm}

\subsection{Two variants of the fast dGN algorithm}\label{sec:variant_fastdGN}
From (\ref{equ_inverse_f1a}), we present two \add{variants} of the fast dGN algorithm which solve the corresponding inverse problem
${\bPhi}^{-1} \bw_{\mbox{\add{\scriptsize$\mu$}}}$.
%\subsubsection{Algorithm LM-1 for low-rank approximation}\label{sssec::LM1}
\begin{enumerate}[(a)]
\item {\textbf{fLM$_{\textbf a}$}.}  ${\bPhi} \triangleq {\bPhi}_{1} = \bI_{NR^2} + {\bPsi}_{\mbox{\add{\scriptsize$\mu$}}} \bK$ comprises $N$ diagonal matrices $\bI_{R^2}$,
 and $N \, (N-1)$ block matrices $\left({{\bGamma}^{\mbox{\add{\scriptsize$(n)$}}}}^{-1} \otimes  \bC^{(n)}\right)  \, \bP_{R} \, \bD^{(n,m)}$, for $n \neq m$.
% given by
% \be
% 	{\bPhi}_{1} = \bI_{NR^2} +  {\bPsi} \, \bK. \label{equ_Phi1}
%\ee
Note that ${\bPhi}_{1}$ is not symmetric, and its density is given by
\be
    d_{{\Phi}_1}  = \frac{N\, (N-1) \, R^4 + N \, R^2}{N^2 \, R^4 } = \frac{(N-1)R^2 + 1}{NR^2} \, .
\ee
For 3-D tensor factorizations, the fast dGN algorithm in which Step~\ref{Step_alg2_inverse} solves ${\bPhi}_{1}^{-1} \bw_{\mbox{\add{\scriptsize$\mu$}}}$ simplifies into the LM-1 algorithm in \cite{Petr_10}.

\item {\textbf{fLM$_{\textbf b}$}.}
${\bPhi} \triangleq {\bPhi}_{2} = \bK^{-1} + {\bPsi}_{\mbox{\add{\scriptsize$\mu$}}}$
is a symmetric matrix of size $NR^2 \times NR^2$ derived from (\ref{equ_Knm_a}) and (\ref{equ_ZiGZ_LM_a}).
% \be
% { {\bPhi}}_{2} = \bK^{-1} + \, \bPsi.\label{equ_Phi2}
% \ee
Theorem \ref{theo_diag_kernel_K} presents an explicit form of $\bK^{-1}$ which 
is a partitioned matrix of $(R^2 \times R^2)$ diagonal matrices. Hence, it has only $N^2 \, R^2$ non-zero entries.
The block diagonal matrix $\bPsi_{\mbox{\add{\scriptsize$\mu$}}}$ (\ref{equ_ZiGZ_LM_a}) is constructed from $N$ $(R^2 \times R^2)$ sub-matrices.
As a consequence, the density of the sparse matrix ${\bPhi}_2 \in \Real^{NR^2 \times NR^2}$ is
\be
    d_{\Phi_2} = \frac{N^2 \, R^2  + N \, R^4 - N\,R^2}{N^2 \, R^4}  = \frac{R^2 + N-1} {N\, R^2} \, .
\ee
\end{enumerate}

Because ${\bPhi}_{1}$  is not symmetric and less sparse than ${\bPhi}_{2}$, solving the linear system $\bPhi_1^{-1} \, \bw_{\mbox{\add{\scriptsize$\mu$}}}$ could be more time consuming than solving $\bPhi_2^{-1} \, \bw_{\mbox{\add{\scriptsize$\mu$}}}$.
Inverse of $\bK$ is not expensive and has the explicit expression given in Theorem \ref{theo_diag_kernel_K}.
However, when the factor matrices have mutually orthogonal columns, \add{$\bK$ is singular because it has collinear columns and rows}.
%We can switch between two variations by verifying whether $\bGamma^{(n,m)}$ consists of any zero or close-to-zero entries.
%
%it is adviced to use (\ref{equ_inverse_f2}) involving $\bPhi_1$ instead of (\ref{equ_inverse_f1}) involving $\bPhi_2$.
%
In Fig.~\ref{fig_Phi_12}, we illustrate the structures and properties of the two matrices $\bPhi_1$ and $\bPhi_2$ for a $3 \times 4 \times 5 \times 6 \times 7$ dimensional tensor composed by $R = 3$ rank-one tensors.
%We can reduce computational cost of inverse problems (\ref{equ_inverse_f1}), (\ref{equ_inverse_f2}) by employing sparsity and symmetric structures of ${\bPhi}_{1}$ and ${\bPhi}_{2}$.

%For large-scale and sparse matrix ${\bPhi}_{2}$ when $R$ and $N$ are both large,  the minimum residual method (MINRES) \cite{PaigeMINRES,templates} is recommended to solve the inverse problem
%\be
%    {\bPhi}_{2} \, \vf =   \bw_{\mu}  \, . \label{equ_inverseproblem_LM}
%\ee

 \begin{figure}[t]
\centering
\hfill
\subfigure[$d_{{\Phi}_1}  = \frac{(N-1)R^2 + 1}{NR^2}$.]{
\psfrag{Phi1}[bc][bc]{\small $\bPhi_1 = \bI + \bPsi_{\mu} \, \bK$}
\includegraphics[width=.4\linewidth, trim = 0cm 0.0cm 0cm 0cm , clip = true]{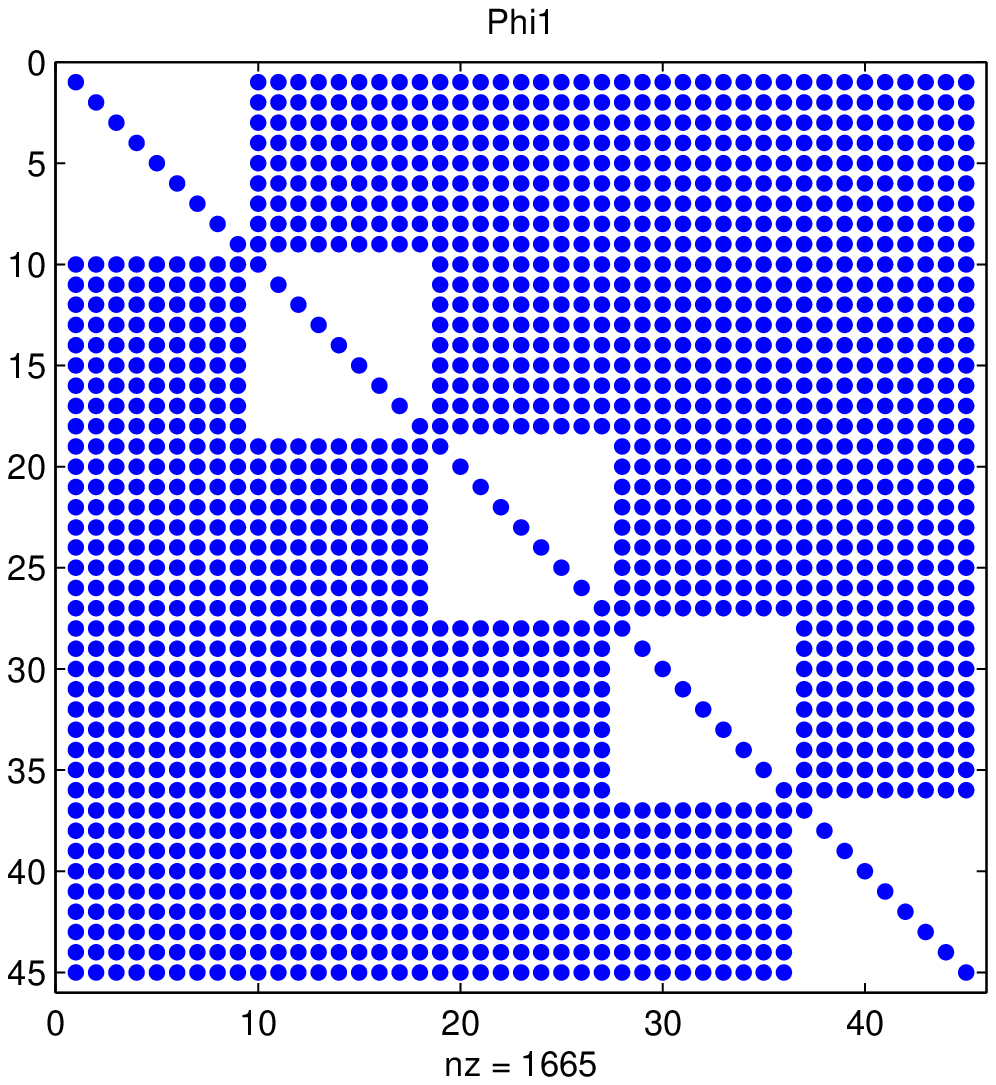}}
\hfill
\subfigure[$d_{\Phi_2} = \frac{R^2 + N-1} {N\, R^2}$.]
{\psfrag{Phi2}[bc][bc]{\small$\bPhi_2 = \bK^{-1} + \bPsi_{\mu}$}
\includegraphics[width=.40\linewidth, trim = 0cm  0cm 0cm 0cm , clip = true]{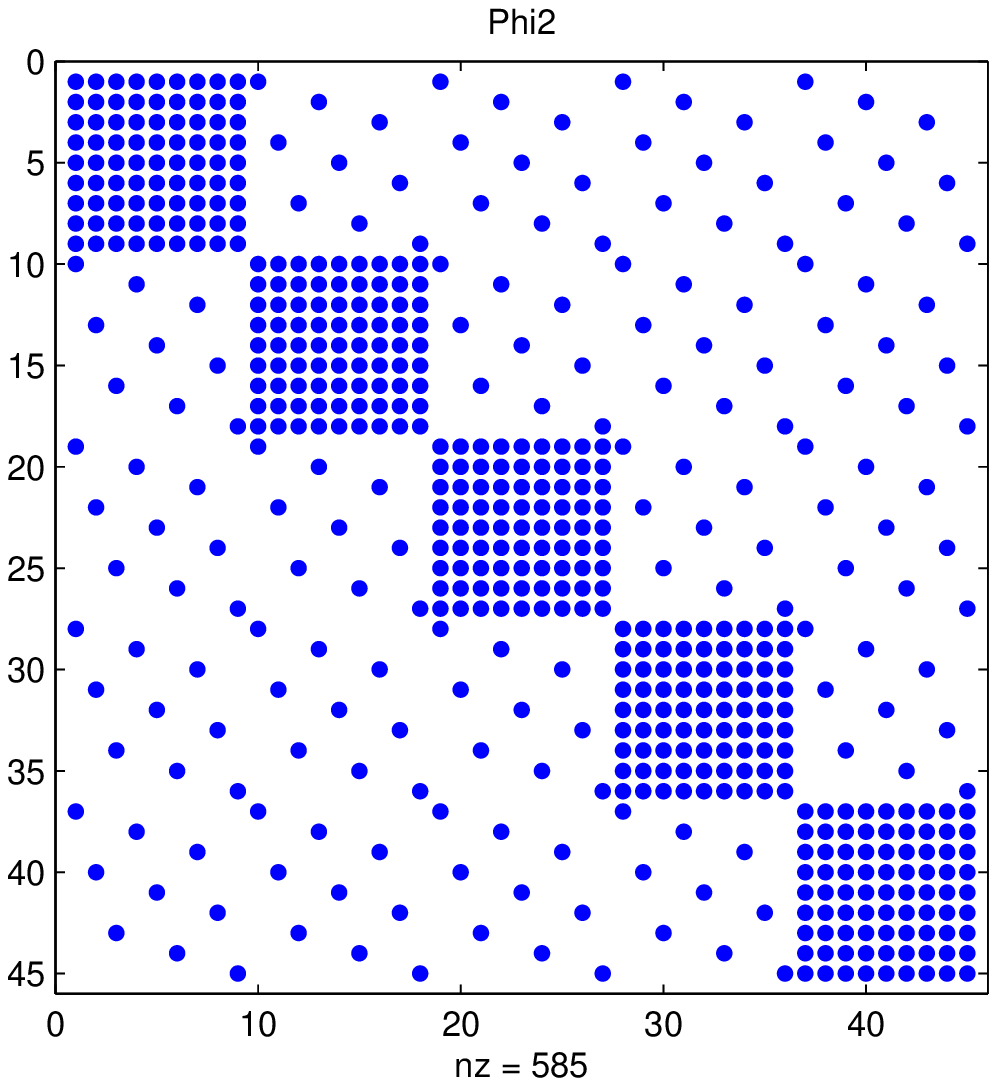}}
\hfill
\vspace{-2ex}
\caption{Illustration of structure of $NR^2 \times N R^2$ sparse matrices $\bPhi_1$ and $\bPhi_2$ for a $3 \times 4 \times 5 \times 6 \times 7$ dimensional tensor composed by $R = 3$ rank-one tensors.
The matrix $\bPhi_1$ is less sparse than the matrix $\bPhi_2$.
Blue dots denote nonzero entries.}
\label{fig_Phi_12}
\end{figure}

%\subsubsection{Complexity of the fast dGN algorithm}\label{sec::complexity}
\subsection{Comparison of complexity between dGN and fast dGN}\label{sec::complexity}
In general, the dGN algorithm \cite{Paatero97,Tomassithesis} constructs the whole approximate Hessian of size $RT \times RT$ from its submatrices $\bH^{(n,m)}$ (see Appendix~\ref{sec::proof_theo_decomH}) which are deduced from ${\bC}^{(n)}$ and ${\bGamma}^{(n)}$.
Computation of ${\bC}^{(n)}$ and ${\bGamma}^{(n)}$ are with of complexity $\textsl O\left(R^2 T\right)$ and $\textsl O\left(NR^2 \right)$, respectively.
According to Theorem~\ref{theo_Hnm_full}, each off-diagonal submatrix  has a complexity of $\textsl O\left(R^2 I_n I_m \right)$, it follows that computation of the whole $\bH$ has the complexity \add{of} $\textsl O\left(R^2T^2\right)$. Note that $\bH$ has $R^2T^2$ elements.  Inverse $\bH^{-1}$ can be computed with a complexity of $\textsl O\left(R^3 T^3\right)$.
The gradient $\bg$ is computed at a cost of $\textsl O\left(NR \mbox{\add{${J}$}}\right)$.
Thus dGN has a complexity per iteration of $\textsl O\left(NR \mbox{\add{${J}$}} + R^3 T^3\right)$.

Complexity of the fLM algorithm is analyzed for each step in Algorithm~\ref{alg_fLM} as follows
\begin{description}
\item[Step~\ref{Step_alg2_K_build}] computes $N$ matrices ${\bC}^{(n)}$ and ${\bGamma}^{(n)}$ with complexity $\textsl O\left(R^2 T\right)$ and $\textsl O\left(NR^2 \right)$ as in dGN. Hence, building up $\bK$ is of complexity $\textsl O\left(N(N-1)(N-2)R^2 \right) = \textsl O\left(N^3R^2 \right)$.

\item[Step~\ref{Step_alg2_invGamma}] \add{inverts} $\bGamma^{(n)}_{\mu}$ , $n = 1, 2, \ldots, N$ at a cost of $\textsl O\left(NR^3 \right)$.

\item[Step~\ref{Step_alg2_dALS}] computes the damped factors $\bA^{(n)}$ at a cost of $\textsl O\left(NR\displaystyle{J}\right)$, and is one of the most expensive \add{steps} in the fast dGN algorithm.
We note that the large workload $\bY_{(n)} \, {\displaystyle \bigodot_{k\neq n}\bA^{(k)}}$ is used for evaluation of gradient, and exists in all CP algorithms such as ALS, OPT.

\item[Step~\ref{Step_alg2_Phi}] builds up the block diagonal matrix \add{$\boldsymbol\Psi_{\mu}$} with a complexity $\textsl O\left(N R^4 \right)$.

\item[Step~\ref{Step_alg2_inverse}] solves the inverse problem \add{$\bPhi^{-1} \bw_{\mu}$} with a cost of $\textsl O\left(N^3 R^6 \right)$.
This step is much faster than inverse of the approximate Hessian $\textsl O\left(R^3 T^3\right)$ due to $R \ll I_n $ or $N R <  T$.
\end{description}

Instead of construction of the approximate Hessian, the fLM algorithm builds up the much smaller matrix $\bPhi$ of size $NR^2 \times NR^2$. Hence, besides the cost of computation of the gradient or the damped ALS factors, fLM computes $\bPhi$ and $\bPhi^{-1}$ at a cost of $\textsl O\left(R^2 T + N^3 R^6\right)$ which is much smaller than the cost for construction of $\bH$ and for $\bH^{-1}$ in dGN.

The total expense of fLM per one iteration is approximately $\textsl O\left(NR\displaystyle{J} + N^3 R^6\right)$.
%
%As a consequence, solving an inverse problem $\vf = \bK \, {\bPhi}_{1}^{-1} \, \bw_{\mu}$,  $\bPhi_1 = \bI_{NR^2} + {\bPsi} \bK$ 
%or $\vf = {\bPhi}_{2}^{-1} \, \bw_{\mu}$, $\bPhi_2 = \bK^{-1} + {\bPsi}$ is  much less expensive than solving the inverse problem $\left( \bH + \mu \bI \right)^{-1} \bg$.
%
%Moreover, the update rule (\ref{equ_fast_low_LM1a}) does not employ the approximate  Hessian $\bH$ and the Jacobian $\bJ$. Hence, they are significantly faster than the dGN algorithms proposed by \cite{Paatero97,Tomassithesis}.
%%
For $N>7$,  the proposed algorithm has the same order of complexity as that of ALS. However, fLM is much faster than ALS because it requires less iterations than ALS.

%\begin{table}
%\begin{tabular}{lcccc}
%$\bC^{(n)}$ $n = 1, \ldots, N$ 	& $\textsl O\left( R^2T\right)$  & $\textsl O\left( N R^2 I\right)$\\ 
%$\bGamma^{(n,m)}$ $n,m = 1, \ldots, N$ 	& $\textsl O\left( N^3 R^2\right)$ \\
%$\bH^{(n)}$ $n = 1, \ldots, N$ 	& $\textsl O\left( N  R^2\right)$ \\
%$\bH^{(n,m)}$ $n = 1, \ldots, N$ 	& $\textsl O\left(R^2  I_n I_m\right)$ \\
%$\bH$ &  $\textsl O\left(R^2  \left(\sum_{n\neq m} I_n I_m\right)\right)$  &  $\textsl O\left(N^2R^2 I^2\right)$ \\
%\end{tabular}
%\end{table}
%
%
%\begin{table}
%\begin{tabular}{ccccc}
%Algorithm & Gradient & Hessian & $\bH^{-1} \bg$ \\
%dGN 	& $\textsl O\left(NR\displaystyle{J}\right)$ & 
%$\textsl O\left( R^2 T + \right)$ 
%&$\textsl O\left( R^3 T^3\right)$\\
%fLM  &  $\textsl O\left(NR\displaystyle{J}\right)$ & 
%&$\textsl O\left(NR\displaystyle{J} + R^2 T +N^3 R^6\right)$
%\end{tabular}
%\end{table}

%%

%\be
%	&&NR I_1 I_2 \cdots I_N - R^2 ( I_1 + I_2 + \cdots + I_N)  \\
%	&\rightarrow& N  I_1 I_2 \cdots I_N - 1/N ( I_1 + I_2 + \cdots + I_N) ^2  \\
%	&\rightarrow& N^2  I_1 I_2 \cdots I_N - ( I_1 + I_2 + \cdots + I_N)^2  \\
%	& =& 	N^2  I_1 I_2 \cdots I_N -  I_1^2 - 2 I_1 (I_2 + \cdots + I_N) - (I_2 + \cdots + I_N)^2   \\	
%	& =& 	(N^2  I_1 - (N-1)^2 ) P_2  -  I_1^2 - 2 I_1 S_2 + 
%	(N-1)^2  I_2 \cdots I_N - (I_2 + \cdots + I_N)^2   \\	
%\ee
%
%\be
%f(I_1) = (N^2  I_1 - (N-1)^2 ) P_2  -  I_1^2 - 2 I_1 S_2\\
%f'(I_1) = N^2   P_2  - 2  S_1  \\
%\ee
%
\subsection{Damping parameter in the LM algorithm}\label{sec::damping_update}

The choice of damping parameter $\mu$ in the fast dGN algorithms (\ref{equ_fast_low_LM1a}) affects the direction and the step size  $\Delta\sta = \bH_{\mu}^{-1}\, \bg$ in the update rule (\ref{equ_GaussNewton_v2}):
$\sta \leftarrow \sta + \Delta\sta$ \cite{NielsendampingLM}.
%Hence, this can also affect the convergence speed of algorithms. 
In this paper, the damping parameter $\mu$ is updated using the  efficient strategy proposed by Nielsen \cite{NielsendampingLM}:
\be
    \mu &\leftarrow&\begin{cases}
        2 \max \left\{ \displaystyle\frac{1}{3}, 1- \left(2 \rho - 1 \right)^3\right\},         \quad  &   \, \rho > 0,  \\
        2 \mu  ,  & otherwise,  \\
    \end{cases} \\
        \rho &=& \frac{\|\ve_{t-1}\|^2_2 -  \|\ve_{t}\|^2_2 }{\Delta\sta^T \left( \bg + \mu \, \Delta\sta \right)} , \\
        \bg &=&   \bJ^T \, \left( \by - \hat\by \right) =
    \left[
        \begin{array}{c}
        \vtr{\bY_{(1)} \, \left({\displaystyle\bigodot_{k\neq 1}\bA^{(k)}}\right)  \, -
        \bA^{(1)} \, {{\bGamma^{(1)}}}} \\
        \vdots \\
        \vtr{\bY_{(N)} \, \left({\displaystyle\bigodot_{k\neq N}\bA^{(k)}}\right)  \, -
        \bA^{(N)} \, {{\bGamma^{(N)}}}} \\
        \end{array}
    \right] \, \quad \in \Real^{R \, T}
     \, . \label{equ_mappingJ}
\ee
where $\ve_{t} = \vtr{{\tY}- {\hat\tY}_{t}}$, the gradient $\bg$ can be straightforwardly derived as in (\ref{equ_proj_GJE}) or in \cite{TomasiBro05,Tomassithesis}.
%\begin{equationarray}{c@{\hspace{1ex}}c@{\hspace{1ex}}c@{\qquad}c@{\hspace{1ex}}c@{\hspace{1ex}}c}
%    \ve_{t} &=& \vtr{{\tY}- {\hat\tY}_{t}}, &
%    \rho &=& \frac{\|\ve_{t-1}\|^2_2 -  \|\ve_{t}\|^2_2 }{\Delta\sta^T \left( \bg + \mu \, \Delta\sta \right)} ,\\
%    \Delta\sta  &=& \sta_{t} - \sta_{t-1} ,  &
%    \bg &=& \bJ^T \, \ve_{t-1} \, .
%\end{equationarray}
The factors $\bA^{(n)}$ will be updated unless the new approximate is lower than the previous one:  $\|\ve_t\|_2 < \|\ve_{t-1}\|_2$.
The algorithm \add{should} stop when $\mu$ increases to a sufficiently large value (e.g., $10^{30}$).
In practice, \add{the} factors $\bA^{(n)}$ are often initialized using the mode-$n$ singular vectors of the data tensor \cite{Lathauwer_HOOI,Kolda08,NMF-book}, then run over ALS (\ref{equ_ALS}) after few iterations.
According to the CP model (\ref{equ_CP_nolambda}), all the components
$\ba^{(n)}_{r} (n \neq N)$ except ones of the last factor are unit-length vectors.
The initial value of the damping parameter $\mu$ is chosen as the maximum diagonal entry of  $\bH$ as
\be
    \mu_0 &=& \tau \max\left\{ \diag\left(\bH\right)  \right\}
    =  \tau \max\left\{ \diag\left({\boldsymbol \bGamma}^{(1)}\right)  \, \cdots \,\diag\left({\boldsymbol \bGamma}^{(n)}\right) \, \cdots\, \diag\left({\boldsymbol \bGamma}^{(N)}\right) \right\}  \notag\\ 
    &=& \tau \max\left\{1,  \diag\left(\bC^{(N)} \right)\right\} 
    \, , \label{equ_mu0}
\ee
where $\tau$ is typically in the range \add{of} $[10^{-8}, 1]$.

\section{Complex-valued tensor factorization}\label{sec::complexalgorithm}
This section aims to extend the dGN algorithms to complex-valued tensors.
Although a real-valued tensor is considered as a complex-valued tensor with zero imaginary part,
for simplicity algorithms for real- and complex-valued tensors are introduced in two separate sections.
For the complex case, CP model is to find complex-valued factors $\bA^{(n)} \in \Complex^{I_n \times R}$.

The damped Gauss-Newton-like update rule (\ref{equ_GaussNewton_v2}) is rewritten to update complex-valued factors \cite{Patrick96,Sorenson80}
%\\[-4ex]
\be
    \boxed{{\sta} \leftarrow {\sta} + \left(\bJ^H \, \bJ  + \mu \bI \right)^{-1} \, \bJ^H \, \left( \sty - {\hat\sty} \right), \label{equ_GaussNewton_v_complex}}
\ee
where symbol ``$H$'' denotes the Hermitian transpose, and the Jacobian $\bJ$ is given in (\ref{equ_Jacobian}).
The approximate Hessian $\bH = \bJ^H \, \bJ$ slightly changes from that for the real-valued tensors. 
A fast and efficient computation method for the complex-valued approximate Hessian $\bH$ will be presented so that the final update rule does not employ both of the Jacobian and the approximate Hessian. 
We consider $\bH$ as a partitioned matrix of $\left(N \times N \right)$ sub-matrices $\bH^{(n,m)} \in \Complex^{RI_n \times RI_m}$, $n, m = 1, 2, \ldots, N$.
Each sub-matrix $\bH^{(n,m)}$ is a partitioned matrix of  $\left(R \times R \right)$ subsub matrices $\bH^{(n,m)}_{r,s} \in \Complex^{I_n \times I_m}$, $n, m = 1, 2, \ldots, N$, $r,s = 1, 2, \ldots, R$.
The explicit expression of the approximate Hessian $\bH$ is deduced from the following theorems which can be derived in a similar manner as for real valued tensors.

\begin{theorem}[Subsub-matrices $\bH^{(n,m)}_{r,s}$]\label{theo_complex_Hnm_rs}
$\bH^{(n,m)}_{r,s}$ are diagonal or rank-one matrices given by
%\\[-4ex]
\be
    \bH^{(n,m)}_{r,s} =  \delta_{n,m} \, \gamma^{(n)}_{r,s} \, \bI_{I_n} + (1 - \delta_{n,m}) \,  
            \gamma^{(n,m)}_{r,s} \, \ba^{(n)}_{r} \, \ba^{(m) \, H}_{s} \, ,
 \label{equ_H_nmrs_complex}
\ee
where $\gamma^{(n)}_{r,s}$ are the $(r,s)$ entries of the Hermitian matrices ${\bGamma}^{(n,m)} = \displaystyle\bigcircledast_{k\neq n,m} \bA^{(k)\, H} \, \bA^{(k)} \label{equ_Gamma_complex}$.
\end{theorem}

\begin{theorem}[Sub-matrices $\bH^{(n,m)}$]\label{theo_complex_Hnm}
With $\bK$ defined as in (\ref{equ_Knm_a}), $\bH^{(n,m)}$ are expressed in an explicit form as
%\\[-4ex]
\be
    \bH^{(n,m)} =  \delta_{n,m}  
            \mbox{\add{$\left({\bGamma}^{(n)} \, \otimes \bI_{I_n} \right)$}}+
                \left(  \bI_R \otimes  \bA^{(n)} \right) \, 
                \bK^{(n,m)}\,
                \left( \bI_R \otimes  \bA^{(m) \, H }  \right)  \, .\label{equ_H_nm_complex}
\ee
\end{theorem}

\begin{theorem}[Low-Rank Adjustment]\label{theo_decomH_complex}
For $NR \ll T$, the approximate Hessian $\bH = \bJ^H \, \bJ$ can be expressed as a low-rank adjustment given by
%\\[-4ex]
\be
    \bH = \bG + \bZ \, \bK \, \bZ^H \, ,\label{equ_Hessian_complex}
\ee
where sparse matrices $\bG$, $\bZ$ and $\bK$ are defined as in (\ref{equ_G}), (\ref{equ_Z}) and (\ref{equ_Knm_a}).
\end{theorem}

The damped Gauss-Newton algorithms for complex-valued tensor factorization are stated in following theorems:
 \begin{theorem}[damped GN algorithm for complex-valued tensor factorizations] \label{theo_alg_complx_GN}
 The factors $\bA^{(n)}$ are updated using the  rule given by
% \\[-4ex]
 \be
    \sta &\leftarrow& \sta  +   \, \left(\bH  + \mu \, \bI \right)^{-1} \,\bg \, , \label{equ_general_complex}
 \ee
 where the approximate Hessian $\bH$ is defined in Theorems \ref{theo_complex_Hnm_rs} or \ref{theo_complex_Hnm},  an Levenberg-Marquardt regularization parameter $\mu >0$ and the gradient $\bg \in \Complex^{RT }$ is computed as\\[-4ex]
 \be
    \bg = {\left[ \vtr{\bY_{(n)}  \, \left(\bigodot_{k\neq n} \bA^{(k) *} \right)    -   \bA^{(n)}   \, {{\bGamma}^{(n)}}^T}^T \right]_{n=1}^{N}}^T \,  ,\label{equ_grad_complex_g}
 \ee
where symbol `*' denotes the complex conjugate.
\end{theorem}

 \begin{theorem}[fast dGN for low rank approximation] \label{theo_alg_complx_sfast_Hessian}
For $NR \ll T$, the factors $\bA^{(n)}$ are updated using the fast update rule given by\\[-3ex]
\be
    \boxed{\bA^{(n)}  \leftarrow  \bA^{(n)}_{\mu}  +  \bA^{(n)} \left( \bI_R -
         \left(  {{\bF}_{n}}  + {{\bGamma}^{(n)}}^T  \right) \,
        {\widetilde\bGamma}_{\mu}^{(n)} \right) \, ,}\label{equ_fast_low_LM1-complex}
\ee
% \be
%    \sta &\leftarrow& \sta + \sta_{ALS} - {\hat\sta}_{ALS} -  {\bL_{\mu}} \, \bB_{\mu} \,{\bw_{\mu}}\, , \label{equ_sfast_1_complex}
% \ee
where ${{\bF}_{n}}$ \add{are} frontal slices of a 3-D tensor $\tF$ whose $\vtr{\tF} = \bB_{\mu}$\add{$\bw_{\mu}$},
 $\bB_{\mu}  = \left(\bK^{-1} + \bPsi_{\mu} \right)^{-1}$ if $\bK$  is invertible, or $\bB_{\mu} = \bK \, \left(\bI + \bPsi_{\mu} \, \bK\right)^{-1}$, and \add{$\bw_{\mu}$} is computed from the damped ALS factors \add{$\bA^{(n)}_{\mu}$}
\be
	{\widetilde \bGamma}_{\mu}^{(n)} &=& \left({{\bGamma}^{(n)}}  + \mu \, \bI_R \right)^{-1} , \\ 		
	\bPsi_{\mu} &=& {\blkdiag}\left( \widetilde{\bGamma}_{\mu}^{(n)}\otimes
            \bA^{(n) H} \bA^{(n)}    \right)_{n=1}^{N}  \, , \label{equ_ZiGZ_LM_cplx} \\
    {\bw_{\mu}} &=& %\bZ^H \, {{\bG}_{\mu}}^{-1} \, \bJ^{H} \, \vtr{\ten{\bE}} &=&
        \vtr{\left[ \bA^{(n) \, H}  \,  \bA^{(n)}_{\mu} -  
        \bGamma \, {\widetilde \bGamma}_{\mu}^{(n)}
         \right]_{n=1}^{N}}
        \, , \label{equ_proj_ZGJE_LM_complex}  \\
    \bA^{(n)}_{\mu} &= & \bY_{(n)} \, {\left(\bigodot_{k\neq n} \bA^{(k) *} \right) } \, {\widetilde \bGamma}_{\mu}^{(n)}\,  .
%     \\
%    {\hat\bA}^{(n)}_{ALS} &= & \bA^{(n)} \, {{\bGamma}^{(n,n) }}^T  \, {\left({{\bGamma}^{(n)}}^T + \mu \, \bI\right)^{-1}} \, .
\ee
\end{theorem}
%We note that ALS for complex-valued CP with damping parameter $\mu $ is given by
%\be
%    \bA^{(n)}_{\mu} \leftarrow
%%     \bY_{(n)} \, \left\{\bA^{*} \right\}^{\odot_{-n}} \, {\left( \{\bA^T \, \bA^{*}\}^{\*_{-n}}+ \mu \, \bI\right)^{-1}} \,
%%     =
%       \bY_{(n)} \, {\left(\bigodot_{k\neq n} \bA^{(k) *} \right) } \, {\left({{\bGamma}^{(n)}}^T + \mu \, \bI\right)^{-1}} \, .
%\ee

%The update rule (\ref{equ_sfast_1_complex}) can be reformulated for each factor $\bA^{(n)}$ as
%\be
%    \boxed{\bA^{(n)}  \leftarrow  \bA^{(n)}_{\mu}  +  \bA^{(n)} \left( \bI_R -
%         \left(  {{\bF}_{n}}  + {{\bGamma}^{(n)}}^T  \right) \,
%        \left({{\bGamma}^{(n)}}^T  + \mu \, \bI_R \right)^{-1}  \right) \, ,}\label{equ_fast_low_LM1-complex}
%\ee
%where ${{\bF}_{n}}$ is defined in Theorem~\ref{theo_fastdGN}. Kronecker product has been eliminated in building up $\bL_{\mu}$.

\section{Experiments - Computer simulations} \label{sec::experiment}

%The dGN algorithm has been successfully confirmed for difficult data (such as collinearity of factors, different magnitudes of factors) by Paatero \cite{Paatero97}. Later, the dGN has been validated again by Tomasi \cite{TomasiBro06,Tomassithesis,Tomasi_tricap06,INDAFAC}.
%%A variation of the dGN algorithm for three way data is the INDAFAC algorithm \cite{INDAFAC}
%%which computes the approximate Hessian, and gradient, and employs Cholesky decomposition to deal with inverse problems $\bH^{-1}$. Therefore, it is straightforward to see that INDAFAC demands much higher computational cost than that of our fast dGN algorithm for three-way data.
%%Moreover, with the same initial values and damping parameters, the dGN and fast dGN algorithms should return similar results.
%The major difference between the fast dGN algorithm and common dGN algorithms is complexity. 
%However, this task is equivalent to determining how much inverse of an $\left(N R^2 \times NR^2 \right)$ matrix in our fast algorithm is faster than inverse of an $\left(\displaystyle R \sum_n{I_n} \times R \sum_n {I_n}\right)$ matrix ($R \ll I_n$) in other dGN algorithms.
%Therefore, in this paper, we do not intend to compare running time between dGN algorithms.
%
  %
The CP algorithms were verified for difficult data with collinear factors in all modes (swamp). Collinearity degree of factors \add{was} controlled by mutual angles between their components.
Collinear factors $\bA^{(n)}$ were generated from random orthonormal factors $\bU^{(n)}$
%Orthonormal factors $\bU^{(n)}$ were first randomly generated from the normal distribution, then converted to collinear factors $\bA^{(n)}$ by a simple modification
\be
    \ba^{(n)}_{r}  =  \bu^{(n)}_{1}  +  \nu \, \bu^{(n)}_{r} \,,  \quad  \quad \nu \in (0, 1], \forall n, \forall r \neq 1 \, . \label{equ_modify_1}
\ee
Mutual angles $\theta_{q,r}$ between $\ba^{(n)}_{q}$ and $\ba^{(n)}_{r}$, $q \neq r$ were in a range of $(0, 60^{o}]$ for $\nu \in (0, 1]$\\[-3ex]
\be
	\tan(\theta_{q,r}) = \begin{cases}
			 \nu  ,&   q = 1, \\
			 \nu \sqrt{\nu^2+2} ,\quad &q \neq 1, r .
			 \end{cases}
\ee
For example, $\nu = 0.1, 0.2, \ldots, 1$ yield $\theta_{1,r} = 6^{\text o},   11^{\text o},   17^{\text o},   22^{\text o},   27^{\text o},   31^{\text o},   35^{\text o}$,  $39^{\text o},   42^{\text o},   45^{\text o}$,
and $\theta_{q,r} = 8^{\text o},   16^{\text o},   23^{\text o},   30^{\text o}$,  $37^{\text o},   43^{\text o},   48^{\text o}$,  $52^{\text o},   56^{\text o}$,   $60^{\text o}$, $q \neq 1, q \neq r$, respectively.
For high $\nu$ such as $\nu = 2$, $\theta_{1,r} \approx 63^{\text o}$ and $\theta_{q,r} \approx  78^{\text o}$, tensor can be quickly factorized by CP algorithms. 
The higher the parameter $\nu$, the lower the collinearity of factors.
It is more difficult to factorize tensors with lower $\nu$ (e.g.\add{,} $\nu$ = 0.1, 0.2).
However, \add{when $\nu > 3$, another issue arises from large difference in magnitude between components.}  
The tensors are still difficult to factorize even thought collinearity of factors is low ($\theta_{1,r} > 71^{\text o}$). CP tensors{, as} in (\ref{equ_CP_nolambda}){,} can \add{equivalently be constrained to be of the form}\\[-3ex]
\be
	\tY = \sum_{r = 1}^{R} \lambda_r \, \ba^{(1)}_r \, \circ \ba^{(2)}_r \, \circ \cdots \ \circ \ba^{(N)}_r,
\ee
where $\|\ba^{(n)}_{r}\|_2 = 1, \forall r$, and \add{each $\lambda_r$ encodes the magnitude}. \add{For this experiment} 
$\lambda_1 = 1$, and $\lambda_r = (1 + \nu ^2)^{N/2}, \forall r > 1$.
Therefore, for $\nu = 3, 4, 5$ and $N = 3$, $\lambda_r = 31.6, 70.1, 132.6, \forall r \neq 1$, respectively.
\add{That means the components $\ba^{(n)}_r$, $r = 2, \ldots, R$ are relatively larger than the first component}.
We analyze synthetic tensors for two cases: error-free and noisy data with additive white Gaussian noise at SNR $(=-10\,\log_{10}{\frac{\|\tY\|_F^2}{\sigma^2 \prod I_n}})$ = 30 dB or 40 dB added to the data tensor
$
{\tensor{\widetilde Y}} =  \tY +  {\sigma} \tN \,  ,
$
where $\tN$ denotes a normally-distributed random tensor of zero mean and unit variance whose $n_{i_1 i_2 \ldots i_N}  \sim N(0,1)$.
%$\sigma = 10^{-SNR/20}   \frac{\|\tY\|_F}{\sqrt{\prod In}}$.

%The OPT algorithm has been tested for collinear data with the same mutual angles $\theta_{q,r} = 26^{\text o}$  and $60^{\text o}$, $\forall q \neq r$ in \cite{JChem-CPOPT}. In our experiments, $\theta_{1,r}$ and  $\theta_{q,r}$, $q > 1$ are different, $\theta_{1,r} < \theta_{q,r}$, and are in a wider range $[6^{\text o}, 60^{\text o}]$.

In order to evaluate the factorizations for collinear data, we measured the Median Squared Angular Error (MedSAE) over multiple runs between the original and estimated components $\ba^{(n)}_r, {\widehat{\ba}}^{(n)}_r$ after matching their orders defined as\\[-3ex]
\be
{\text{MedSAE}}(\ba^{(n)}_r, {\widehat\ba}^{(n)}_r) = 10\log_{10}\left(\text{median}\left(\alpha^{(n) 2}_{r} \right)\right) \; \;(\text{dB}),\label{equ_MedSAE}
\ee
where $\alpha^{(n)}_{r} = \arccos{\frac{\ba^{(n) H}_r \, {\widehat\ba}^{(n)}_r}{\|\ba^{(n)}_r\|_2 \|\widehat\ba^{(n)}_r\|_2}}$.
Cram\'er-Rao Induced Bound (CRIB) on $\alpha^{(n) 2}_{r}$ was computed from the Cram\'er--Rao lower bound (CRLB) for estimating the component $\ba^{(n)}_{r}$
\cite{DBLP:conf/icassp/TichavskyK11,PetrCRIB,DBLP:journals/tsp/TichavskyK11,ZbynekSSP11} 
\be
	\text{CRIB}(\alpha^{(n)2}_r)  =  10\log_{10}\frac{\tr{\left((\bI_{I_n}  - \ba^{(n)T}_{r} \ba^{(n)}_{r}/ {\|\ba^{(n)}_{r}\|^2}) {\text{CRLB}(\ba^{(n)}_{r})}\right)}}{\|\ba^{(n)}_{r}\|^2} \, \quad (\text{dB}).
\ee
For our simulations, due to the same collinearity degree $\nu$ for all the components, we have
\be
\text{CRIB}(\alpha^{(n) 2}_r) &=&  \text{CRIB}(\alpha^{(1) 2}_r), \quad \forall r, \forall n, \notag\\
\text{CRIB} (\alpha^{(n) 2}_r)  &=& \text{CRIB}(\alpha^{(n) 2}_2),  \quad \forall n, r  = 2, \ldots, R. \notag
\ee

The average MedSAEs for the estimated components were compared against the average CRIB.
%Fig.~\ref{fig_CRIB} illustrates the angular CRIB versus $\nu$ given in (\ref{equ_modify_1}) at SNR = 30 dB for various sizes $R$ and dimensions $N$. Legend describes factor sizes $I_n \times R$ and dimension of tensors $N$.
%CRIB for the same tensors at other SNR can be straightforwardly deduced from those in Fig.~\ref{fig_CRIB}.
%For example, CRIB at SNR = 20 dB or 40 dB is shifted up or down 10 dB from that at SNR = 30 dB. CRIB for the test case of $N = 3$, $I_n = 50$, $R = 20$ and at SNR = 20 dB 
%deduced from that at SNR = 30 dB (dark-yellow line) is above -30 dB.
%We also have CRIB for the test case of $N = 3$, $I_n = 50$, $R = 5$ and SNR = 40 dB 
%deduced from that at SNR = 30 dB (pink line) is lower than -30 dB.
It is important to note that an MedSAE lower than -30 dB, -26 dB or -20 dB means two components are different by a mutual angle less than $2^{\text{o}}$, $3^{\text{o}}$ and $6^{\text{o}}$, respectively.
%This gives us a rough evaluation of a rather good approximation in which MedSAE $< -30$ dB.
%As seen in Fig.~\ref{fig_CRIB}, the angular CRIB for the test case of $I_n = 50$, $R = 5$, $N = 3$ and SNR = 30 dB is lower than -30 dB for $\nu = 0.2, 0.3, \ldots, 1$. That means we can obtain a good approximate in which mutual angles of components are less than $2^{\text{o}}$.
 Practical simulations show that it is difficult for MedSAE to reach a CRIB $\ge$ -30 dB, since collinearity of factors has been destroyed by noise. Discussion on effects of noise on collinear data in Appendix~\ref{sec::Collinear_and_noise} gives us insight into when CP algorithms are not stable,  and when they succeed in retrieving collinear factors from noisy tensors.
%Fig.~\ref{fig_CRIB} also reveals that we cannot retrieve collinearity components from factorization of 3-D tensors of rank $R = 15, 20$ at SNR = 20 dB for $\nu < 0.5$.
%
%Therefore, we will not analyze simulations having CRIB $\ge$ -26 dB.
%We do not analyze simulations in which CRIB $\ge$ -26 dB.

\subsection{Comparison between dGN and fLM for 3-D tensor factorizations}

This section compares \add{performance of fLM and} the standard dGN algorithm in the Matlab routines PARAFAC3W developed by Tomasi \cite{TomasiBro06,INDAFAC}.
The dGN algorithm \cite{INDAFAC} computes the approximate Hessian and gradient, and employs Cholesky decomposition and back substitution to solve the inverse problems $\bH^{-1} \bg$.
Unfortunately, this toolkit supports only 3-D data.
The {fLM$_{\text a}$} algorithm was verified, and shortly denoted by fLM.

%with its two variations {fLM$_{\text a}$} and {fLM$_{\text b}$} in Section \ref{sec:variant_fastdGN} is denoted by .
%We note that the two variations {fLM$_{\text a}$} and {fLM$_{\text b}$} are equivalent in the sense of performance.

In the first set of experiments, random synthetic tensors were generated from 3 collinear factor matrices of size $I \times R$ where $I = 100$ and $R = 5, 10, 20, 30, 40, 60$ and $\nu = 0.5$.
%Tensors having $R\approx I$ were not considered in our simulations.  
\add{From each noise-free CP tensor $\tY$ composed from $\bA^{(n)}\in \Real^{I \times R}$, twenty noisy tensors $\tensor{\widehat Y}$ of 30 dB SNR were generated. 
There are in total 200 rank-$R$ tensors $\tensor{\widehat Y}$.}
MedSAE for each component was deduced from 200 runs for each test case.
%
% However, instead of presenting the average or median value of 20 MedSAE values, we computed the median squared angular errors over 200 runs for each test case.

%
Both algorithms were initialized by the same factors which were the mode-$n$ singular vectors of the data tensor \cite{Lathauwer_HOOI}.
\add{Algorithms stop when 10 differences of successive relative errors $ \varepsilon = \displaystyle \frac{\|{\tY} - {{\tensor{\widehat Y}}}\|_F}{\|{\tY}\|_F}$ were lower than $10^{-8}$, or until the maximum number of iterations (1000) was achieved}.
Execution time for each algorithm was measured using the stopwatch command: ``tic'' ``toc'' of MATLAB release 2009a on a computer which \add{had} 2 quadcore 3.33 GHz processors and 64 GB memory.
Tucker compression was not used in the simulations.
The dGN in \cite{INDAFAC} was adapted to follow the same stopping criteria and the same computational time measurements, while its other parameters were set to default values.
%Printing results of algorithms to the screen was disabled. 

Fig.~\ref{fig_dGNvsfLM_3D_rtime_I100} visualizes the overall execution times in seconds and the average execution times per iteration for both algorithms.
The speed-up ratios for the overall decomposition between dGN and fLM \add{were} approximately 6.4, 14.6, 35.1, 16.7, 7.8 and 2.8  times for $R$ = 5, 10, 20, 30, 40, 60 respectively\add{,
while} the speed-up ratios per iteration \add{were} respectively 5.6, 14.7, 20,7, 11.3, 6.5 and 2.7.
%% Speed Ratio dGN/fLM
%   6.3611    5.6050
%   14.5715   14.7229
%   35.0695   20.7281
%   16.6962   11.2691
%    7.7920    6.5313
%    2.7956    3.1028
We note that the numbers of iterations of dGN and fLM \add{were} slightly different because of differences between them in controlling the damping parameters.
%Hence, 
%Note that the number of iterations affects the total execution time of the algorithms.

In Fig.~\ref{fig_dGNvsfLM_3D_SAE_I100}, we illustrate the average MedSAE values of dGN \cite{INDAFAC} and fLM. 
 The mean MedSAEs for the first components $\ba^{(n)}_1$, $n = 1, \ldots, N$ were  calculated over $N$  MedSAE($\alpha^{(n) 2}_{1}$); whereas the mean MedSAEs for the other components $\ba^{(n)}_r$, $r = 2, 3, \ldots, R$, $n = 1, \ldots, N$ were calculated over $(N \times (R-1))$ MedSAE($\alpha^{(n) 2}_{r\ge 2}$).
 Fig~\ref{fig_dGNvsfLM_3D_SAE_I100} shows that the average values of MedSAE($\alpha^{(n) 2}_{r}$), $r \ge 2$, $\forall n$, asymptotically \add{attained} the CRIB.
It means that both dGN and fLM well reconstructed components $\ba^{(n)}_r$, $r = 2, \ldots, R$, $\forall n$ even for $R = 60$.
To be accurate, CRIB is a theoretical lower bound on the mean {of the} square angular error, not on {the} median. In these simulations, the \add{median and mean SAEs} appeared to be nearly identical so that only the former one is shown.

For the first components $\ba^{(n)}_1$, performances of dGN and fLM \add{were} equivalent in the sense of collinearity reconstruction for small $R  = 5, 10$.
For $R = 20, 30$, fLM still reconstructed the first components.
%The results indicate that for $R = 20$ and 30, MedSAEs of fLM were much better than those of dGN.
Note that although MedSAEs \add{were} different, the relative approximation errors $\varepsilon$ of two algorithms were almost the same but they were not presented here.
%For example, $\varepsilon = 3.16e-2, 3.15e-2, 3.15e-2$ for $R = 5, 20, 30$, respectively.
The difference in component reconstruction was caused by implementation of the control strategy for damping parameter.
For $R \ge 40$, the average MedSAEs of the two algorithms were much worse than the CRIB,
and they \add{were not able to }reconstruct the first components.
Indeed, we cannot recover the first components due to noise for high $R$.

 \begin{figure}[t!]
\centering
\psfrag{nu}[t][t]{\scalebox{.8}{\color[rgb]{0,0,0}\setlength{\tabcolsep}{0pt}\begin{tabular}{c}$\nu$\end{tabular}}}%
\psfrag{MSAE (dB)}[b][b]{\scalebox{.7}{\color[rgb]{0,0,0}\setlength{\tabcolsep}{0pt}\begin{tabular}{c}MSAE (dB)\end{tabular}}}%
\psfrag{SAE (dB)}[b][b]{\scalebox{.7}{\color[rgb]{0,0,0}\setlength{\tabcolsep}{0pt}\begin{tabular}{c}Square Angular Error (dB)\end{tabular}}}%
\psfrag{CRLB}[tl][tl]{\scalebox{.6}{\color[rgb]{0,0,0}\footnotesize CRIB}}%
\psfrag{ALS}[tl][tl]{\scalebox{.6}{\color[rgb]{0,0,0}\footnotesize ALS}}%
\psfrag{LS}[tl][tl]{\scalebox{.6}{\color[rgb]{0,0,0}\footnotesize LS}}%
\psfrag{OPT}[tl][tl]{\scalebox{.6}{\color[rgb]{0,0,0}\footnotesize OPT}}%
\psfrag{dGN}[tl][tl]{\scalebox{.6}{\color[rgb]{0,0,0}\footnotesize dGN}}%
\psfrag{fLMa}[tl][tl]{\scalebox{.6}{\color[rgb]{0,0,0}\footnotesize fLM}}%
\psfrag{MedSAE(an1)}[bl][bl]{\scalebox{.6}{\color[rgb]{0,0,0}\small MedSAE($\alpha^{(n)2}_1$)}}%
\psfrag{MedSAE(anr)}[bl][bl]{\scalebox{.6}{\color[rgb]{0,0,0}\small MedSAE($\alpha^{(n)2}_{r \ge 2}$)}}%
\psfrag{R}[bc][bc]{\scalebox{1}{\color[rgb]{0,0,0}\setlength{\tabcolsep}{0pt}\begin{tabular}{c}\vspace{-2em} \footnotesize $R$\end{tabular}}}%
\psfrag{Running time (secs)}[b][b]{\scalebox{1}{\color[rgb]{0,0,0}\footnotesize Execution time (secs)}}%

%\psfrag{CRLB}[t][t]{\scalebox{.6}{\color[rgb]{0,0,0}\setlength{\tabcolsep}{0pt}\begin{tabular}{c}\footnotesize CRIB\end{tabular}}}%
%
\subfigure[Overall execution time and average execution time per iteration.]{
\includegraphics[width=.51\linewidth, trim = 0.0cm -.5cm 0cm 0.cm,clip=true]{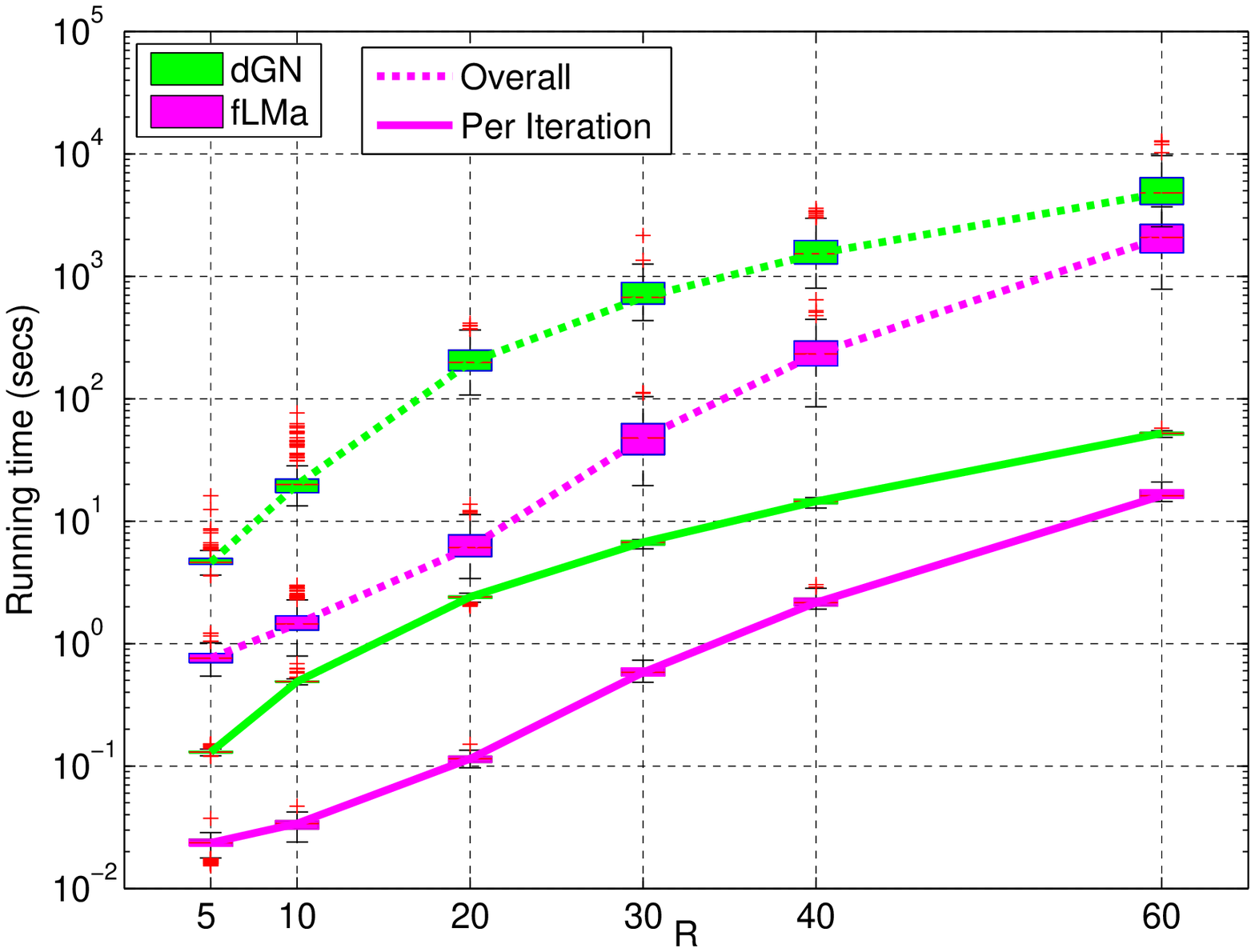}\label{fig_dGNvsfLM_3D_rtime_I100}
}
\hfill
\psfrag{R}[bc][bc]{\scalebox{1}{\color[rgb]{0,0,0}\setlength{\tabcolsep}{0pt}\begin{tabular}{c}\vspace{-1em} \footnotesize $R$\end{tabular}}}%
\subfigure[MedSAE and CRIB]{
\includegraphics[width=.45\linewidth, height =.40\linewidth, trim = 0.0cm 0cm 0cm 0.0cm,clip=true]{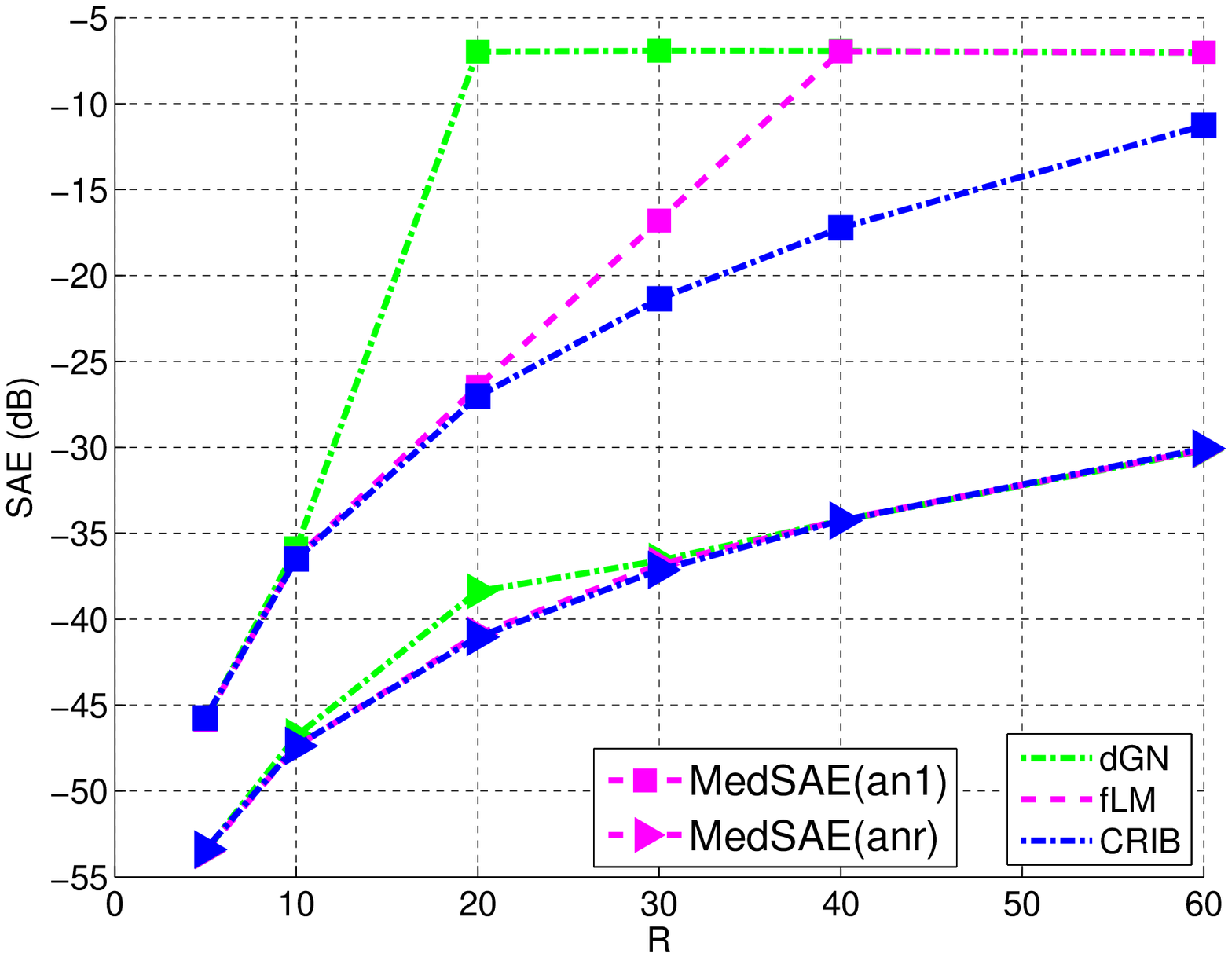}\label{fig_dGNvsfLM_3D_SAE_I100}
}
\vspace{-1.5ex}
\caption{Comparison between the dGN (green lines) and fLM (magenta lines) algorithms for factorization of $100 \times 100 \times 100$ dimensional tensors composed by collinear factors for various $R$ at SNR = 30 dB: \subref{fig_dGNvsfLM_3D_rtime_I100} the overall execution times in second (dashed lines)  and the average execution times per iteration (solid lines); \subref{fig_dGNvsfLM_3D_SAE_I100} the average MedSAE values (dB) of the first components $\ba^{(n)}_1$ (square marker) and of other components $\ba^{(n)}_r$ (triangular marker), $r = 2, \ldots, R$, $n = 1, 2, 3$.
}\label{fig_dGN_fLMrtime}
\end{figure} 

\add{In order to analyze complexity of the two algorithms for higher ranks $R \rightarrow I$, we decomposed tensors of the same dimensions whose entries were randomly generated. The rank $R$ varied from 5 to $I = 100$. The amount of allocated memory and average execution time per iteration were measured on the computer (PC1) in the previous simulations and on a computer (PC2) which had 2.67 GHz i7 CPU and 4 GB of memory.
The results were summarized in Fig.~}\ref{fig_dGNvsfLM_3D}.
\add{
For high rank $R \ge 50$, dGN required more than 4 GB of memory and could consume 20 GB of memory for $R = 100$
whereas fLM need less than 4 GB of memory.
On PC1 which had 64 GB of memory, fLM was slightly more time consuming for $R \ge 90$ than dGN because the advantage of the fast inversion in} (\ref{equ_inverse_f2a}) \add{was lost.
However, dGN became dramatically time consuming on PC2 when $R \ge 40$.}

\begin{figure}
%\psfrag{nu}[t][t]{\scalebox{.8}{\color[rgb]{0,0,0}\setlength{\tabcolsep}{0pt}\begin{tabular}{c}$\nu$\end{tabular}}}%
%
%
\subfigure[Allocated memory.]{
\includegraphics[width=.48\linewidth, trim = 0.0cm 0cm 0cm 0cm,clip=false]{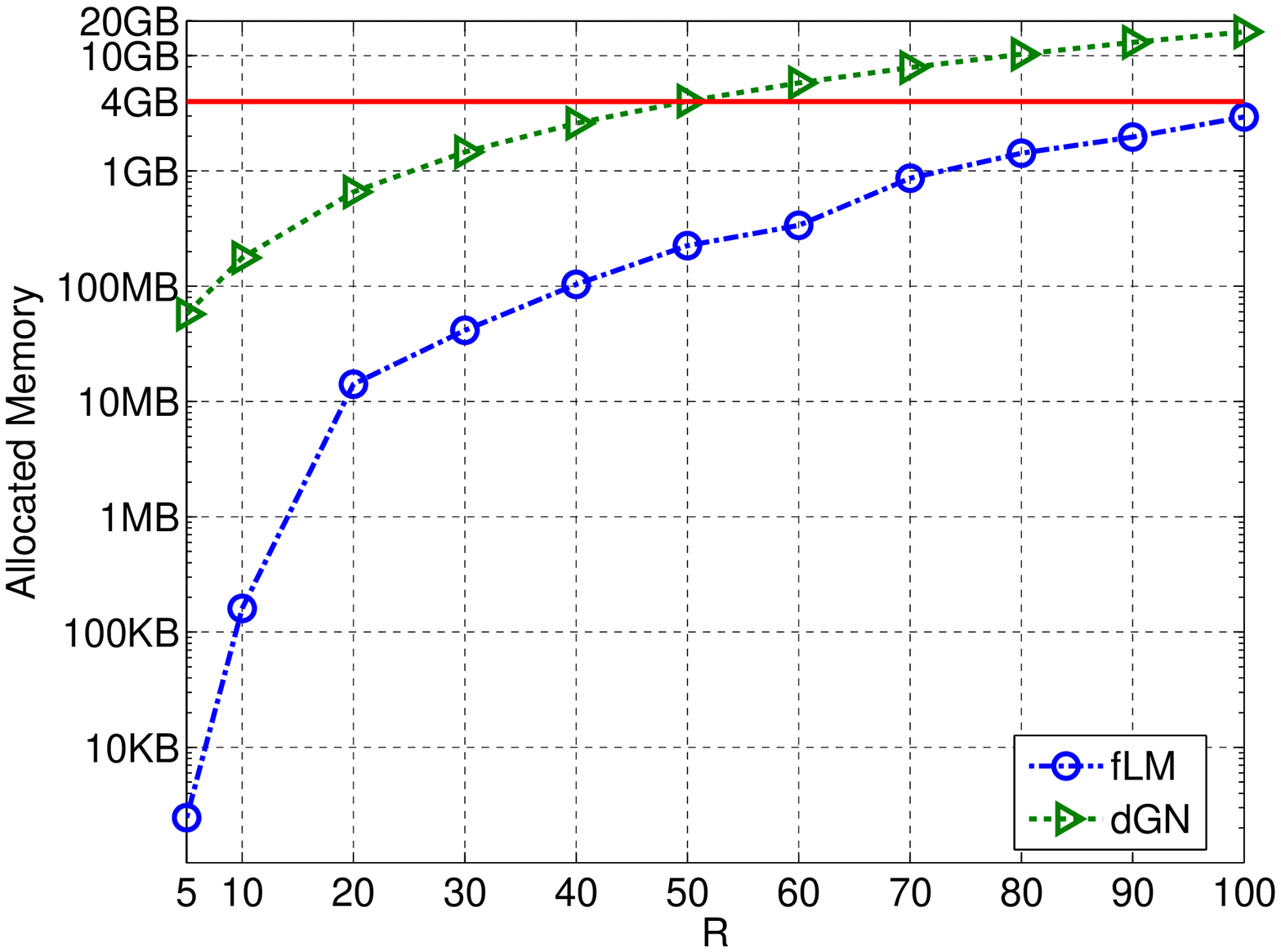}\label{fig_fLM_vs_dGN_allocmem_I100}
}
\hfill
\subfigure[Execution time per iteration.]{
\includegraphics[width=.48\linewidth, trim = 0.0cm 0.0cm 0cm 0cm,clip=false]{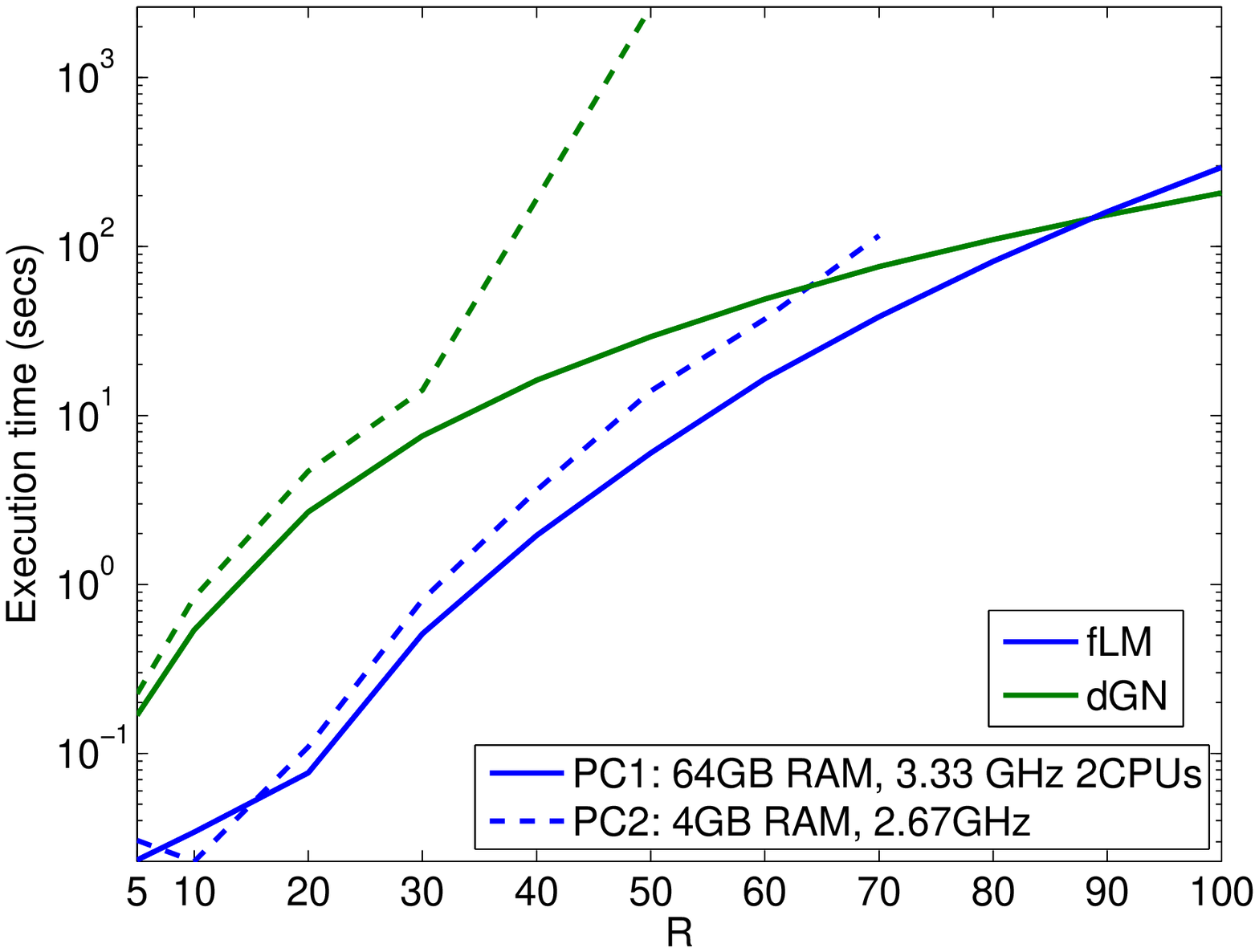}\label{fig_fLM_vs_dGN_exectime_I100}
}
\caption{Memory requirements and execution time per iteration of dGN and fLM in approximation of $100 \times 100 \times 100$ dimensional tensors by rank-$R$ tensors where $R = 5, 10, 20, \ldots, 100$.}
\label{fig_dGNvsfLM_3D}
\end{figure}

\subsection{Factorization of higher-order real-valued tensors}\label{sec::exp_real_tensor}

The proposed algorithms have been extensively verified and compared with the ALS algorithm plus line seach in the N-way toolbox \cite{Nwaytoolbox}, %and the OPT algorithm \cite{JChem-CPOPT,tensortoolbox24} 
for 4-D tensors of size $I_n = 50$, various ranks $R = 5, 10, 15$, and \add{with} different collinearity degree $\nu = 0.1, 0.3, 0.5, 0.7, 0.9$.
The 4-D tensors were corrupted by additive Gaussian noise at SNR = 40 dB.
For each pair ($\nu$, $R$) MedSAE \add{was} computed from 400 runs. Execution times (seconds) were measured on a computer that had  6-core i7 3.33 GHz processor and 24 GB memory. 

Algorithms were analyzed under the same experimental conditions as in the previous simulations. 
%Those are leading singular values for initialization, 
They iterated until successive relative errors $ \varepsilon$ were lower than $10^{-12}$,
or the maximum number of iterations (5000)  \add{was achieved}.
The ALS algorithm plus line seach (ALSls) was adapted to have the same stopping criteria.%\cite{Badertensor,tensortoolbox24}

  \begin{figure}[t!]
\centering
\psfrag{nu}[t][t]{\scalebox{.8}{\color[rgb]{0,0,0}\setlength{\tabcolsep}{0pt}\begin{tabular}{c}$\nu$\end{tabular}}}%
\psfrag{R}[t][t]{\scalebox{.8}{\color[rgb]{0,0,0}\setlength{\tabcolsep}{0pt}\begin{tabular}{c}$R$\end{tabular}}}%
\psfrag{MSAE (dB)}[b][b]{\scalebox{.7}{\color[rgb]{0,0,0}\setlength{\tabcolsep}{0pt}\begin{tabular}{c}MedSAE (dB)\end{tabular}}}%
\psfrag{CRLB}[tl][tl]{\scalebox{.6}{\color[rgb]{0,0,0}\footnotesize CRIB}}%
\psfrag{ALSls}[tl][tl]{\scalebox{.6}{\color[rgb]{0,0,0}\footnotesize ALSls}}%
\psfrag{LS}[tl][tl]{\scalebox{.6}{\color[rgb]{0,0,0}\footnotesize LS}}%
\psfrag{OPT}[tl][tl]{\scalebox{.6}{\color[rgb]{0,0,0}\footnotesize OPT}}%
\psfrag{fLMa}[tl][tl]{\scalebox{.6}{\color[rgb]{0,0,0}\footnotesize fLM}}%
\psfrag{Running time (secs)}[b][b]{\scalebox{1}{\color[rgb]{0,0,0}\footnotesize Execution time (secs)}}%
%\psfrag{CRLB}[t][t]{\scalebox{.6}{\color[rgb]{0,0,0}\setlength{\tabcolsep}{0pt}\begin{tabular}{c}\footnotesize CRIB\end{tabular}}}%
%
\vspace{-1.5ex}
%\hfill
\subfigure[4-D tensors, $\bA^{(n)} \in \Real^{50 \times 5}$, SNR = 40 dB.]{\includegraphics[width=.48\linewidth, trim = 0.0cm 0cm 0cm 0.0cm,clip=true]{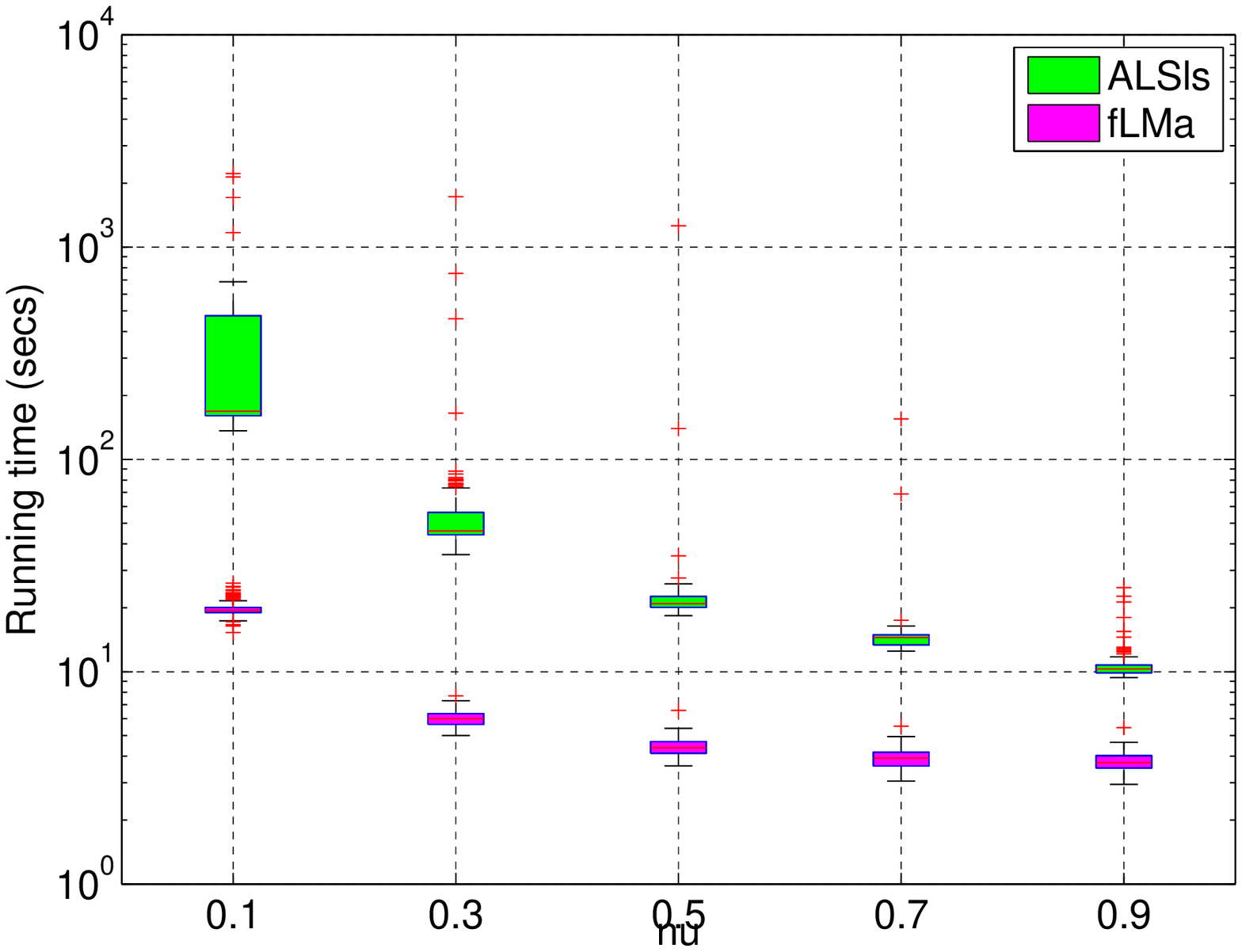}\label{fig_4D_rtimevsnu_I50R5}}
\hfill
\subfigure[4-D tensors, $\bA^{(n)} \in \Real^{50 \times 10}$, SNR = 40 dB.]{\includegraphics[width=.48\linewidth, trim = 0.0cm 0cm 0cm 0.0cm,clip=true]{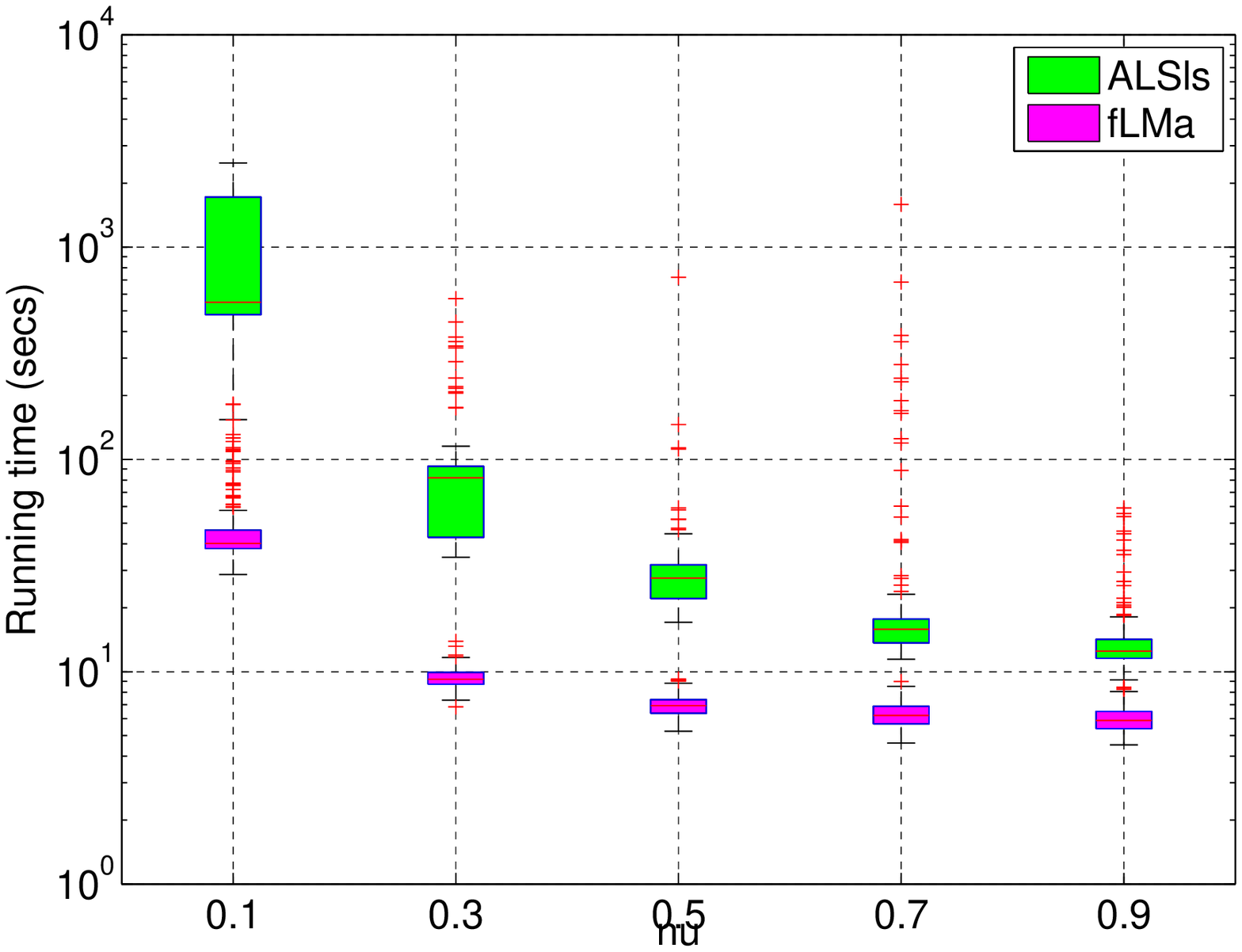}\label{fig_4D_rtimevsnu_I50R10}}
\hfill
\subfigure[4-D tensors, $\bA^{(n)} \in \Real^{50 \times 15}$, SNR = $40$ dB.]{
\includegraphics[width=.48\linewidth, trim = 0.0cm 0cm 0cm 0.0cm,clip=true]{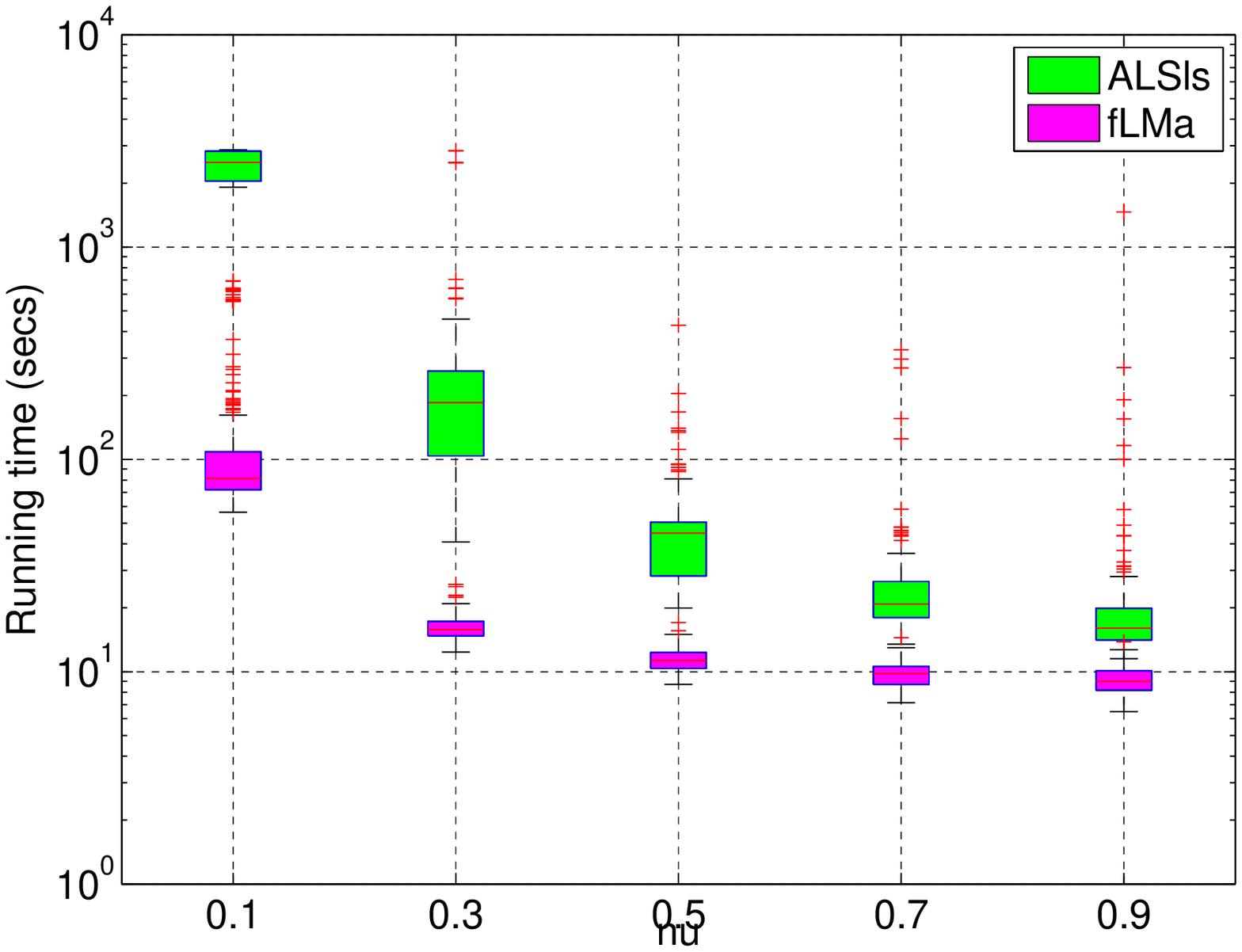}
\label{fig_4D_rtimevsnu_I50R15}
}
\hfill
\subfigure[4-D tensors, $R = 5, 10 , 15$, SNR = 40 dB.]{
\includegraphics[width=.48\linewidth, trim = 0.0cm 0cm 0cm 0.0cm,clip=true]{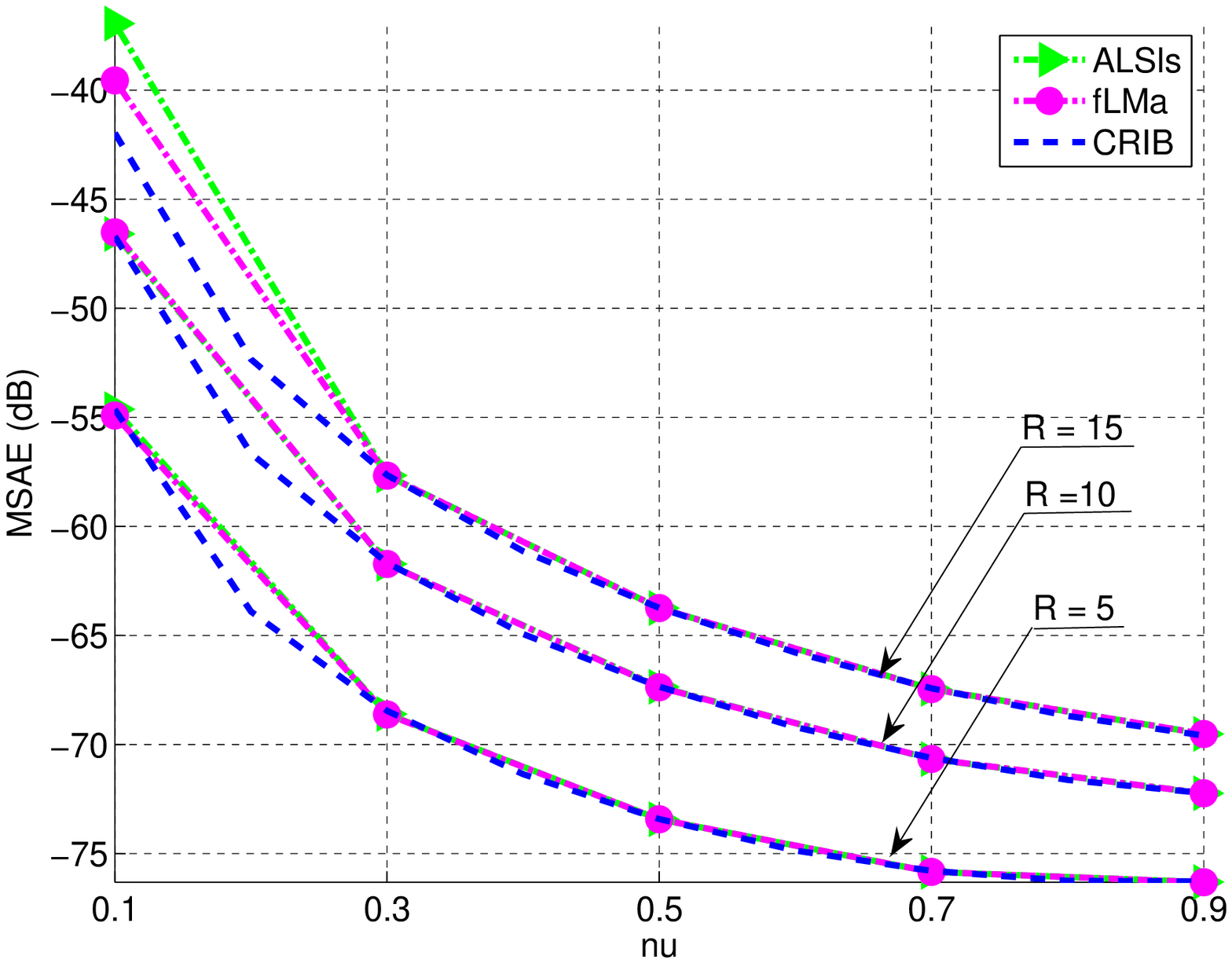}
\label{fig_4D_MSAE_I50Rall}
}
\vspace{-2ex}
\caption{Comparison between ALSls and fLM for factorizations of 4-D tensors of size $50 \times 50 \times 50 \times 50$ at SNR = 40 dB.
\subref{fig_4D_rtimevsnu_I50R5}-\subref{fig_4D_rtimevsnu_I50R15} 
execution times (seconds) {were} measured when algorithms factorized tensors of various ranks $R = 5, 10, 15$.
%Algorithms ran until \add{they reached} a derivative of successive relative errors of $10^{-12}$ or 5000 iterations.
\subref{fig_4D_MSAE_I50Rall} the average MedSAE (dB) for all components compared with CRIB.
}\label{fig_MSAE_4D}
\end{figure}%\vspace{-1.5ex}

At SNR = 40 dB and ranks $R = 5, 10, 15$, CRIBs are relatively high ($> 40$ dB) for most $\nu$ (see Fig.~\ref{fig_4D_MSAE_I50Rall}). Hence, CPD algorithms easily \add{estimated} collinear factors and \add{obtained} high MedSAE comparable to \add{the} CRIB. 
Fig.~\ref{fig_4D_MSAE_I50Rall} shows that MedSAEs of ALSls and fLM \add{were} almost similar and \add{approached} CRIB except those for $R = 15$ and $\nu = 0.1$.
It should be noted that factorization became more difficult \add{in the case of} higher rank $R$ and lower $\nu$.
Execution times of algorithms for different $R$ and $\nu$ are illustrated in Figs.~\ref{fig_4D_rtimevsnu_I50R5}-\ref{fig_4D_rtimevsnu_I50R15}.
The results indicate that the higher the collinearity degree (\add{i.e.,} smaller $\nu$) the more time-consuming the algorithms. 
For example, ALSls \add{on} average ran 2083 iterations in 957 seconds to factorize 4-D noisy tensors \add{when} $R = 10$ and $\nu = 0.1$.
However, \add{when keeping the tensor size and rank $R$ and changing $\nu = 0.9$}, this algorithm ran 34 iterations in 14 seconds.
For the same tensors with $\nu = 0.1$, fLM took only 48.6 seconds \add{on} average to execute 384 iterations, and took 6 seconds for 21 iterations \add{with} $\nu = 0.9$.
That means fLM \add{was} 21 times faster than ALS \add{with} $\nu = 0.1$.
For 4-D tensors of $R = 15$ and \add{with} $\nu = 0.1$, ALSls ran 4225 iterations in 2255 seconds \add{on} average, while fLM took only 103 seconds to execute 494 iterations. Hence, fLM was 24.7 times faster than ALSls for the difficult test case.
More execution times and speed ratios are given in Table~\ref{tab_fLMvsALS}.
Speed ratio between ALSls and fLM \add{was} high for highly collinear data (e.g., $\nu = 0,1$).
For example, fLM was at least 17.1 times and up to 24.8 times faster than ALSls for collinear data with $\nu = 0.1$. 
For lower collinearity degree, ALSls \add{quickly factorized} the tensor after few iterations. Although the speed ratio decreased, fLM \add{was} still approximately 3 times faster than ALSls.

\begin{table}
\caption{Comparison of average execution times (seconds) between fLM and ALSls for factorizations of 4-D and 5-D tensors of size $I_n = 50$ at SNR = 40 dB composed by collinear factors with various $\nu = 0.1, 0.3, 0.5, 0.9$ and for various $R$. For each pair ($N, I_m, R, \nu$), speed-up ratio and execution times are given as indicated in the subtable at the bottom.}
\small
\begin{center}
\begin{tabular*}{1\linewidth}{l>{\flushright}p{2em}>{\flushright\columncolor[gray]{.85}}b{2.5em} >{\flushright}p{1.9em}>{\flushright\columncolor[gray]{.85}}b{2.5em} >{\flushright}p{1.9em}>{\flushright\columncolor[gray]{.85}}b{2.5em}>{\flushright}p{1.9em}>{\flushright\columncolor[gray]{.85}}b{2em} >{\flushright}p{1.9em}>{\flushright\columncolor[gray]{.85}}b{2em} ccc}
\multicolumn{1}{c}{Tensor's size}& \multicolumn{10}{c}{Collinear degree $\nu$}  \\\cline{2-11}
($N$-D, $I_m \times R$) & \multicolumn{2}{c}{0.1}  & \multicolumn{2}{c}{0.3}& \multicolumn{2}{c}{0.5}& \multicolumn{2}{c}{0.7}& \multicolumn{2}{c}{0.9}&\\
\cline{1-11}
\multirow{2}{*}{4-D, $50\times 5$}
   &  & 347  &  & 65  & & 28   &   & 15  &   & 11 &  \\
&\multirow{-2}{*}{17.1}&20 &\multirow{-2}{*}{11.1}  & 6 &  \multirow{-2}{*}{6}  &4.4  & \multirow{-2}{*}{3.9} & 3.9 & \multirow{-2}{*}{2.8} & 3.8 &\\\cline{1-11}
\multirow{2}{*}{4-D, $50\times 10$}  &  & 957  && 90  &   & 34    & & 40   & & 11 &\\%\cline{3-3}\cline{5-5}\cline{7-7}\cline{9-9}\cline{11-11}
& \multirow{-2}{*}{21.2} & 49 &   \multirow{-2}{*}{9.6} & 9 &\multirow{-2}{*}{4.9} &7  & \multirow{-2}{*}{6}  & 6 &  \multirow{-2}{*}{2.5} & 6 &\\\cline{1-11}
\multirow{2}{*}{4-D, $50\times 15$} &   & 2,201  &  & 263  & & 48    & & 29   & & 29& \\%\cline{3-3}\cline{5-5}\cline{7-7}\cline{9-9}\cline{11-11}
&\multirow{-2}{*}{24.8} & 99 & \multirow{-2}{*}{15.4} & 16 & \multirow{-2}{*}{4.2}  &11  &  \multirow{-2}{*}{3} & 10 & \multirow{-2}{*}{2.9}  & 9 &\\\cline{1-11}
\multirow{2}{*}{5-D, $50\times 5$} &  & 17,245  &  & 2,747  & & 1,240   && 821  &  & 730 &  \\%\cline{3-3}\cline{5-5}\cline{7-7}\cline{9-9}\cline{11-11}
& \multirow{-2}{*}{22}  & 790 & \multirow{-2}{*}{8.1}  & 346 &  \multirow{-2}{*}{4.6}  &453  &  \multirow{-2}{*}{4.2}  & 205 &\multirow{-2}{*}{3.4}  & 251 &\\
\cline{1-11}
\end{tabular*}
\label{tab_fLMvsALS}\\[1ex]
\footnotesize
\begin{tabular*}{1\linewidth}
{|l>{\columncolor[gray]{.85}}l|}\cline{1-2}
\multirow{2}{1.5em}{ratio}& Execution time$_{\text{ALSls}}$ (seconds) \\ \cline{2-2}
&Execution time$_{\text{fLM}}$  (seconds) \\\cline{1-2}
\end{tabular*}
\end{center}
\end{table}

\subsection{Factorization of complex-valued tensors}\label{sec::exp_cplex_tensor}
In the next set of simulations, we considered factorization of complex-valued tensors.
Factors $\bA^{(n)} \in \Complex^{70 \times R}$ were generated in the same manner as for experiments in the previous section. However, they had random real and imaginary parts.
In addition to collinearity degrees $\nu = 0.1, 0.2,\ldots, 0.5$, we considered 
%$\nu = 0.05$ for mutual angles $\theta_{1,r}  \approx 3^{\text o}, r \neq 1$,  and
 $\nu = 3, 4, 5$. 
We note that although collinearity of factors is low 
for high $\nu = 3, 4, 5$ ($\theta_{1,r} > 71^{\text o}$), the tensors are still difficult to factorize. 
 
%Normally, the CP algorithms can straightforwardly extended to handle complex-valued tensors.
We compared fLM with ALS plus line search (ALSls). 
\add{Algorithms stopped when differences between successive relative errors were lower than $10^{-8}$, or the maximum number of iterations (2000) was achieved}.
In Figs.~\ref{fig_70x5_4_cplx}-\subref{fig_70x15_4_cplx}, we illustrate the average MedSAE of all factors for $70 \times 70 \times 70 \times 70$ dimensional tensors with ranks $R = 5$ and 15 over 200 runs.
ALSls achieved good performance \add{with} $\nu = 0.2$, and excellent MedSAE \add{with} $\nu$ = 0.3, 0.4 and 0.5.
However, for high collinearity degree $\nu$ = 4 and 5, ALSls did not obtain perfect reconstruction.
The fLM algorithm outperformed ALSls for all test cases.
Figs.~\ref{fig_70x5_4_NoIter_cplx}-\subref{fig_70x15_4_NoIter_cplx} indicate that the number of iterations of ALSls \add{tended} to decrease gradually as $\nu$ \add{increased} from 0.1 to 5.
%However, ALS and ALSls still need at least 1000 iterations in order to successfully factorize 3-D tensors of rank $R = 5$.
%For $\nu$ = 3, 4 and 5, the number of iterations of ALS and ALSls increases again, and quickly passes over the maximum value of 2000 while they still get stuck in local minima.
For $\nu = 3, 4, 5$, ALSls stopped after tens of iterations because there \add{was} not any significant change in the relative error.
Figs.~\ref{fig_70x5_4_NoIter_cplx}-\subref{fig_70x15_4_NoIter_cplx} also reveal that fLM \add{required} less iterations for higher $\nu$.
Difference in magnitude between components \add{did} not affect fLM.

 \begin{figure}[t!]
\centering
\psfrag{nu}[t][t]{\scalebox{.8}{\color[rgb]{0,0,0}\setlength{\tabcolsep}{0pt}\begin{tabular}{c}$\nu$\end{tabular}}}%
\psfrag{MSAE (dB)}[b][b]{\scalebox{.7}{\color[rgb]{0,0,0}\setlength{\tabcolsep}{0pt}\begin{tabular}{c}MSAE (dB)\end{tabular}}}%
\psfrag{CRLB}[tl][tl]{\scalebox{.6}{\color[rgb]{0,0,0}\footnotesize CRIB}}%
\psfrag{ALS}[tl][tl]{\scalebox{.6}{\color[rgb]{0,0,0}\footnotesize ALS}}%
\psfrag{LS}[tl][tl]{\scalebox{.6}{\color[rgb]{0,0,0}\footnotesize LS}}%
\psfrag{OPT}[tl][tl]{\scalebox{.6}{\color[rgb]{0,0,0}\footnotesize OPT}}%
\psfrag{fLMa}[tl][tl]{\scalebox{.6}{\color[rgb]{0,0,0}\footnotesize fLM}}%
\psfrag{No. Iterations}[b][b]{\scalebox{.7}{\color[rgb]{0,0,0}\setlength{\tabcolsep}{0pt}\begin{tabular}{c}No. Iterations\end{tabular}}}%

%\psfrag{CRLB}[t][t]{\scalebox{.6}{\color[rgb]{0,0,0}\setlength{\tabcolsep}{0pt}\begin{tabular}{c}\footnotesize CRIB\end{tabular}}}%
%
\vspace{-1.5ex}
\subfigure[4-D tensors, $\bA^{(n)} \in \Complex^{70 \times 5}$, SNR = $+ \infty$ dB.]{
\includegraphics[width=.48\linewidth, trim = 0.0cm 0cm 0cm 0.8cm,clip=true]{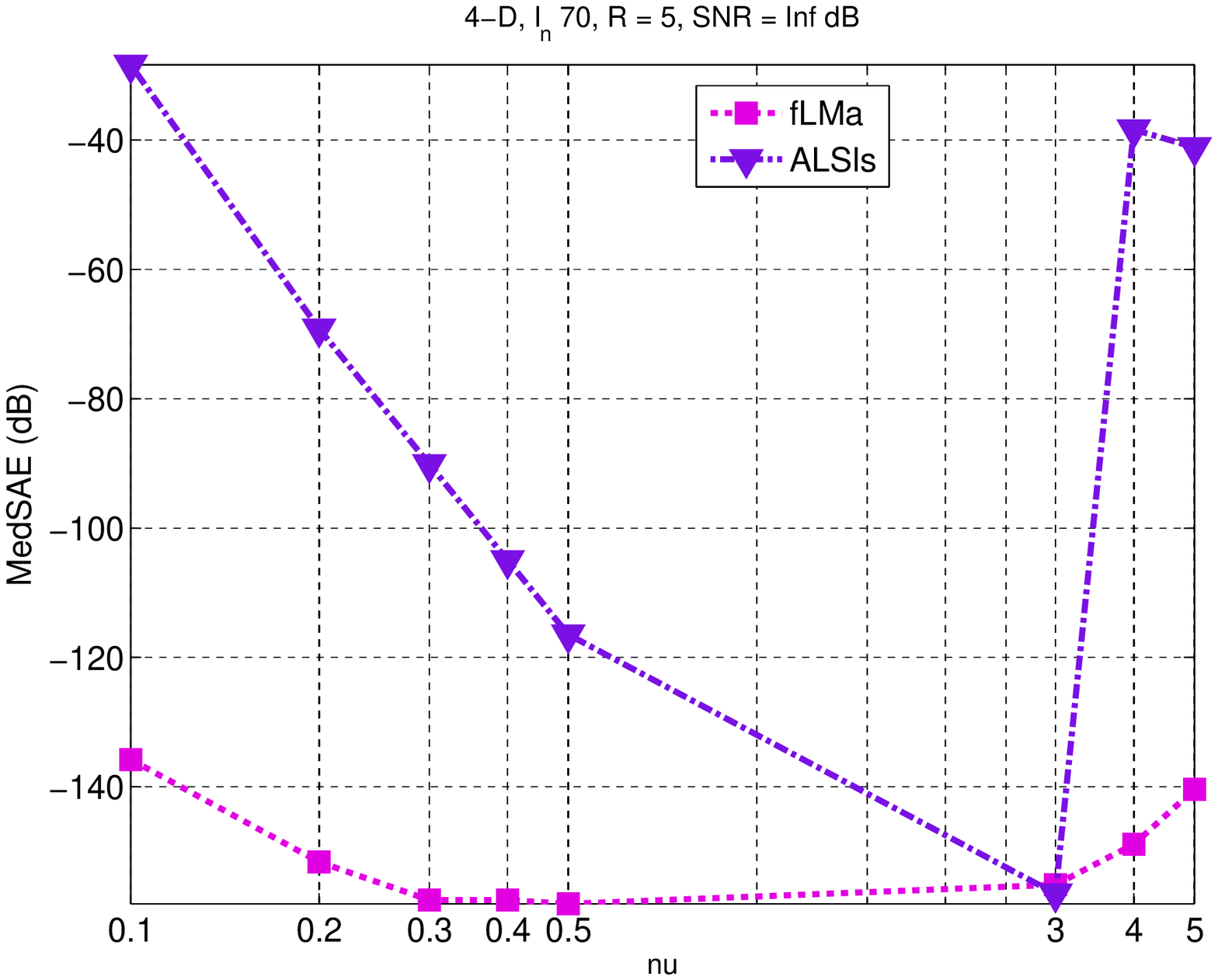}\label{fig_70x5_4_cplx}
}
%\hfill
%\subfigure[3-D tensors, $\bA^{(n)} \in \Complex^{50 \times 10}$, SNR = $+ \infty$ dB.]{
%\includegraphics[width=.48\linewidth, trim = 0.0cm 0cm 0cm 0.8cm,clip=true]{fig_cplxtyp1_N3_I50_R10_SNRInf_MSAEmed}\label{fig_50x10_3_cplx}
%}
\hfill
\subfigure[4-D tensors, $\bA^{(n)} \in \Complex^{70 \times 15}$, SNR = $+ \infty$ dB.]{
\includegraphics[width=.48\linewidth, trim = 0.0cm 0cm 0cm 0.8cm,clip=true]{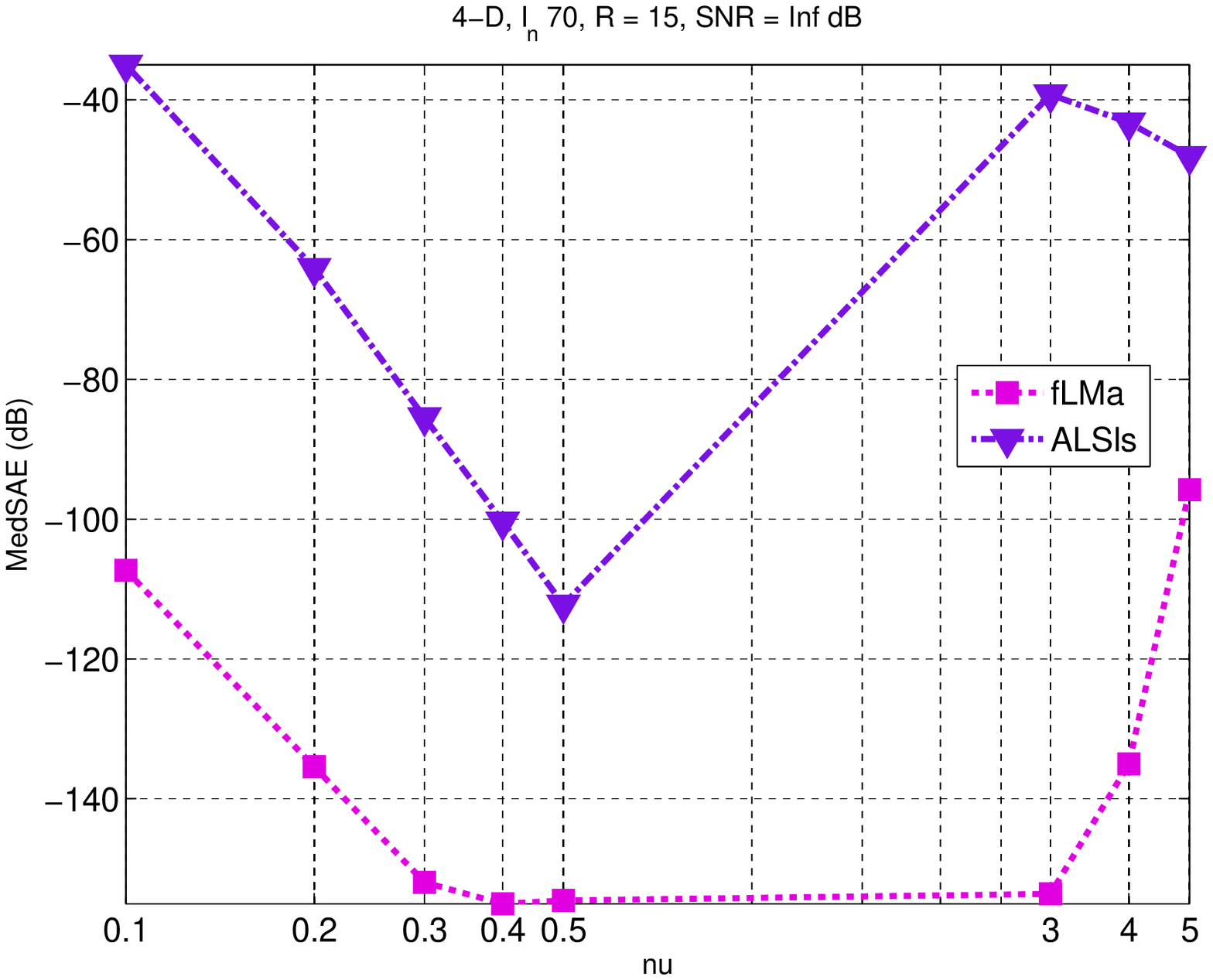}\label{fig_70x15_4_cplx}
}
\vspace{-1.5ex}
\subfigure[4-D  tensors, $\bA^{(n)} \in \Complex^{70 \times 5}$.]{\includegraphics[width=.48\linewidth, trim = 0.0cm 0cm 0cm 0cm,clip=true]{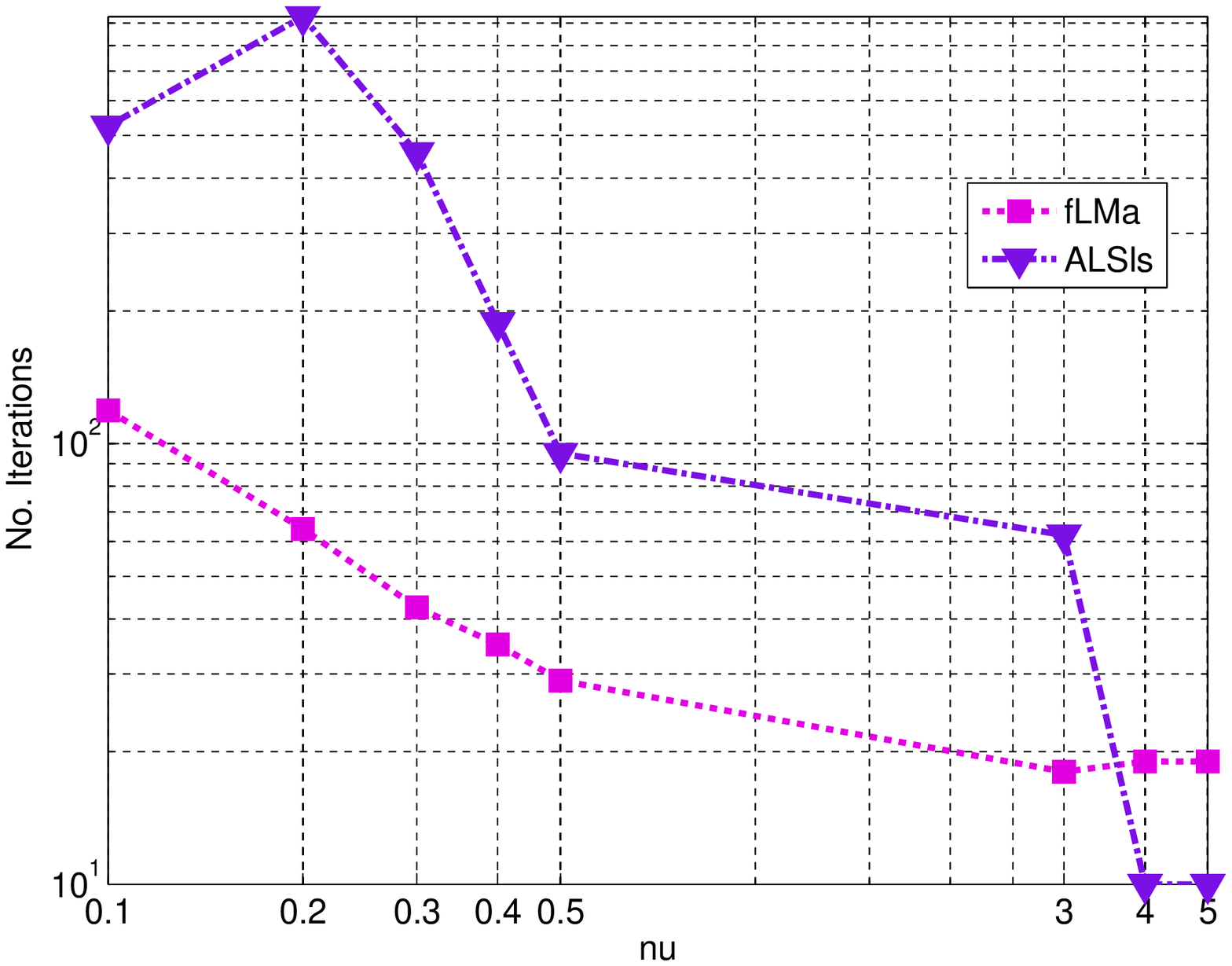}
\label{fig_70x5_4_NoIter_cplx}}
\hfill
\subfigure[4-D  tensors, $\bA^{(n)} \in \Complex^{70 \times 15}$.]{\includegraphics[width=.48\linewidth, trim = 0.0cm 0cm 0cm 0cm,clip=true]{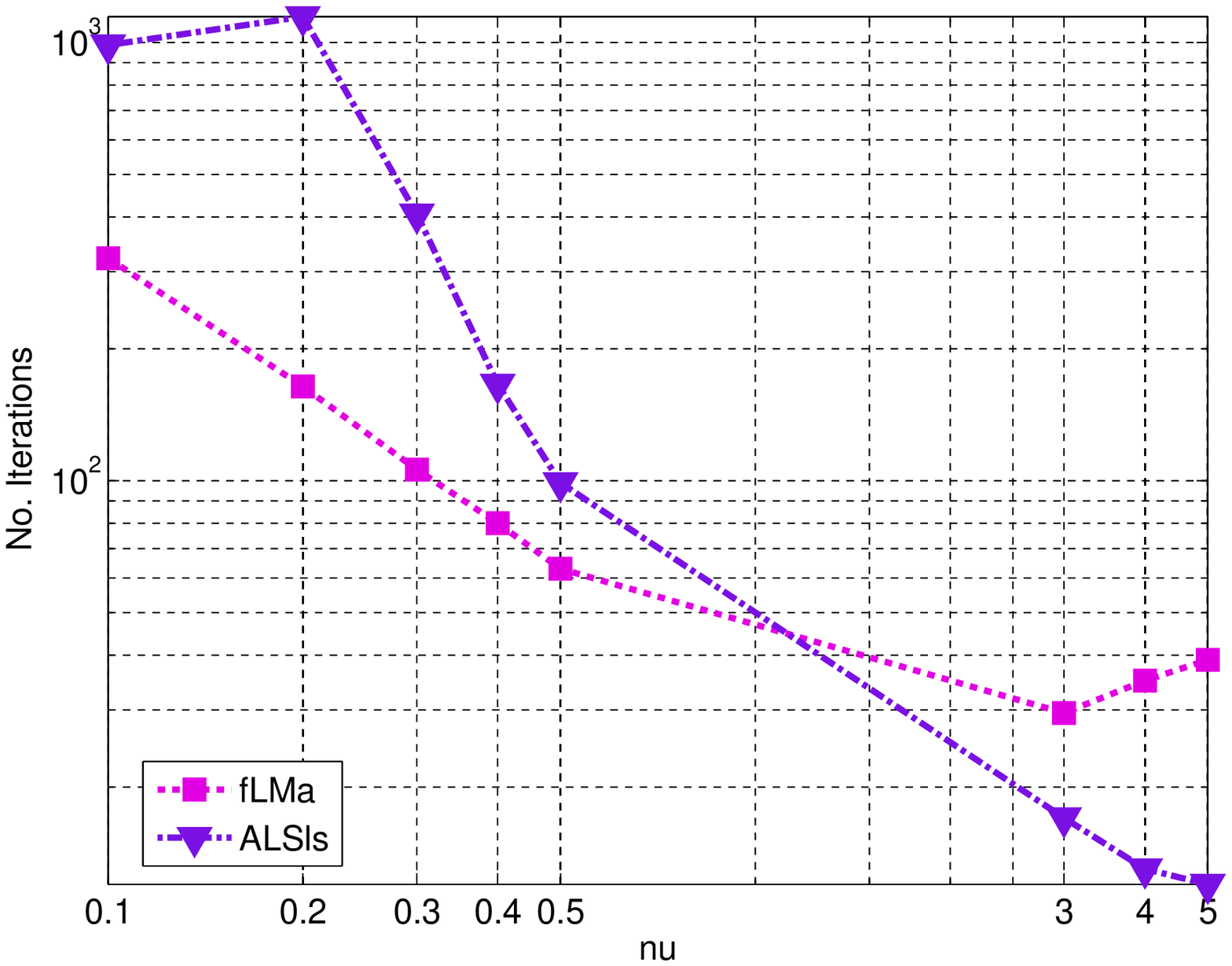}
\label{fig_70x15_4_NoIter_cplx}}
\hfill
\vspace{-1ex}
\caption{Illustration for MSAE for factorization of 4-D complex-valued tensors with size $I_n = 70$ and ranks $R = 5, 15$.
Algorithms stopped as they reached a derivative of successive relative errors of $10^{-8}$ or 2000 iterations. 
}\label{fig_MSAE_70xX_cplx}
\end{figure}

\section{Conclusions} \label{sec:conclusion}

Simulations for real- and complex-valued tensors \add{confirmed} the fLM algorithm \add{was} faster than dGN and ALS, and \add{outperformed} ALS in the sense of  approximation accuracy (MedSAE) for difficult test cases.
Moreover, MedSAE of fLM \add{was} comparable to CRIB for most test cases even for noisy tensors.
For the collinearity modification used in the simulations, we also show that for the same tensor size and collinearity degree, the higher rank $R$ the data tensor has, the more difficult the factorization is to retrieve the factor.
For the same size $I_n$, rank $R$, and collinearity degree, the higher the dimensions of the data tensor, the higher the performance of factorization can be achieved.
%Gaussian noise at SNR = 20 dB can destroy the collinearity structure of factors.
%Furthermore, it is almost impossible to retrieve exact collinear factors at $\nu = 0.1, 0.2$ for 3-D noisy tensors with higher rank even at SNR = 30 dB or 40 dB. 
%The simulations also show that although fLM has high computational cost per one iteration, this algorithm is faster than dGN and ALS due to less number of iterations than others for the test data.

 Most CP algorithms incorporated with line-search techniques work well for general data,
 but often fail for highly collinear data with bottlenecks or swamps.
The dGN/LM algorithms \cite{Paatero97,Tomassithesis} can deal with such data, but 
demand extreme computational cost associated with large-scale inverse of approximate Hessians. 
In this paper, by employing the special structure of the approximate Hessian, a fast inverse for the approximate Hessian has been derived, and low complexity damped Gauss Newton algorithms have been proposed for factorization of low rank real- and complex-valued tensors.
The proposed algorithm avoids building up the whole approximate Hessian and \add{its} inverse \add{by working with} much smaller matrices of size $N R^2 \times N R^2$ instead of $\left(\displaystyle R T \times R T\right)$.
Extensive experiments for tensor factorizations showed that our algorithms  outperformed ``state-of-the-art'' algorithms for difficult benchmarks for both real and complex-valued tensors.
The proposed dGN/LM algorithms can be extended to the nonnegative CPD in which factors are nonnegative matrices. Moreover, our algorithms can be simplified to estimate only one factor for supersymmetric tensor factorization which can be found in multiway clustering, or to the INDSCAL decomposition \cite{ZbynekSSP11,NMF-book}.

\section*{Acknowledgments}
%The authors are indebted to the two referees whose suggestions and corrections led to major improvements over the two  of the manuscript.
The authors wish to thank the referees for the very constructive and detailed comments and suggestions which led to major improvements in the manuscript.
They also thank for Dr. Benedikt L\"osch and Mr. Austin Brockmeier for their suggestions that helped improving the manuscript.

\appendix 
\section{Commutation Matrices} \label{sec::app_Permute}
%An $\left(IJ \times IJ\right)$ permutation matrix  $\bP_{I,J}$  arranges vectorization of an $\left(I\times J\right)$ matrix $\bX$ to  $\vtr{\bX^T}$, that is
% $\vtr{\bX} = \bP_{I,J} \, \vtr{\bX^T}$ 
% %defined given in Theorem 4.3.8 in 
% with $\bP_{J,I} = \bP_{I,J}^T  = \bP_{I,J}^{-1}$  \cite{MR1091716}.
%%
\add{A commutation matrix $\bQ_n$ expresses connection between vectorizations of tensor unfoldings}, and often exists in construction of the Jacobian $\bJ$ and the approximate Hessian $\bH$ in dGN algorithms for CP and Tucker decompositions \cite{Phan_LM_NTD}.

\begin{lemma}{\bf(mode-$n$ to mode-$1$ unfolding)}\label{lemma_Pn}
Commutation matrix \add{$\bQ_n$} which maps  $\vtr{\bA_{(1)}} = \vtr{\tA} = \bQ_n \vtr{\bA_{(n)}}$ is 
given by
%\be
%	\bQ_n =  \bI_{I_{n+1:N}} \otimes \bP_{I_n, I_{1:n-1}}, \qquad {\text{with}} \; I_{i:j} = \prod_{k = i}^{j} I_k. 
%\ee
%
\add{$\bQ_n =  \bI_{I_{n+1:N}} \otimes \bP_{I_{1:n-1},I_n}$},  with $I_{i:j} = \prod_{k = i}^{j} I_k$.
\end{lemma}

\section{Proof of Theorem~\ref{theo_decomH}}\label{sec::proof_theo_decomH}

In order to prove Theorem~\ref{theo_decomH}, \add{we seek explicit expressions for the Jacobian and the approximate Hessian in the next section}.
\begin{lemma}\label{lem_Jacobian}
The Jacobian matrix $\bJ$ has a form of \cite{Paatero97,TomasiBro05}\\[-3ex]
\be
	\bJ &=& \left[ \bQ_n \, \left(\left({\displaystyle\bigodot_{k\neq n} \bA^{(k)}}\right) \otimes  \bI_{I_n}\right)  \right]_{n=1}^{N} .\label{equ_Jacobian}
\ee
\end{lemma} 

We express the approximate Hessian $\bH$ as an $N \times N$ block matrix $\bH  = \left[ \bH^{(n,m)}\right]_{n,m}$, $\bH^{(n,m)}$ of size $RI_n \times RI_m$.

\begin{theorem}{(see also \cite{Paatero97,Tomassithesis})}\label{theo_Hnm_full}
A submatrix $\bH^{(n,m)}$ has an explicit expression given by\\[-3ex]
\be
\bH^{(n,m)}  = \delta_{n,m} \, 
   \left({\bGamma}^{(n)} \otimes \, \bI_{I_n}\right) + 
    \left(  \bI_R \otimes  \bA^{(n)} \right) \, \bK^{(n,m)} \,
    \left( \bI_R \otimes  \bA^{(m) \, T }  \right) , \quad  \forall n, \forall m .
    \label{equ_Hnm_full}
\ee
\end{theorem}\\[-5ex]

By establishing expressions for submatrices $\bH^{(n,m)}$, we can prove Theorem~\ref{theo_decomH}.
\begin{proof}(Theorem \ref{theo_decomH})
From (\ref{equ_Hnm_full}), we construct a sparse matrix $\bG$ consisting all block matrices $\bH^{(n)}$ $n = 1, 2, \ldots, N$,  that is
\be
    \bG &=&  {\blkdiag}\left( \bH^{(n)}
            \right)_{n=1}^{N} = {\blkdiag}\left(  {\boldsymbol\Gamma}^{(n)} \otimes \bI_{I_n}
            \right)_{n=1}^{N}.
\ee

From Theorem \ref{theo_Hnm_full}, and by using the product of block matrices, it is straightforward to decompose $\bH - \bG$ into three matrices defined in Theorem~\ref{theo_decomH} as
\be
    \bH - \bG =  \bZ  \, \bK \,     \bZ^T \, .
\ee
%This completes the proof of Theorem~\ref{theo_decomH}.
\end{proof}

\section{Proof of Theorem~\ref{theo_inverseH}}\label{sec::proof_theo_inverseH}
\begin{proof}
The damped approximate Hessian $\bH_{\mu} = \bG + \mu \bI_{RT} + \bZ \bK \bZ^T$ is adjusted from $\bG_{\mu} = \bG + \mu \bI_{RT}$ by a low-rank matrix $\bZ \bK \bZ^T$. Hence, its inverse can be quickly computed by applying the binomial inverse theorem (see page 18  \cite{0521386322})\\[-4ex]
\be
	\bH_{\mu}^{-1} %&=& \left(\bG_{\mu} + \bZ \, \bK \, \bZ^T  \right)^{-1}  \\
	&=& \begin{cases}
		{\bG}_{\mu}^{-1}  - {\bG}_{\mu}^{-1} \, \bZ  \left(\bK^{-1}  +  \bZ^T \bG_{\mu}^{-1} \bZ  \right)^{-1}  \,  \bZ^T  \, {\bG}_{\mu}^{-1}  , \quad &\text{if $\bK$ is invertible,} \label{equ_Hinverse1}
		\\
{\bG}_{\mu}^{-1}  - {\bG}_{\mu}^{-1} \, \bZ  \, \bK \, \left(\bI_{NR^2}   +  \bZ^T \bG_{\mu}^{-1} \bZ \bK  \right)^{-1}  \,  \bZ^T  \, {\bG}_{\mu}^{-1}  , \quad & \text{otherwise.}\label{equ_Hinverse2}
		\end{cases}
\ee
Denote by $\widetilde{\bG}_{\mu}$ inverse of the block diagonal matrix ${\bG}_{\mu}$ which is also a block diagonal matrix
\be
	\widetilde{\bG}_{\mu} =  \left({\blkdiag}
\left( \left( {\bGamma^{(n)}}+ \mu \, \bI_R\right) \otimes \bI_{I_n} \right)_{n = 1}^N \right)^{-1} = {\blkdiag}
\left( \widetilde{\bGamma}_{\mu}^{(n)} \otimes \bI_{I_n} \right)_{n = 1}^N . \notag
\ee
Similarly, \add{we denote} ${\bL_{\mu}}  = {\bG}_{\mu}^{-1} \, \bZ$ and ${\bPsi_{\mu}}  = \bZ^T \,  {\bG}_{\mu}^{-1} \, \bZ$. \add{From} (\ref{equ_Z}) and by taking into account $\left( \widetilde{\bGamma}_{\mu}^{(n)} \otimes \bI_{I_n} \right) \left(\bI_R \otimes \bA^{(n)}
            \right) = \widetilde{\bGamma}_{\mu}^{(n)}  \otimes \bA^{(n)}$, we have\\[-3ex]
\be
	{\bL_{\mu}}  &=& {\bG}_{\mu}^{-1} \, \bZ = {\blkdiag}
\left( \widetilde{\bGamma}_{\mu}^{(n)} \otimes \bI_{I_n} \right)_{n = 1}^N  \, {\blkdiag}\left(\bI_R \otimes \bA^{(n)}
            \right)_{n=1}^{N} \notag\\
            &=& {\blkdiag}\left(
	 \widetilde{\bGamma}_{\mu}^{(n)}  \otimes \bA^{(n)} \right)_{n = 1}^{N}  \notag. \\
	 {\bPsi}_{\mu}  &=& {\blkdiag}\left(\widetilde{\bGamma}_{\mu}^{(n)}\otimes
            \bC^{(n)}  \right)_{n=1}^{N} .
\ee
%Similarly, denote ${\bPsi}  = \bZ^T \,  {\bG}_{\mu}^{-1} \, \bZ$ has a compact expression as in (\ref{equ_ZiGZ_LM}). 
Finally, we define $\bB$ as in (\ref{equ_inverse_f2a}), and easily deduce (\ref{equ_inverse_H_LM}) from (\ref{equ_Hinverse1}).
\end{proof}

\section{Proof of Lemma~\ref{theo_elim_Jacobian}}\label{sec::proof_theo_elim_Jacobian}
\begin{proof}
%\subsubsection{Elimination of Jacobian}\label{sec::avoidJacobian}
From (\ref{equ_Jacobian}), (\ref{equ_invG}), 
and note that $\vtr{\tE} =  \bQ_n \, \vtr{\tE_{(n)}}$, where \add{$\bQ_n$} is defined in Lemma~\ref{lemma_Pn}, 
the product ${\widetilde{\bG}_{\mu}} \, \bg$ can be expressed in a block form as

\minrowclearance 5pt
\begin{equationarray}{lcl}
%\be
    \left({\widetilde\bG_{\mu}}  \, \bJ^{T} \vtr{\tE}\right)^T  &=&
        {\left[\vtr{\tE}^T \bQ_n   \left( \left(\left({\displaystyle\bigodot_{k\neq n} \bA^{(k)}}\right)\,  \widetilde{\bGamma}_{\mu}^{(n)} \,   \right)
        \otimes  \bI_{I_n} \right)   \right]_{n=1}^{N}}
%       &=&
%       {\left[\vtr{\bE_{(n)}}^T \,  \left( \left( \bA^{\odot_{-n}}\, {{\boldsymbol\Gamma}^{(n)}}^{-1} \,   \right)
%       \otimes  \bI_{I_n} \right)   \right]_{n=1}^{N}}^T \notag\\
        =
        {\left[
        \vtr{ \bE_{(n)} \left({\displaystyle\bigodot_{k\neq n} \bA^{(k)}}\right)  \widetilde{\bGamma}_{\mu}^{(n)} }^T
        \right]_{n=1}^{N} }\, \notag\\
        &=& {\left[
        \vtr{ \bY_{(n)}  \, \left({\displaystyle\bigodot_{k\neq n} \bA^{(k)}}\right)\,   \widetilde{\bGamma}_{\mu}^{(n)} -  \bA^{(n)}\, \left({\displaystyle\bigodot_{k\neq n} \bA^{(k)}}\right)^T \, \left({\displaystyle\bigodot_{k\neq n} \bA^{(k)}}\right)\,   \widetilde{\bGamma}_{\mu}^{(n)} }^T
        \right]_{n=1}^{N} }\,  \notag\\
%        &=& {\left[
%        \vtr{ \bY_{(n)}  \, \left({\displaystyle\bigodot_{k\neq n} \bA^{(k)}}\right)\,   \widetilde{\bGamma}_{\mu}^{(n)}-  \bA^{(n)}\,  {\bGamma^{\mbox{\add{\scriptsize$(n)$}}}} \,   \widetilde{\bGamma}_{\mu}^{(n)} }^T
%        \right]_{n=1}^{N} }\,   \notag\\
        &=& {\left[
        \vtr{ \bA^{(n)}_{\mu}-    \bA^{(n)}\,  {\bGamma^{(n)}} \,   \widetilde{\bGamma}_{\mu}^{(n)}}^T
        \right]_{n=1}^{N} }\, . \label{equ_proj_GJE}
%\ee
\end{equationarray}
\minrowclearance 0pt
Similarly, a convenient formula to compute $\bL_{\mu}^T \, \bg$ is given by\\[-2ex]
\be
    {\bw_{\mu}} &=& \bL_{\mu}^T \bJ^{T} \, \vtr{\tE} =
    {\left[ \vtr{ \bA^{(n) \, T} \left(  \bA^{(n)}_{\mu} -   \bA^{(n)}\,  {\bGamma^{(n)}} \,   \widetilde{\bGamma}_{\mu}^{(n)} \right)}^T \right]_{n = 1}^{N}}^T  \notag \\
%        &=& \vtr{\left[ \bA^{(n) \, T}   \bA^{(n)}_{\mu} -  \bA^{(n) \, T} \bA^{(n)} \,{\bGamma^{\mbox{\add{\scriptsize$(n)$}}}} \, \widetilde{\bGamma}_{\mu}^{(n)}  \right]_{n=1}^{N}} \notag\\
        &=& \vtr{\left[ \bA^{(n) \, T}   \bA^{(n)}_{\mu} -  \,{\bGamma} \, \widetilde{\bGamma}_{\mu}^{(n)}  \right]_{n=1}^{N}}
        \, . \label{equ_proj_ZGJE_LM}
\ee
Finally, for each frontal slice $\bF_{n}$ of the tensor $\tF \in \Real^{R \times R \times N}$ whose $\vtr{\tF}  = \bB_{\mu} \bw_{\mu}$, we have
\be
\left(\widetilde{\bGamma}_{\mu}^{(n)}  \otimes \bA^{(n)} \right) \, \vtr{\bF_{n}} =
\vtr{ \bA^{(n)}  \,
     \bF_{n} \,\widetilde{\bGamma}_{\mu}^{(n)} }. \label{equ_Fn} \ee 
     From (\ref{equ_iGZ_LM}), \add{we obtain} (\ref{equ_fast_L_psi_a}). 
%\be
%     {\bL_{\mu}} \, \vf    &=& {\blkdiag}\left(
%	 \widetilde{\bGamma}_{\mu}^{(n)}  \otimes \bA^{(n)} \right)_{n = 1}^{N}  \, 
%\\
%&=&
%%       {\blkdiag}{\left(\widetilde{\bGamma}_{\mu}^{(n)}  \otimes \bA^{(n)} \right)_{n = 1}^{N}}   \,\left( \left[\vtr{\bF_{n}}^T\right]_{n=1}^{N}  \right)^T\,
%%        \notag\\
%%    &=&
% \left[
% \begin{array}{c}
%     \vtr{ \bA^{(1)}  \,
%     \bF_{1} \,
%     \widetilde{\bGamma}_{\mu}^{(1,1)}} \\[-1ex]
%     \vdots\\[-1ex]
%     \vtr{ \bA^{(n)}  \,
%     \bF_{n} \,
%     \widetilde{\bGamma}_{\mu}^{(n)}}
%     \\[-1ex]
%     \vdots \\[-1ex]
%     \vtr{ \bA^{(N)}  \,
%     \bF_{N} \,
%     \widetilde{\bGamma}_{\mu}^{(N,N)}}
%\end{array}     \right] . \label{equ_fast_L_psi}
%\ee
Each product inside (\ref{equ_Fn}) has a complexity of $\textsl{O}\left(I_n \, R^2 + R^3\right)$. Hence, $\bL_{\mu} \, \vf$ in (\ref{equ_fast_L_psi_a}) has a complexity of $\textsl{O}\left(TR^2 + NR^3\right) \approx \textsl{O}\left(TR^2\right)$ which is lower than $\textsl{O}\left(TR^3\right)$ by a factor $R$ for building up $\bL_{\mu}$ and direct computation ${\bL_{\mu}} \, \vf$. Furthermore, this fast computation does not use any significant temporary extra-storage.
\end{proof}

\section{Inverse of The Kernel Matrix $\bK$} \label{sec::app_Kernel}

\begin{theorem}\label{theo_diag_kernel_K}
Inverse of $\bK$ defined in (\ref{equ_Knm_a}) is a partitioned matrix
${\widetilde{\bK}} = {\bK}^{-1}$ whose blocks ${\widetilde{\bK}}^{(n,m)}$, for $n = 1,\ldots, N, m = 1, \ldots, N$ are given by
\be
    {\widetilde{\bK}}^{(n,m)} =
%    \begin{cases}
%        -\displaystyle \frac{N-2}{N-1} \, \bP_{R}  \, \diag\left(\vtr{\bC^{(n)}} \oslash \vtr{\bGamma^{\mbox{\add{\scriptsize$(n)$}}}}\right),        \quad  &n = m \, ,  \\
%         \displaystyle\frac{1}{N-1} \,\bP_{R}  \, \diag\left( \1\oslash \vtr{\bGamma^{(n,m)}}\right) \, , \quad  & n \neq m \,.
%    \end{cases} 
\left(\frac{1}{N-1} - \delta_{n,m}\right) \, \diag\left(\vtr{\bC^{(n)} \* \bC^{(m)} \oslash {\bGamma}}\right)\, \bP_{R}.
\label{equ_inv_K}
\ee
\end{theorem}

\section{Effects of noise on collinear data}\label{sec::Collinear_and_noise}

This section discusses briefly effects of noise on factorization of collinear tensor generated by the modification (\ref{equ_modify_1}).
%$\tE$
%	  
%\be
%	\tY = \tI \times_1 \bA^{(1)} \bQ  \times_2 \bA^{(2)} \bQ \cdots \times_N \bA^{(N)} \bQ
%	= 
%\ee
Consider matrix factorization of the mode-$n$ tensor unfolding\\[-4ex]
\be
	\bY_{(n)} = \bA^{(n)} \left(\bigodot_{k\neq n} \bA^{(k)}\right)^T + \bE_{(n)}. 
\ee\\[-3ex]
%leading left singular components of $\bY_{(n)}$ are good initialization for $\bA^{(n)}$\cite{Lathauwer_HOOI,Kolda08}.  
Analysis of singular values of $\bY_{(n)}$ or eigenvalues of $\bY_{(n)} \, \bY_{(n)}^T$ allow predicting whether factorization succeeds in retrieving collinear factors from noisy tensors.
This also gives insight into when CP algorithms are not stable, and yield non-unique solution.

The modification (\ref{equ_modify_1}) can be expressed as
$\bA^{(n)} = \bU^{(n)} \, \bQ$, where 
$\bQ = \left[\begin{array}{cc}
	 1 & \1_{R-1}^T\\
	 \0_{R-1} & \nu\, \bI_{R-1}
	 \end{array}
	  \right] \in \Real^{R\times R}$. 
In theory, for noisy tensors $\tY$ with $I_n = I, \forall n$, we have 
\be	
	\bY_{(n)} \, \bY_{(n)}^{T}  = \bA^{(n)} \, {\bGamma^{\mbox{\add{\scriptsize$(n)$}}}}  \, \bA^{(n) \, T}  +  \bE_{(n)} \, \bE_{(n)}^T 
	= \bU^{(n)} \, {\bSigma}  \, \bU^{(n)\,T} + \sigma^2 \, I^{N-1} \, \bI_{I_n}
	\,.
\ee
where $\displaystyle \bSigma = \bQ \, \left(\bQ^T \, \bQ \right)^{\bullet[N-1]} \, \bQ^T$, $[\bA]^{\bullet[p]}$ denotes element-wise power, and\\[-3ex]
%$\sigma^2$ is calculated as  a function of SNR, $R, I$, $N$ and $\nu$
\be
	\sigma^2 = \frac{\|\tY\|_F^2}{10^{\text{SNR}/10} \, I^{N}} = \frac{R^2 + (R-1) \,  xy-1}{10^{\text{SNR}/10} \, I^{N}}, \qquad x = 1 + \nu^2, y = x^{N-1}.
\ee
It is straightforward to prove that 
$
	\bSigma = \left[ \begin{array}{@{}c@{\hspace{1ex}}c@{}}
				R^2 + (R-1) \,( y -1)&  \nu \, (R + y-1) \, \1_{R-1}^T \\
				  \nu \, (R+ y -1) \, \1_{R-1} & (x-1) \left(\1_{R-1}\, \1_{R-1}^T  + (y-1) \bI_{R-1} \right)
 			\end{array} \right]
$
%
%$
%	\bSigma = \left[ \begin{array}{@{}c@{\hspace{1ex}}c@{}}
%				 \alpha &  \beta \, \1_{R-1}^T \\
%				  \beta \, \1_{R-1} & (x-1) \1_{R-1}\, \1_{R-1}^T  + \gamma \bI_{R-1}  
% 			\end{array} \right]
%$
%where $\alpha = R^2 + (R-1) \,( y -1)$, $\beta = \nu \, (R + y-1)$, $\gamma = (x-1)(y-1)$,
%alpha = R^2 -R +1 + (R-1) * (1+lambda^2)^(N-1);
%beta =  lambda * ( R-1 + (1+lambda^2)^(N-1)) ;
%gamma = ((1+lambda^2)^(N-1) -1);
has $(R-2)$ identical eigenvalues $\lambda_r = (x-1) (y-1), r = 2, \ldots, R-1$, and 
its largest and smallest eigenvalues $\lambda_1 > \lambda_r > \lambda_R$ are solutions of a quadratic equation\\[-3ex]
\be
	\lambda_1 +  \lambda_R &=& xy + (R-2) \, (R+x+y) +3  ,  \\
	\lambda_1 \,  \lambda_R &=&  (x-1) (y-1) = \lambda_r \,, \qquad 2\le r \le R-1\, .
\ee

Fig.~\ref{fig_eig_100x15x3} illustrates $\lambda_r$ $(r= 1, \ldots, R)$ for 3-D noiseless tensors with  $I = 100$ and $R = 15$ compared with the noise levels $\sigma^2 \, I^{N-1}$ at SNR = 20 dB and 30 dB.
The higher the collinearity degree of factor, the smaller the eigenvalues $\lambda_r$.
If eigenvalues $\lambda_r$ are considerably lower than the noise level $\sigma^2 \, I^{N-1}$, the factorization becomes infeasible, e.g., as $\nu \le 0.1$.

 Because $\bU^{(n)}$ are orthonormal,
$\bY_{(n)} \, \bY_{(n)}^{T}$ has $R$ leading eigenvalues ${\tilde\lambda}_r = \lambda_r + \sigma^2 I^{(N-1)}, r = 1, \ldots, R$, and $(I-R)$ eigenvalues ${\tilde\lambda}_i = \sigma^2 I^{(N-1)}, i = R+1, \ldots, I$.
In Fig.~\ref{fig_eig_noise_100x15x3}, we plot eigenvalues ${\tilde\lambda}_i$ for noisy tensors having the same dimension as that of tensors illustrated in Fig.~\ref{fig_eig_100x15x3}.
The largest eigenvalue ${\tilde\lambda}_1$ significantly exceeds the noise levels\add{, whereas} ${\tilde\lambda}_R$ is quite close to the noise level at SNR = 20 dB for $\nu \le 0.3$,
or at SNR = 30 dB for $\nu \le 0.1$.

\begin{figure}
\psfrag{nu}[t][t]{\scalebox{.8}{\color[rgb]{0,0,0}\setlength{\tabcolsep}{0pt}\begin{tabular}{c}$\nu$\end{tabular}}}%
\subfigure[Eigenvalues $\lambda_r, r = 1, \ldots, R$ (= 15), $I_n = 100$, $N$ = 3.]{
\psfrag{eig-1}[tl][tl]{\scalebox{.6}{$\lambda_1$}}%
\psfrag{eig-2,...,-(R-1)}[tl][tl]{\scalebox{.6}{$\lambda_2, \ldots, \lambda_{R-1}$}}%
\psfrag{eig-R}[tl][tl]{\scalebox{.6}{$\lambda_{R}$}}%
\psfrag{eig-(R+1),...,I}[tl][tl]{\scalebox{.6}{$\lambda_{R+1},\ldots,\lambda_{I}$}}%
\psfrag{SNR = 20dB}[tl][tl]{\scalebox{.5}{SNR = 20 dB}}%
\psfrag{SNR = 30dB}[tl][tl]{\scalebox{.5}{SNR = 30 dB}}%
\psfrag{SNR = 40dB}[tl][tl]{\scalebox{.5}{SNR = 40 dB}}%
\psfrag{R = 5, I = 50, N = 3}[tl][tl]{\scalebox{.5}{}}%
\includegraphics[width=.48\linewidth, trim = 0.0cm 0cm 0cm 0cm,clip=false]{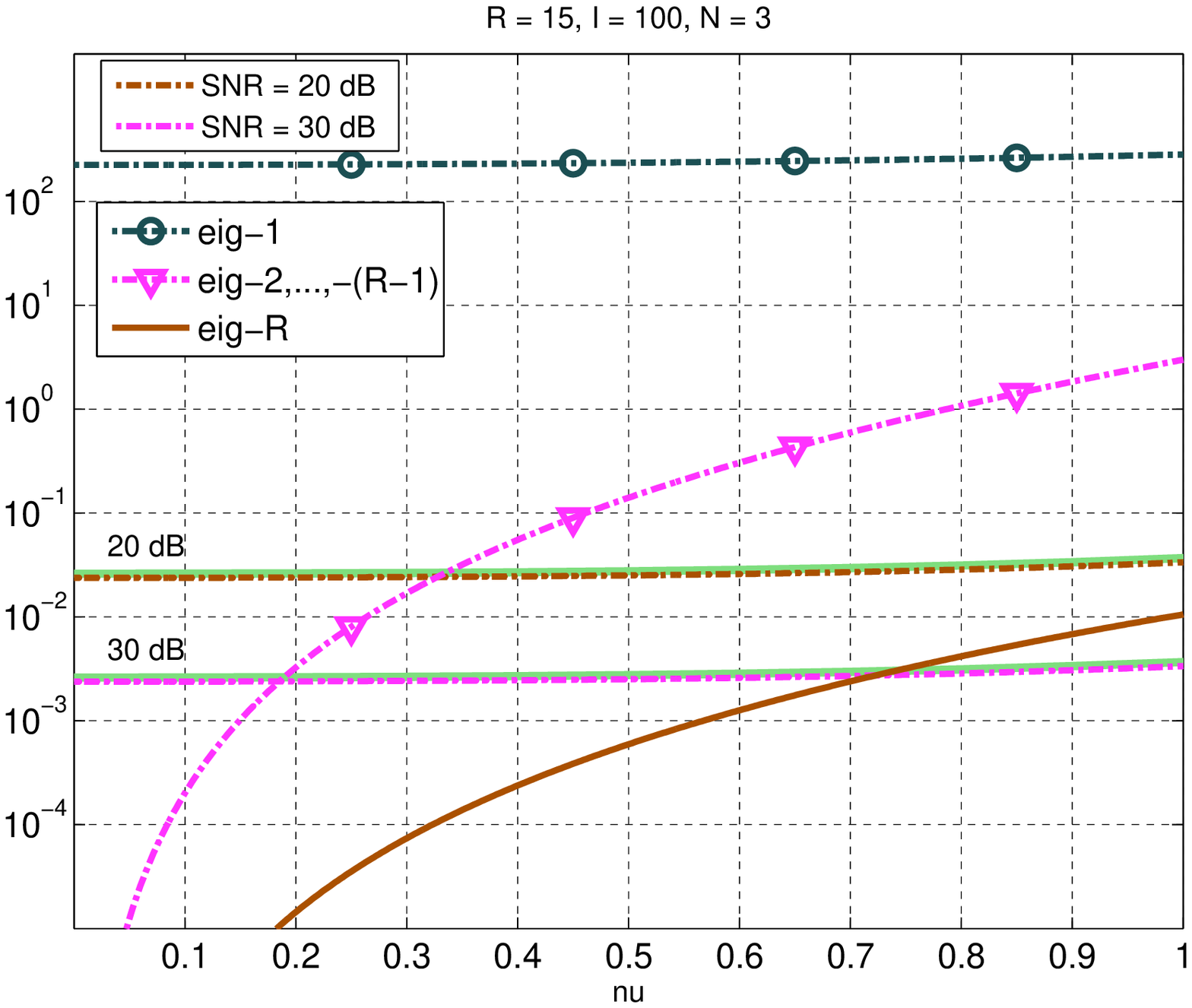}\label{fig_eig_100x15x3}
}
\hfill
\subfigure[Eigenvalues $\tilde\lambda_i, i = 1, \ldots, I_n$  (= 100), $R = 15$, $N$ = 3.]{
\psfrag{eig-1}[bl][bl]{\scalebox{.55}{$\tilde\lambda_1$}}%
\psfrag{eig-2,...,-(R-1)}[bl][bl]{\scalebox{.55}{$\tilde\lambda_2, \ldots, \tilde\lambda_{R-1}$}}%
\psfrag{eig-R}[tl][tl]{\scalebox{.55}{$\tilde\lambda_{R}$}}%
\psfrag{eig-(R+1),...,I}[bl][bl]{\scalebox{.55}{$\tilde\lambda_{R+1},\ldots,\tilde\lambda_{I}$}}%
\psfrag{SNR = 20dB}[bl][bl]{\scalebox{.5}{SNR = 20 dB}}%
\psfrag{SNR = 30dB}[bl][bl]{\scalebox{.5}{SNR = 30 dB}}%
\psfrag{SNR = 40dB}[bl][bl]{\scalebox{.5}{SNR = 40 dB}}%
\psfrag{R = 5, I = 50, N = 3}[tl][tl]{\scalebox{.5}{}}%
\includegraphics[width=.48\linewidth, trim = 0.0cm 0.0cm 0cm 0cm,clip=false]{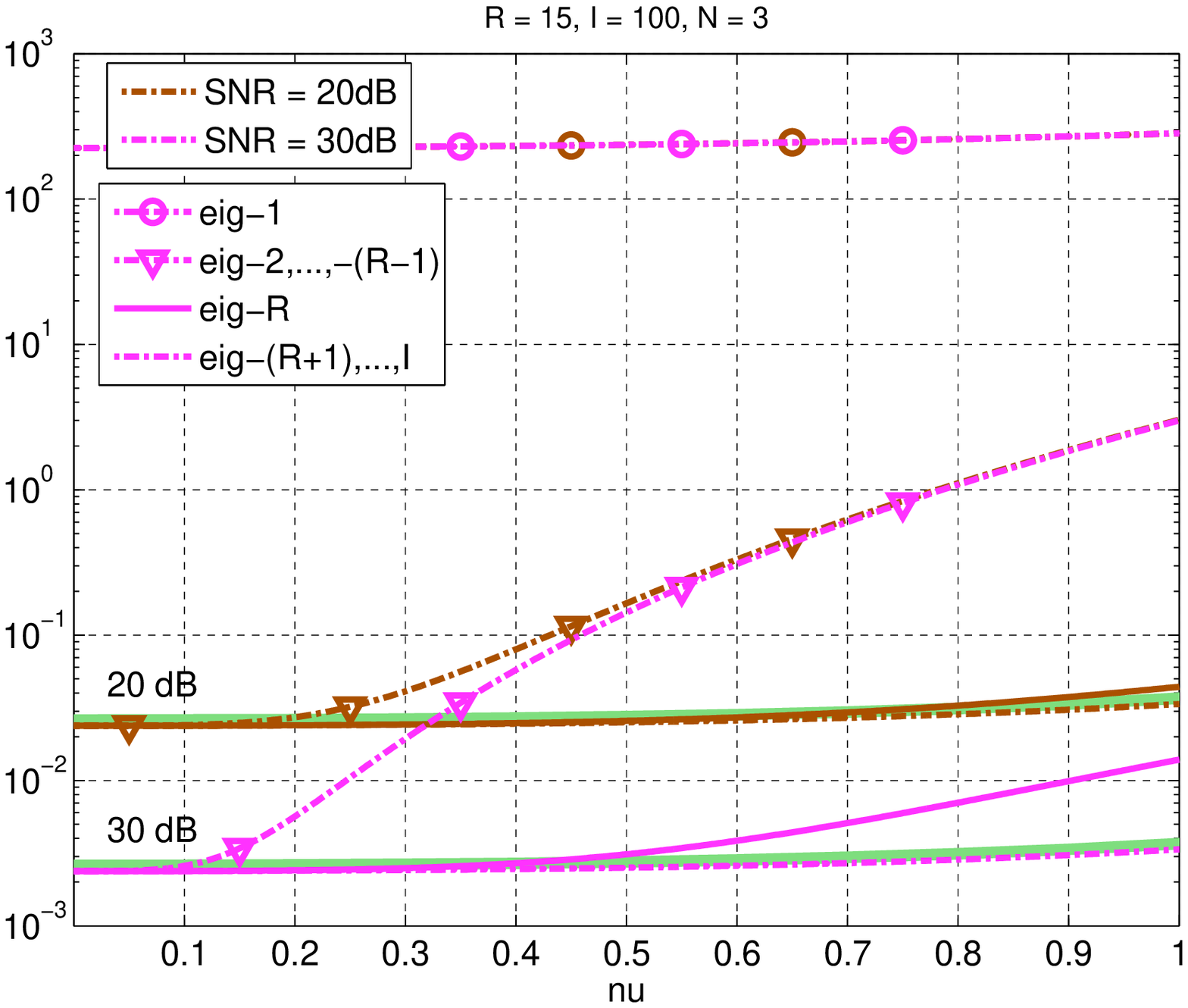}\label{fig_eig_noise_100x15x3}
}
\caption{Analysis of eigenvalues of $\bY_{(n)} \, \bY_{(n)}^T$ for 3-D tensors of size $I_n = 100$ and rank $R = 15$.
$R$ leading eigenvalues $\lambda_r$ for noiseless tensors and $\tilde \lambda_r (r = 1, \ldots, R)$ for noisy tensors are compared with noise levels (green shading) at SNR = 20 dB and 30 dB. 
The more the eigenvalues are in the noise zone, the more difficult the factorization of noisy tensors to retrieve collinear factors become.}
\label{fig_eig_noise}
\end{figure}

\bibliographystyle{siam}
\bibliography{bibligraphy_thesis}

\begin{thebibliography}{10}

\bibitem{JChem-CPOPT}
{\sc E.~Acar, D.~M. Dunlavy, and T.~G. Kolda}, {\em A scalable optimization
  approach for fitting canonical tensor decompositions}, Journal of
  Chemometrics, 25 (2011), pp.~67--86.

\bibitem{Nwaytoolbox}
{\sc C.A. Andersson and R.~Bro}, {\em The {N}-way toolbox for {MATLAB}},
  Chemometrics Intell. Lab. Systems, 52 (2000), pp.~1--4.

\bibitem{Bro1997}
{\sc R.~Bro}, {\em {PARAFAC. T}utorial and applications}, in Special Issue 2nd
  Internet Conf. in Chemometrics (INCINC'96), vol.~38, Chemom. Intell. Lab.
  Syst, 1997, pp.~149--171.

\bibitem{Carroll_Chang}
{\sc J.D. Carroll and J.J. Chang}, {\em Analysis of individual differences in
  multidimensional scaling via an n-way generalization of {E}ckart--{Y}oung
  decomposition}, Psychometrika, 35 (1970), pp.~283--319.

\bibitem{NMF-book}
{\sc A.~Cichocki, R.~Zdunek, A.-H. Phan, and S.~Amari}, {\em Nonnegative Matrix
  and Tensor Factorizations: Applications to Exploratory Multi-way Data
  Analysis and Blind Source Separation}, Wiley, Chichester, 2009.

\bibitem{Comon09}
{\sc P.~Comon, X.~Luciani, and A.~L.~F. de~Almeida}, {\em Tensor
  decompositions, alternating least squares and other tales}, Jour.
  Chemometrics, 23 (2009), pp.~393--405.

\bibitem{Lathauwer_HOOI}
{\sc L.~{De Lathauwer}, B.~De Moor, and J.~Vandewalle}, {\em {On the best
  rank-1 and rank-(R1,R2,\ldots,RN) approximation of higher-order tensors}},
  SIAM J. Matrix Anal. Appl., 21 (2000), pp.~1324--1342.

\bibitem{Patrick96}
{\sc P.~Guillaume and R.~Pintelon}, {\em A {Gauss-Newton}-like optimization
  algorithm for weighted nonlinear nonlinear least squares problems}, IEEE
  Trans. Signal Processing, 44 (1996), pp.~2222--2228.

\bibitem{Guo2011}
{\sc X.~Guo, S.~Miron, D.~Brie, and A.~Stegeman}, {\em Uni-mode and partial
  uniqueness conditions for {CANDECOMP/PARAFAC} of three-way arrays with
  linearly dependent loadings}, SIAM J. Matrix Anal. Appl.,  (2011).

\bibitem{Harshman}
{\sc R.A. Harshman}, {\em Foundations of the {PARAFAC} procedure: Models and
  conditions for an explanatory multimodal factor analysis}, UCLA Working
  Papers in Phonetics, 16 (1970), pp.~1--84.

\bibitem{Hayashi}
{\sc C.~Hayashi and F.~Hayashi}, {\em A new algorithm to solve
  {PARAFAC}-model}, Behaviormetrika, 11 (1982), pp.~49--60.

\bibitem{Hitchcock1927}
{\sc F.L. Hitchcock}, {\em Multiple invariants and generalized rank of a p-way
  matrix or tensor}, Journal of Mathematics and Physics, 7 (1927), pp.~39--79.

\bibitem{0521386322}
{\sc R.~A. Horn and C.~R. Johnson}, {\em Matrix Analysis}, Cambridge University
  Press, 1990.

\bibitem{Kolda08}
{\sc T.G. Kolda and B.W. Bader}, {\em Tensor decompositions and applications},
  SIAM Review, 51 (2009), pp.~455--500.

\bibitem{ZbynekSSP11}
{\sc Z.~Koldovsk{\'y}, P.~Tichavsk{\'y}, and A.-H. Phan}, {\em Stability
  analysis and fast damped {Gauss-Newton} algorithm for {INDSCAL} tensor
  decomposition}, in Statistical Signal Processing Workshop (SSP), IEEE, 2011,
  pp.~581--584.

\bibitem{Li_Sidiropoulos}
{\sc X.Q. Liu and N.D. Sidiropoulos}, {\em {Cramer-Rao} lower bounds for
  low-rank decomposition of multidimensional arrays}, IEEE Transactions on
  Signal Processing, 49 (2001), pp.~2074--2086.

\bibitem{MITCHELL94}
{\sc B.~C. Mitchell and D.~S. Burdick}, {\em Slowly converging {PARAFAC}
  sequences: {S}wamps and two-factor degeneracies}, Jour. Chemometrics, 8
  (1994), pp.~155--168.

\bibitem{NielsendampingLM}
{\sc H.~B. Nielsen}, {\em Damping parameter in {Marquardt's} method}, tech.
  report, Department of Mathematical Modelling, DTU, 1999.

\bibitem{Paatero}
{\sc P.~Paatero}, {\em Least-squares formulation of robust nonnegative factor
  analysis}, Chemometrics and Intelligent Laboratory Systems, 37 (1997),
  pp.~23--35.

\bibitem{Paatero97}
\leavevmode\vrule height 2pt depth -1.6pt width 23pt, {\em A weighted
  non-negative least squares algorithm for three-way {PARAFAC} factor
  analysis}, Chemometrics Intelligent Laboratory Systems, 38 (1997),
  pp.~223--242.

\bibitem{Paatero99}
\leavevmode\vrule height 2pt depth -1.6pt width 23pt, {\em The multilinear
  engine: A table-driven, least squares program for solving multilinear
  problems, including the n-way parallel factor analysis model}, Journal of
  Computational and Graphical Statistics, 8 (1999), pp.~854--888.

\bibitem{Phan_LM_NTD}
{\sc A.-H. Phan, P.~Tichavsk{\'y}, and A.~Cichocki}, {\em Damped {Gauss-Newton}
  algorithm for nonnegative {Tucker} decomposition}, in Statistical Signal
  Processing Workshop (SSP), IEEE, 2011, pp.~665 --668.

\bibitem{Sorenson80}
{\sc H.~W. Sorenson}, {\em Parameter estimation: principles and problems},
  Marcel Dekker, NY, USA, 1980.

\bibitem{Petr_10}
{\sc P.~Tichavsk{\'y} and Z.~Koldovsk{\'y}}, {\em Simultaneous search for all
  modes in multilinear models}, in Proc. IEEE International Conference on
  Acoustics, Speech, and Signal Processing ({ICASSP10}), 2010, pp.~4114 --
  4117.

\bibitem{DBLP:conf/icassp/TichavskyK11}
\leavevmode\vrule height 2pt depth -1.6pt width 23pt, {\em Stability of
  {CANDECOMP-PARAFAC} tensor decomposition}, in Proc. IEEE International
  Conference on Acoustics, Speech, and Signal Processing (ICASSP11), 2011,
  pp.~4164--4167.

\bibitem{DBLP:journals/tsp/TichavskyK11}
\leavevmode\vrule height 2pt depth -1.6pt width 23pt, {\em Weight adjusted
  tensor method for blind separation of underdetermined mixtures of
  nonstationary sources}, IEEE Transactions on Signal Processing, 59 (2011),
  pp.~1037--1047.

\bibitem{PetrCRIB}
{\sc P.~Tichavsk{\'y}, A.-H. Phan, and Z.~Koldovsk{\'y}}, {\em {Cram\'er-Rao}
  induced bounds for {CANDECOMP/PARAFAC} tensor decomposition},  (2012),
  p.~submitted.

\bibitem{INDAFAC}
{\sc G.~Tomasi}, {\em {INDAFAC} and {PARAFAC3W}}.
\newblock \url{http://www.models.kvl/dk/source/indafac/index.asp}, 2003.

\bibitem{Tomassithesis}
\leavevmode\vrule height 2pt depth -1.6pt width 23pt, {\em Practical and
  Computational Aspects in Chemometric Data Analysis}, PhD thesis,
  Frederiksberg, Denmark, 2006.

\bibitem{Tomasi_tricap06}
\leavevmode\vrule height 2pt depth -1.6pt width 23pt, {\em Recent developments
  in fast algorithms for fitting the {PARAFAC} model}, in TRICAP 2006, Greece,
  2006, TRICAP.

\bibitem{TomasiBro05}
{\sc G.~Tomasi and R.~Bro}, {\em {PARAFAC} and missing values}, Chemometrics
  Intelligent Laboratory Systems, 75 (2005), pp.~163--180.

\bibitem{TomasiBro06}
\leavevmode\vrule height 2pt depth -1.6pt width 23pt, {\em A comparison of
  algorithms for fitting the {PARAFAC} model}, Computational Statistics and
  Data Analysis, 50 (2006), pp.~1700--1734.

\bibitem{doi:10.1137/110843587}
{\sc A.~Uschmajew}, {\em Local convergence of the alternating least squares
  algorithm for canonical tensor approximation}, SIAM Journal on Matrix
  Analysis and Applications, 33 (2012), pp.~639--652.

\end{thebibliography}

\end{document}